\documentclass[preprint,3p]{elsarticlenofooter}
\usepackage[italian,english]{babel}
\usepackage[usenames,dvipsnames]{xcolor}
\usepackage{amsmath}
\usepackage{esint} 
\usepackage{mathrsfs} 
\usepackage{amssymb}
\usepackage{bbm} 
\usepackage{graphicx} 
\usepackage{subfig} 
\usepackage{booktabs} 
\usepackage{tabu} 
\usepackage{numprint} 
\usepackage[hidelinks,]{hyperref} 
\usepackage{algorithm}
\usepackage{algpseudocode}
\usepackage{amsthm} 
\usepackage{indentfirst} 
\usepackage[babel]{csquotes}
\usepackage{array}
\usepackage{tabularx}
\usepackage{amstext}
\usepackage{epstopdf} 
\usepackage{epigraph}
\usepackage{enumitem} 
\usepackage{pdfpages} 
\usepackage[title]{appendix}
\usepackage{empheq}

\graphicspath{{img/}}

\allowdisplaybreaks

\newcommand{\tensel}{\mathbb{E}} 
\newcommand{\matm}{\beta} 
\newcommand{\compliance}{\mathcal{J}} 
\newcommand{\mass}{M} 
\newcommand{\masslim}{\overline{M}} 
\newcommand{\control}{\eta}
\newcommand{\controlset}{\mathcal{Y}}

\newcommand{\symgrad}[1]{\nabla^\text{s} {#1}} 
\newcommand{\symgradsubscript}[2]{\nabla^\text{s}_{#1} {#2}} 
\newcommand{\Itwo}{\vec{I}_2} 
\newcommand{\Ifour}{\vec{I}_4} 
\newcommand{\idxmat}{m} 
\newcommand{\mesh}{\mathcal{T}_h} 
\newcommand{\fespace}{V_h} 
\newcommand{\CHm}{\sigma} 
\newcommand{\Qrot}{\vec{Q}} 
\newcommand{\Rrot}{\vec{R}} 
\newcommand{\freeIDXzl}[1]{E^z_{#1}}
\newcommand{\freeIDXz}{F}
\newcommand{\freeIDXzUp}{G}
\newcommand{\freeIDXml}[1]{E^\matm_{#1}}
\newcommand{\rhomprime}{\rho_\idxmat^\prime}

\newcommand{\ResultsPath}{results}

\newcommand{\ResultsFormatBASE}{png}
\newcommand{\ResultsFormatBASErendering}{png}

\newcommand{\ResultsFormatCHO}{png}

\newcommand{\ResultsBASESquareDueMat}{Square_150x150_r.1_-49593.3}
\newcommand{\ResultsBASESquareDueMatRot}{Square_150x150_r.1_-49585.9}
\newcommand{\ResultsBASESquareQuattroMat}{Square_150x150_r.1_-49386.7}
\newcommand{\ResultsBASESquareOttoMat}{Square_150x150_r.1_-49582.7}
\newcommand{\ResultsBASELshape}{LshapeCorner_200x200_r.15_1.01837e+07}
\newcommand{\ResultsBASEOverbridge}{overbridge_40x200_r.15_-2.97812e+06}

\newcommand{\ResultsCHOSquare}{Square_200x200_r.2_r.08_1000_1_4angle_ztot.32_cont_r.1_rest_9.84637e+06}

\newcommand{\ResultsCHOLshapeUnOrdine}{LshapeCorner_200x200_ztot.35_r.1_r.08_5mat_rest_9.76245e+06}
\newcommand{\ResultsCHOLshapeTreOrdini}{LshapeCorner_200x200_ztot.35_r.1_r.08_5mat1000_1_rest_9.82596e+06}

\newcommand{\ResultsCHOOverbridge}{overbridge_CHO_60x300_r.25_r.08_continuato_r.15_merged_9.82737e+06}

\newcommand{\virgolette}[1]{``{#1}''} 

\renewcommand{\vec}[1]{\boldsymbol{\mathbf{#1}}} 
\newcommand{\diff}[1]{\,\mathrm{d}{#1}} 
\newcommand{\dx}{\diff{\vec{x}}} 
\newcommand{\laplacian}{\Delta} 
\newcommand{\diverg}[1]{\operatorname{div} {#1}} 
\newcommand{\indicator}[2]{\mathbbm{1}_{#1}(#2)} 
\renewcommand{\theta}{\vartheta}
\renewcommand{\epsilon}{\varepsilon}
\renewcommand{\phi}{\varphi}

\theoremstyle{plain}\newtheorem{theorem}{Theorem}
\theoremstyle{plain}
\theoremstyle{plain}\newtheorem{proposition}{Proposition}
\theoremstyle{plain}\newtheorem{lemma}{Lemma}
\theoremstyle{plain}
\theoremstyle{definition}
\theoremstyle{remark}\newtheorem{remark}{Remark}

\begin{document}

\begin{frontmatter}
	
	\title{Topology optimization of multiple anisotropic materials, with application to self-assembling diblock copolymers}
	
	\journal{Computer Methods in Applied Mechanics and Engineering}
	
	\author[mat]{F.~Regazzoni\corref{cor1}}
	\ead{francesco.regazzoni@polimi.it}

	\author[mat]{N.~Parolini}
	\ead{nicola.parolini@polimi.it}
	
	\author[mat]{M.~Verani}
	\ead{marco.verani@polimi.it}
	
	\address[mat]{MOX, Dipartimento di Matematica, Politecnico di Milano, Milano, Italy}
	
	\cortext[cor1]{Corresponding author. Dipartimento di Matematica, Politecnico di Milano, piazza Leonardo da Vinci 32, 20133, Milano, Italy}

	\begin{abstract} 
		We propose a solution strategy for a multimaterial minimum compliance topology optimization problem, which consists in finding the optimal allocation of a finite number of candidate (possibly anisotropic) materials inside a reference domain, with the aim of maximizing the stiffness of the body.
		As a relevant and novel application we consider the optimization of self-assembled structures obtained by means of diblock copolymers. Such polymers are a class of self-assembling materials which spontaneously synthesize periodic microstructures at the nanoscale, whose anisotropic features can be exploited to build structures with optimal elastic response, resembling biological tissues exhibiting microstructures, such as bones and wood.
		For this purpose we present a new generalization of the classical Optimality Criteria algorithm to encompass a wider class of problems, where multiple candidate materials are considered, the orientation of the anisotropic materials is optimized, and the elastic properties of the materials are assumed to depend on a scalar parameter, which is optimized simultaneously to the material allocation and orientation.	
		Well-posedness of the optimization problem and well-definition of the presented algorithm are narrowly treated and proved. The capabilities of the proposed method are assessed through several numerical tests.
	\end{abstract}

	
\end{frontmatter}

\section{Introduction} \label{sec:intro}

The geometry and the topology of structures have a great impact on their performances. Therefore, the efficient use of material is crucial in many fields of application, from automotive industry, to bioengineering or MEMS industry. This explains the great interest towards topology optimization recorded in the past decades, both in the academic and in the industrial world. The classical topology optimization (TopOpt) problem looks for the optimal distribution of a given amount of isotropic material inside a prescribed domain, in order to optimize the mechanical response of the body to a given load. The performance of the design is measured by means of the so-called compliance, to be minimized, defined as twice the elastic energy computed at equilibrium~\cite{bendsoe:book,rev1,rev2}. 

To overcome the computational complexity of large $0-1$ type integer programming problems, the shape of the body is typically tracked by a density variable taking values in $[0,1]$, and a suitable interpolation scheme penalizing densities different form $0$ and $1$ is employed, one of the most popular being the SIMP (Solid Isotropic Material with Penalization) formulation~\cite{bendsoe:article89,bendsoe:book}. An efficient algorithm to solve the minimum compliance problem is the OC (Optimality Criteria) method, a fixed point algorithm based on the optimality conditions~\cite{oc}. In alternative, methods of sequential convex programming can be employed, like CONLIN (CONvex LINearization, see~\cite{conlin}) and MMA (Method of Moving Asymptotes, see~\cite{svanberg1987method,svanberg1993method}).

Alternative approaches for TopOpt are the level-set method and the phase-field method. With the first one, the borders of the body are determined as level-sets of a scalar function defined over the domain~\cite{wang2003level}. 
With the second approach, the TopOpt process is interpreted as a phase transition process, where a functional made of two contributions is minimized: the first term is an Allen-Cahn/Cahn-Hilliard type energy, which penalizes intermediate densities by means of a double well potential as well as the mean curvature of the border of the body, and the second term is proportional to the compliance~\cite{wang2004phase,dede2012isogeometric,blank2012phase}.

\subsection{Multimaterial TopOpt} \label{sec:intro:multimat}

It is well known that bodies exhibiting microstructures can be very efficient from the structural perspective. As a matter of fact nature exhibits plenty of examples of elastic bodies that, having to resist to mechanical loadings, present a fine-scale structure: bones, for instance, exhibit a sponge structure, and wood reveals a quasi-periodic microstructure.
In fact, given a structure, it is always possible to enhance its stiffness with a refinement of the topology, i.e. by introducing holes without changing the total volume; by iterating this process one ends up with a microstructure (\cite{allaire:book}).
Moreover, a microscopic structure can endow the medium with anisotropic properties at the macroscopic level. In this way the material is made lighter, enhancing at the same time its stiffness in the direction of the load, and making it more compliant in the other directions. Therefore, anisotropy is a key feature to build structures optimized for a prescribed purpose.

Most of the TopOpt formulations and algorithms are well suited for anisotropic materials: the SIMP formulation itself, in spite of its name, can be applied to this case by simply replacing the isotropic constitutive law with an anisotropic one. However, there is few sense in optimizing the distribution of a single anisotropic material, since the optimal level and type of anisotropy may depend strongly from point to point of the domain. This observation leads in a natural way to the problem of finding the optimal distribution of a given amount of material, with the possibility of choosing in each point of the domain among void and a certain number of candidate anisotropic materials, featuring different properties.


A multimaterial TopOpt has been formulated in a SIMP framework in~\cite{sigmund1997design} and~\cite{gibiansky2000multiphase}. In these works periodic microstructures made of two isotropic phases and void are optimized at the microscopic level, employing one design variable to track the topology of the structure, and one design variable to control the balance between the two phases. 
A constraint on the total amount of material was set independently for each phase.

The formulation was generalized in~\cite{stegmann2005discrete,lund2005structural,lund2009buckling}, under the name of DMO (Discrete Material Optimization), to encompass an arbitrary number of phases. In these works the resource constraint is set on the total mass, rather that on mass of the single phases. For the numerical resolution of the optimization problem MMA was employed. In this formulation the interpolation between phases is such that the increase of the density associated with one phase automatically penalizes other phases. A different generalization of the SIMP (and RAMP) formulation was given in~\cite{hvejsel2011material}. Also in this case the optimization problem was solved by means of the MMA.

In \cite{zuo2017multi} an ordered SIMP interpolation was proposed to solve multimaterial TopOpt problems without the introduction of new variables. With this formulation however, since the choice among the candidate materials is determined by a single variable, gradient-based update schemes are somehow short-sighted, making the numerical solution very likely to fall into local minima which may be far from the global minimum, as shown by the numerical results reported in the paper.

In~\cite{yin2001topology} a multimaterial TopOpt problem is solved by means of a peak functions interpolation schemes. The interpolation of different phases is obtained by means of gaussian peaks, which are gradually made sharper to steer the solution towards the selection of a pure phases. The advantage of this approach is that the number of phases can be increased without changing the number of design variables.

The SFP (Shape Functions with Penalization) method~\cite{bruyneel2011sfp} exploits bilinear finite element shape functions as weights to interpolate among material phases. The shape functions are defined over a reference quadrangular element, each of whose vertexes represents a pure phase.

In~\cite{tavakoli2014alternating} an alternating active phase algorithm for multimaterial TopOpt was proposed. The algorithm splits the multimaterial optimization problem into a sequence of two-phases sub-problems, which are solved using traditional binary phase optimization solvers, by means of the OC method.

In~\cite{wang2004color,wang2005design} the Colour Level Set method, an extension of the level sets approach to multiple materials, was presented. In this approach $n$ level-set functions were employed to represent the distribution of $2^n$ different materials, mimicking the combination of $n$ primary colours. 

The phase field method was generalized to three phases (void included) in~\cite{wang2004synthesis}, exploiting a periodic phase-field variable which determines the mixture of the different phases. A different generalization, available for an arbitrary number of phases, was presented in~\cite{zhou2007multimaterial,blank2014multi} with a modified Cahn-Hilliard equation, and in~\cite{tavakoli2014multimaterial} with a modified Allen-Cahn equation. In these works, however, a mass constraint is set for each phase separately.

With regard to truss TopOpt, a multimaterial version of the minimum compliance problem was proposed in \cite{zhang2017multi}, where a problem with a generalized volume constraints setting was solved by means of a design update scheme (named ZPR), which updates the design variables associated with each volume constraint independently.


\subsection{TopOpt of self-assembling materials} \label{sec:self-assembly}

Unfortunately, there are technical constraints in realizing the microstructures predicted by the theory, since traditional manufacturing methods do not allow the realization of this kind of patterns at a low cost. Additive manufacturing allowed to partially fill this gap, making possible the realization of virtually any kind of periodic pattern at a very fine scale. However, the idea which motivates the present work is to realize microstructures in a completely different way, which can bring to much finer resolutions, by exploiting self-assembling materials. 

Self-assembly is an exciting chemical-physical phenomenon by which a disordered arrangement of components re-arrange in an organized structure, as a consequence of local interactions (typically of chemical nature). The possibility of predict and even control the kind of pattern opens to a huge number of applications, which go far beyond the structural field, thanks to the possibility of controlling the mechanical, magnetic, electronic and optical properties of the material.

One of the most studied class of self-assembling materials is that of diblock copolymers~\cite{hamley1998book}, linear chain molecules made of two subchains of monomers of type $A$ and $B$, covalentely joined to each other. When the temperature of a disordered melt of diblock copolymers undergoes a critical value, the competition between repulsion and chemical bounds of the subchains makes the melt segregate into A-rich and B-rich domains, forming regular periodic patterns, with characteristic length ranging from 1 to 50 $nm$. The resulting pattern depends mainly on the dimensionless parameter $\CHm = (N_A - N_B) / (N_A+N_B)$, taking values in $[-1,1]$, where $N_A$ and $N_B$ are the number of monomers of type $A$ and $B$ in each molecule. In three dimensions the observed patterns are of four types, namely lamellae, double gyroids, hexagonally packed cylinders, and spheres~\cite{hamley1998book}, while in two dimensions we have two possible patterns, namely stripes and spots~\cite{choksi2011_2d} (see Figure~\ref{fig:CHm_dependence}). We notice that inside each pattern class (e.g. stripes in 2D) the pattern changes continuously with respect to $\CHm$, while keeping fixed its topology (e.g. $A$-stripes get thicker with respect to $B$-stripes as $\CHm$ increases from $-0.2$ to $0.2$).

\begin{figure}
	\centering
	\includegraphics[width=.8\textwidth,trim={0 13.25cm 0 0}, clip]{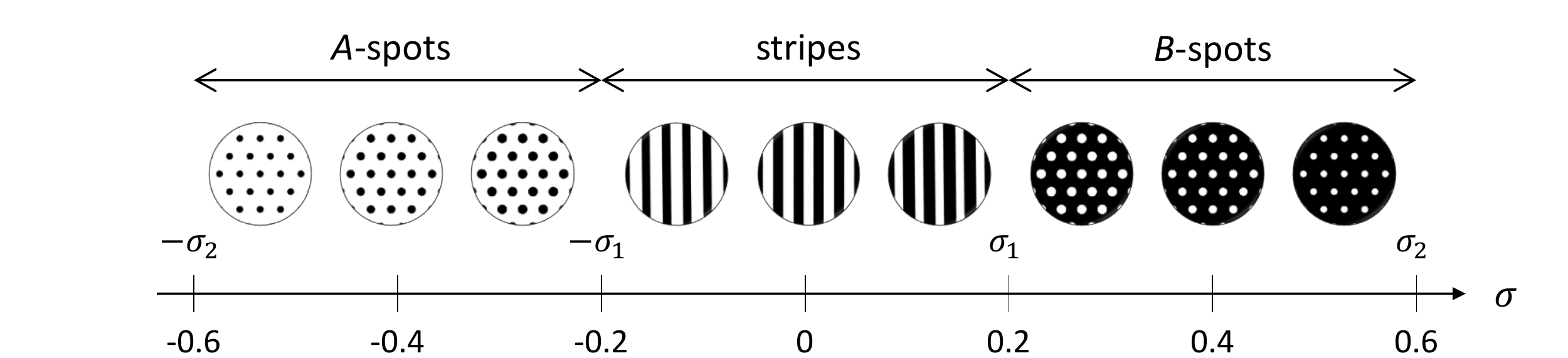}
	\caption{Regular patterns formed by diblock copolymers in 2D, for different values of $\CHm$, for $\gamma = 20$ (see Section~\ref{sec:applicationtodiblockcopolymers} for the meaning of $\gamma$). Black denotes material $A$, white material $B$.}
	\label{fig:CHm_dependence}
\end{figure}

From a mathematical perspective, the phase separation of diblock copolymers is modelled as the $H^{-1}$ gradient flow of the Ohta-Kawasaki functional~\cite{ohta1986equilibrium}, the Allen-Cahn/Cahn-Hilliard energy supplied with a non-local term modelling long range interactions. By solving the corresponding nonlinear parabolic PDE, known as Cahn-Hilliard-Oono (CHO) equation~\cite{oono1990cell}, until the system has reached the equilibrium, it is possible to get a numerical approximation of the periodic pattern associated with a given value of $\CHm$~\cite{choksi2009_3d,choksi2003derivation}. Finally, by means of homogenization theory~\cite{allaire:book,francfort1986homogenization,murat1997hconvergence}, it is possible to get the homogenized stiffness tensor $\tensel(\CHm)$, namely the effective tensor which models at the macroscopic level the elastic properties of the periodic pattern obtained with a given value of $\CHm$.

One of the motivations of the present work is the concept of obtaining optimized bodies by laying down a hot diblock copolymers solution, changing from place to place the composition $\CHm$, and letting the solution cool, so that phase separation takes place and produces the desired pattern. The task of forcing the orientation of the patterns can be accomplished by means of the so-called directed self-assembly (DSA) technologies (see~\cite{lin2005self,stoykovich2005directed,darling2007directing}).

This again leads to the problem of finding the optimal distribution inside a domain of a given number of candidate materials. In this case, strictly speaking, the material is the same for each value of $\CHm$ (namely an arrangement of diblock copolymers); however, since it exhibits different kinds of microstructures (lamellae, gyroids, etc.) according to the local value of $\CHm$, it is described by different homogenized tensors, and can thus be regarded as different materials, one for each kind of microstructure.


Motivated by these issues, in this work we introduce a novel generalization of the so-called Optimality Criteria method for the multimaterial TopOpt problem, where materials are allowed to change their orientation and have elastic properties which possibly depend on a scalar variable. 
The algorithm requires the determination of two Lagrange multipliers, namely a scalar, associated with the volume constraint, and a function of space, associated with a local constraint on the sum of the densities in each point of the domain. For the determination of such multipliers we employ two nested bisection loops.
The proposed algorithm is then applied to the problem of controlling the phase separation of diblock copolymers melts, in order to exploit such self-assembled microstructures to design bodies with high elastic performances.

\subsection{Outline}

The outline of the paper is as follows. In Section~\ref{sec:pb_topot} we formulate the \textit{basic} and the \textit{generalized} multimaterial TopOpt problems considered throughout this work, and we prove existence of solutions. In Section~\ref{sec:algorithm} we derive the optimality conditions for both problems, and we propose an algorithm for their numerical resolution. Moreover we prove that the proposed fixed-point scheme is compliant with the derived optimality conditions, and we show its well-definition. In Section~\ref{sec:applicationtodiblockcopolymers} we show how the problem of directing the self-assembly of diblock copolymers to get bodies with optimal elastic properties can be placed in the framework introduced in the previous sections. Section~\ref{sec:results} provides numerical examples assessing the properties of the proposed TopOpt scheme. Section~\ref{sec:conclusions} provides final remarks on the presented methodology and outlines for future research.

\section{Continuous optimization problems} \label{sec:pb_topot}

In this section we introduce the multimaterial TopOpt problem we want to solve. The problem is firstly introduced in its basic version~\eqref{eqn:pb_ct_base}, and it is later generalized in~\eqref{eqn:pb_ct_generalized}. Then the filtered version of both problems is introduced.

\subsection{Multi-material TopOpt: basic problem}

We suppose to be given a design domain $\Omega \subset \mathbb{R}^d$ (where $d=2$ or $3$), where the structure must be contained in. The boundary of the domain is partitioned into a Dirichlet part $\Gamma_D$, where the body is anchored, and a Neumann part $\Gamma_N$. We then suppose to be given a bulk load $\vec{b} \in H^{-1}(\Omega;\mathbb{R}^d)$ and a boundary load $\vec{t} \in H^{-1/2}(\Gamma_N)$. We want to optimize the mechanical response of the body to the forcing terms $\vec{b}$ and $\vec{t}$. Denoting by $\vec{u}$ the displacement field, the equilibrium equation is given by:
\begin{equation}
	\label{eqn:stateeqn_weak}
		\text{find } \vec{u} \in H^{1}_{\Gamma_D}(\Omega; \mathbb{R}^d) \quad \text{s.t.} \quad
		a(\vec{u},\vec{v}) = l(\vec{v}) \quad \forall \vec{v} \in H^{1}_{\Gamma_D}(\Omega; \mathbb{R}^d)\, ,
\end{equation}
where, denoting by $\symgrad{(\cdot)} = \frac{\nabla(\cdot) + {\nabla(\cdot)}^T}{2}$ the symmetric part of the gradient operator, we have defined the following bilinear and linear forms:
\begin{equation}
		a(\vec{u},\vec{v}) = \int_{\Omega} \tensel \symgrad{\vec{u}} : \symgrad{\vec{v}} \dx \, , \qquad
		l(\vec{v}) = \langle\vec{b},\vec{v}\rangle_{H^{-1},H^{1}_{\Gamma_D}} + \langle\vec{t},\vec{v}\rangle_{H^{-1/2},H^{1/2}} \, .
\end{equation}
The stiffness tensor $\tensel$ is determined by the layout of the structure, as detailed below. In this work we consider the problem of finding the optimal distribution of $N$ candidate materials, characterized by the stiffness tensors $(\tensel_\idxmat)_{\idxmat=1,\dots,N}$. 
The structure is described by $N$ design variables $\left(z_\idxmat:\Omega \rightarrow \{0,1\}\right)_{\idxmat=1,\dots,N}$, with the following meaning:
\begin{equation*}
z_\idxmat(\vec{x}) =  
\begin{cases}
1 & \text{if material $\idxmat$ is present in $\vec{x}$}  \\
0 & \text{otherwise.}
\end{cases}
\end{equation*}
To avoid materials to overlap, the following constraint should be adopted:
\begin{equation}
\label{eqn:topopt_multimat_constraint_sum_z}
\sum_{\idxmat=1}^{N} z_\idxmat(\vec{x}) \leq 1 \qquad \forall\, \vec{x} \in \Omega.
\end{equation}
Notice that in each point of the domain either a single $z_\idxmat$ is one and the others are zero, or all $z_\idxmat$ are zero (in the points where the body is not present). In this setting the resulting stiffness tensor is given by:
\begin{equation}
\label{eqn:SIMP_linear_interpolation}
\tensel(\vec{x}) = \sum_{\idxmat=1}^{N} z_\idxmat(\vec{x}) \tensel_\idxmat \, .
\end{equation}
To avoid trivial solutions we introduce a resource constraint $\masslim$ on the total mass. Denoting by $\rho_\idxmat$ the mass density of the material $\idxmat$, the mass of the body is given by:
\begin{equation}
\label{eqn:mass_definition}
\mass = \sum_{\idxmat=1}^{N} \int_{\Omega} z_\idxmat(\vec{x}) \rho_\idxmat \dx\, .
\end{equation}
	\begin{remark}
	The mass density $\rho_\idxmat$ may be interpreted more generally as the \textit{cost} associated to the material $\idxmat$, and the constraint $\masslim$ as the amount of available resources.
	\end{remark}
Since the compliance, to be minimized, is given by $\compliance = l(\vec{u})$, we are led to the following formulation of the minimum compliance problem:
\begin{subequations} \label{eqn:pb_ct_naive}	
\begin{empheq}[left=\empheqlbrace]{align}
   &     \textstyle 
\underset{(z_\idxmat)_\idxmat}{\text{minimize}} 
     & & \textstyle  
l(\vec{u}) 
\\ &     \textstyle 
\text{subject to:} 
     & & \textstyle 
a(\vec{u},\vec{v}) = l(\vec{v}) \quad \forall \vec{v} \in H^{1}_{\Gamma_D}(\Omega; \mathbb{R}^d) 
\\ & & & \textstyle 
\mass \leq \masslim 
\\ & & & \textstyle 
z_\idxmat(\vec{x}) \in \{0,1\} \qquad \forall\, \vec{x} \in \Omega \text{ for } \idxmat=1,\dots,N
\\ & & & \textstyle 
\sum_{\idxmat=1}^{N} z_\idxmat(\vec{x}) \leq 1 \qquad \forall\, \vec{x} \in \Omega \textstyle 
\\ &     \textstyle 
\text{where:} 
     & & \textstyle 
\tensel(\vec{x}) = \sum_{\idxmat=1}^{N} z_\idxmat(\vec{x}) \tensel_\idxmat 
\\ & & & \textstyle 
\mass = \sum_{\idxmat=1}^{N} \int_{\Omega} z_\idxmat(\vec{x}) \rho_\idxmat \dx\, .
\end{empheq}
\end{subequations}
As it is well-known, the discrete nature of~\eqref{eqn:pb_ct_naive} makes it hard to be handled when the number of design variables grows. Hence, typically in TopOpt the discrete design variables $z_\idxmat:\Omega \rightarrow \{0,1\}$ are replaced by continuous ones $z_\idxmat:\Omega \rightarrow [0,1]$, and the intermediate densities are penalized. One of the most popular penalization strategies comes under the name of SIMP (Solid Isotropic Material with Penalization, see~\cite{bendsoe:article89,bendsoe:book}), and envisages the substitution of the linear interpolation~\eqref{eqn:SIMP_linear_interpolation} with the following power-law interpolation, where $p$ is typically set to 3:
\begin{equation}
\label{eqn:SIMP_powerlaw_interpolation}
\tensel(\vec{x}) = \sum_{\idxmat=1}^{N} z_\idxmat(\vec{x})^p \tensel_\idxmat  \, .
\end{equation}
Moreover, to avoid singularities in the equilibrium equation, a lower bound $z_{min}>0$ for the design variables is introduced.

The SIMP formulation of the problem reads as follows:
\begin{subequations} \label{eqn:pb_ct_base}	
	\begin{empheq}[left=\empheqlbrace]{align}
	&     \textstyle 
	\underset{(z_\idxmat)_\idxmat}{\text{minimize}} 
	& & \textstyle  
	l(\vec{u}) 
	\\ &     \textstyle 
	\text{subject to:} 
	& & \textstyle 
	a(\vec{u},\vec{v}) = l(\vec{v}) \quad \forall \vec{v} \in H^{1}_{\Gamma_D}(\Omega; \mathbb{R}^d) 
	\\ & & & \textstyle 
	\mass \leq \masslim 
	\\ & & & \textstyle 
	z_{min} \leq z_\idxmat(\vec{x}) \leq 1 \qquad \forall\, \vec{x} \in \Omega \text{ for } \idxmat=1,\dots,N
	\\ & & & \textstyle 
	\sum_{\idxmat=1}^{N} z_\idxmat(\vec{x}) \leq 1 \qquad \forall\, \vec{x} \in \Omega \textstyle 
	\\ &     \textstyle 
	\text{where:} 
	& & \textstyle 
	\tensel(\vec{x}) = \sum_{\idxmat=1}^{N} z_\idxmat(\vec{x})^p \tensel_\idxmat 
	\\ & & & \textstyle 
	\mass = \sum_{\idxmat=1}^{N} \int_{\Omega} z_\idxmat(\vec{x}) \rho_\idxmat \dx\, .
	\end{empheq}
\end{subequations}

\subsection{Multi-material TopOpt: generalized problem}

In this work, besides the basic problem~\eqref{eqn:pb_ct_base}, we will consider a generalized multimaterial TopOpt problem. First of all, since the mechanical response of an anisotropic structure depends strongly on its orientation, we consider the problem of optimizing the local orientation of the materials. For this purpose we introduce a new set of control variables, $\theta_\idxmat(\vec{x})$, for $\idxmat=1,\dots,N$. Such variables, in general, belong to the rotation group SO($d$), but in two dimensions they can be represented by a point of the interval $[0,2\pi)$, the rotation angle, while in three dimensions they can be represented by the three Euler's angles. We denote by $\Qrot(\theta)$ the 8-th order tensor associated with the rotation of 4-th order tensors. For instance, in two dimensions, interpreting $\theta$ as the counter-clockwise rotation angle of the material, we have:
\begin{equation}
\label{eqn:tensorQ_tensorR}
Q_{ijklpqrs}(\theta) = R_{ip}(\theta)R_{jq}(\theta)R_{kr}(\theta)R_{ls}(\theta) 
\, , \qquad
\Rrot(\theta) =
\begin{pmatrix}
\cos(\theta) & -\sin(\theta) \\
\sin(\theta) & \cos(\theta)
\end{pmatrix} \, .
\end{equation}
The tensor $\Qrot(\theta)$ allows to model the stiffness tensor of the rotated materials, by replacing expression~\eqref{eqn:SIMP_powerlaw_interpolation} with:
\begin{equation*}
\tensel(\vec{x}) = \sum_{\idxmat=1}^{N} z_\idxmat(\vec{x})^p  \Qrot(\theta_\idxmat(\vec{x})) \tensel_\idxmat \, .
\end{equation*}
As a further generalization, we consider the case when the elastic properties of the candidate materials depend on an additional scalar variable. In this case, instead of $N$ materials, we have $N$ \textit{classes} of materials, whose properties are parametrized by the variables $\matm_\idxmat \in [\underline{\matm_\idxmat},\overline{\matm_\idxmat}]$, for $\idxmat=1,\dots,N$. For instance, when the anisotropic material consists in a microstructure, the variables $\matm_\idxmat$ may describe its microscopical properties. This is the case of self-assembling diblock copolymers, where the $N$ classes are the possible topologies of patterns (e.g. in 2D $A$-spots, stripes and $B$-spots), while the monomers proportion $\CHm$ plays the role of the variables $\matm_\idxmat$ (see Section~\ref{sec:self-assembly}). The application of the current framework to the case of diblock copolymers will be extensively treated in Section~\ref{sec:applicationtodiblockcopolymers}.

Notice that also the specific weight $\rho_\idxmat$ possibly depends on $\matm_\idxmat$. In the following we will suppose that the meaning of $\matm_\idxmat$ is chosen in such a way that $\matm_\idxmat \mapsto \rho_\idxmat(\matm_\idxmat)$ is a not decreasing function.
Therefore, assuming that the parameters $\matm_\idxmat$ can change in space (i.e. $\matm_\idxmat = \matm_\idxmat(\vec{x})$), we consider the following generalized problem:
\begin{subequations} \label{eqn:pb_ct_generalized}	
	\begin{empheq}[left=\empheqlbrace]{align}
	&     \textstyle 
	\underset{(z_\idxmat,\matm_\idxmat,\theta_\idxmat)_\idxmat}{\text{minimize}}
	& & \textstyle  
	l(\vec{u}) 
	\\ &     \textstyle 
	\text{subject to:} 
	& & \textstyle 
	a(\vec{u},\vec{v}) = l(\vec{v}) \quad \forall \vec{v} \in H^{1}_{\Gamma_D}(\Omega; \mathbb{R}^d) 
	\\ & & & \textstyle 
	\mass \leq \masslim 
	\\ & & & \textstyle 
	z_{min} \leq z_\idxmat(\vec{x}) \leq 1 \qquad \forall\, \vec{x} \in \Omega \text{ for } \idxmat=1,\dots,N
	\\ & & & \textstyle 
	\sum_{\idxmat=1}^{N} z_\idxmat(\vec{x}) \leq 1 \qquad \forall\, \vec{x} \in \Omega \textstyle 
	\\ & & & \textstyle 
	\underline{\matm}_\idxmat \leq \matm_\idxmat(\vec{x}) \leq \overline{\matm}_\idxmat \qquad \forall\, \vec{x} \in \Omega \text{ for } \idxmat=1,\dots,N
	\\ &     \textstyle 
	\text{where:} 
	& & \textstyle 
	\tensel(\vec{x}) = \sum_{\idxmat=1}^{N} z_\idxmat(\vec{x})^p \Qrot(\theta_\idxmat(\vec{x})) \tensel_\idxmat(\matm_\idxmat(\vec{x}))
	\\ & & & \textstyle 
	\mass = \sum_{\idxmat=1}^{N} \int_{\Omega} z_\idxmat(\vec{x}) \rho_\idxmat(\matm_\idxmat(\vec{x})) \dx \, .
	\end{empheq}
\end{subequations}
In the following, we will suppose that the parameters $z_{min}$ and $\masslim$ have non-trivial values, namely:
\begin{subequations}
	\label{eqn:topopt_multimat_nontrivialparam}
	\begin{equation}
	\label{eqn:topopt_multimat_nontrivialparam1}
	0 < z_{min} < \frac{1}{N} \, ,
	\end{equation}
	\begin{equation}
	\label{eqn:topopt_multimat_nontrivialparam2}
	\masslim > |\Omega| z_{min} \sum_{\idxmat=1}^{N} \rho_\idxmat(\overline{\matm}_\idxmat)\, .
	\end{equation}
\end{subequations}

\subsection{On the existence of solutions}
It is well known that TopOpt problems lack of existence of solution in general (in~\cite{allaire:book} some counterexamples are shown). The ill-posedness of TopOpt problems is caused by the possibility of progressively improving the performances of a structure by topology refinements. In other words, it is possible to build a sequence of finer and finer designs with growing stiffness, without the sequence converging to any admissible design.

%

This translates into numerical instabilities when one tries to solve computationally the problem. When the problem is discretized, the resolution of the solution is bounded by that of the computational mesh. Thus, when the mesh is refined, the solution may exhibit a refinement of the topology, hindering mesh independence. Moreover, numerical solutions obtained with linear finite elements often exhibit checkerboard patterns, because of a bad numerical modelling that overestimates the stiffness of this kind of configurations~\cite{diaz1995checkerboard}, which have no physical relevance and are mesh-dependent.

This kind of problems is typically tackled by restricting the set of admissible design ruling out the possibility for fine structures to formate. This can be obtained by explicitly penalizing the \textit{total variation} ($TV$) of the density (by the so-called \textit{perimeter penalization}, see for instance~\cite{borrvall:perimeter}), or by filtering the design variables or their sensitivities by means of a convolution kernel (see for instance~\cite{sigmund:checkerboards}). 


The approach considered in this work is equivalent to that known in the literature as \textit{density filtering}, but since in this framework there are other design variables besides the densities ($\matm_\idxmat$ and $\theta_\idxmat$) we will refer to it as \textit{variables filtering}. 

Consider a non-negative Lipschitz convolution kernel $K$, such as:
\begin{equation}
\label{eqn:topopt_K_definition}
K(\vec{x}) = \max\{0, \|\vec{x}\| - r_{min} \}\, ,
\end{equation}
where the convolution radius $r_{min}$ determines the minimum admissible length-scale. Given a generic variable $\psi$ (where $\psi = z_\idxmat,\matm_\idxmat,\theta_\idxmat$), we denote by $\hat{\psi}$ the filtered variable, obtained through a normalized convolution with $K$, i.e.
\begin{equation}
\label{eqn:topopt_filtering_definition}
\hat{\psi}(\vec{x}) 
= \frac{(K * \psi) (\vec{x})}{(K * 1) (\vec{x})}  
=\frac{\int_{\Omega}  K(\vec{x} - \vec{z}) \psi(\vec{z}) \diff{\vec{z}}} {\int_{\Omega}K(\vec{x} - \vec{z}) \diff{\vec{z}}}  \, .
\end{equation}
In the following we will suppose that the denominator in~\eqref{eqn:topopt_filtering_definition} is bounded from below (notice that this assumption is violated only by very weird domains):
\begin{equation}
\label{eqn:topopt_filtering_lowerboundK}
\int_{\Omega}K(\vec{x} - \vec{z}) \diff{\vec{z}} \geq K_{min} \qquad \forall \, \vec{x} \in \Omega\, .
\end{equation}
Then, in equations~\eqref{eqn:SIMP_powerlaw_interpolation} and~\eqref{eqn:mass_definition}, we replace the original variables with the filtered ones. We remark that the filtered variables satisfy the same constraints as the original ones, thanks to the following lemma, whose proof follows easily from the definition of filtered variables, and from linearity of the convolution operator:
\begin{lemma}
	\label{lemma:filtered_constraints}
	Suppose that $\psi \in L^1(\Omega)$ is such that $\psi(\vec{x}) \leq C$ a.e. in $\Omega$ for some $C \in \mathbb{R}$. Then its filtered counterpart satisfies $\hat{\psi}(\vec{x}) \leq C$ for all $\vec{x} \in \Omega$.
\end{lemma}


To sum up, the filtered optimization problem reads as follows:
\begin{subequations} \label{eqn:pb_ct_generalized_filtered}	
	\begin{empheq}[left=\empheqlbrace]{align}
	&     \textstyle 
	\underset{(z_\idxmat,\matm_\idxmat,\theta_\idxmat)_\idxmat}{\text{minimize}}
	& & \textstyle  
	l(\vec{u}) 
	\\ &     \textstyle 
	\text{subject to:} 
	& & \textstyle 
	a(\vec{u},\vec{v}) = l(\vec{v}) \quad \forall \vec{v} \in H^{1}_{\Gamma_D}(\Omega; \mathbb{R}^d) 
	\\ & & & \textstyle 
	\mass \leq \masslim 
	\\ & & & \textstyle 
	z_{min} \leq z_\idxmat(\vec{x}) \leq 1 \qquad \forall\, \vec{x} \in \Omega \text{ for } \idxmat=1,\dots,N
	\\ & & & \textstyle 
	\sum_{\idxmat=1}^{N} z_\idxmat(\vec{x}) \leq 1 \qquad \forall\, \vec{x} \in \Omega \textstyle 
	\\ & & & \textstyle 
	\underline{\matm}_\idxmat \leq \matm_\idxmat(\vec{x}) \leq \overline{\matm}_\idxmat \qquad \forall\, \vec{x} \in \Omega \text{ for } \idxmat=1,\dots,N
	\\ &     \textstyle 
	\text{where:} 
	& & \textstyle 
	\tensel(\vec{x}) = \sum_{\idxmat=1}^{N} \hat{z}_\idxmat(\vec{x})^p \Qrot(\hat{\theta}_\idxmat(\vec{x})) \tensel_\idxmat(\hat{\matm}_\idxmat(\vec{x}))
	\\ & & & \textstyle 
	\mass = \sum_{\idxmat=1}^{N} \int_{\Omega} \hat{z}_\idxmat(\vec{x}) \rho(\hat{\matm}_\idxmat(\vec{x})) \dx \, .
	\end{empheq}
\end{subequations}

\begin{remark} 
In the following we will refer to the filtered variables as \virgolette{physical variables}, and to the original ones as \virgolette{control variables}. Indeed the first ones are the only ones with physical relevance, while the second ones are mere mathematical tools. Therefore the solution of the optimization problem is represented by the physical variables. This is the reason why the total mass is computed with respect to the filtered physical variables.
\end{remark}

The regularizing effect of the filtering operator allows to recover compactness in the set of admissible designs, making it possible to prove the following existence result for the minimum compliance problem.

\begin{theorem}
	\label{thm:topopt_exists}
	Let $\Omega \in \mathbb{R}^d$  be a bounded Lipschitz domain. Consider a subset $\Gamma_D$ of $\partial \Omega$ of positive $(d-1)$-dimensional Hausdorff measure. Let $K \in L^1(\mathbb{R}^d)$ be a non-negative Lipschitz function satisfying~\eqref{eqn:topopt_filtering_lowerboundK}. Suppose then that the specific weights $\rho_\idxmat(\matm)$ are continuous functions of their arguments, and that the stiffness tensors $\tensel_\idxmat(\matm)$ are as well continuous in $\matm$ and equicoercive, i.e. there exists a constant $\alpha > 0$ such that:
	\begin{equation*}
	\tensel_\idxmat(\matm) \vec{\xi} : \vec{\xi} \geq \alpha | \vec{\xi} |^2 \qquad 
	\forall \, \vec{\xi} \in \mathbb{R}^{d \times d}, \quad 
	\forall \matm \in [\underline{\matm_\idxmat},\overline{\matm_\idxmat}], \quad
	\text{for } \idxmat=1,\dots,N .
	\end{equation*} 
	Then, for every forcing term $(\vec{b},\vec{t}) \in H^{-1}(\Omega;\mathbb{R}^d) \times H^{-1/2}(\Gamma_N)$, the problem~\eqref{eqn:pb_ct_generalized_filtered} admits at least a solution.
\end{theorem}

The proof, reported in Appendix~\ref{app:proof_topopt_exists}, follows the direct method of Calculus of Variations. Notice that a proof for the existence of solutions of the standard minimum compliance problem with density filtering is reported in~\cite{bourdin2001filters}. Our proof exploits similar techniques, but, since in~\cite{bourdin2001filters} the filtering operator is not normalized (we have simply $\hat{\psi}(\vec{x}) = (K * \psi) (\vec{x})$), the density has to be extended to the whole $\mathbb{R}^d$ in order to avoid the smoothing effect around the boundary of the domain. Our approach overcomes this inconvenience.

\section{Discrete optimization problems and algorithms} \label{sec:algorithm}

The target of this section is the numerical solution of problem~\eqref{eqn:pb_ct_generalized_filtered}. To reach this goal several steps are required: we start considering the simplest version of the problem, namely the unfiltered formulation~\eqref{eqn:pb_ct_base}, and we derive a finite element formulation and an optimization algorithm for this version (Section~\ref{sec:algorithm_base_nf}). Then, in Section~\ref{sec:algorithm_generalized_nf} we extend the algorithm to the generalized problem~\eqref{eqn:pb_ct_generalized}, and finally, in Section~\ref{sec:algorithm_generalized_filt}, to the filtered problem~\eqref{eqn:pb_ct_generalized_filtered}.

We remark that intermediate formulations are of some interest on their own. Indeed, the basic problem~\eqref{eqn:pb_ct_base} can be useful for some applications, such as discrete orientation optimization (see Section~\ref{sec:results:basepb}); Moreover, non-filtered formulations provide the basis for the application of different filtering techniques to the problems at hand.



\subsection{Basic problem}
\label{sec:algorithm_base_nf}

First, we consider a finite element discretization of the state equation~\eqref{eqn:stateeqn_weak}. The state variable $\vec{u}$ is discretized through continuous finite elements (P1 for instance), while for the design variables $z_\idxmat$, $\matm_\idxmat$ and $\theta_\idxmat$ piecewise constant elements (P0) are employed.

Consider a triangulation $\mesh = \{e_l\}_{l=1,...,N_e}$ in $N_e$ elements of the reference domain $\Omega$. We introduce a finite elements subspace of the space $H^{1}_{\Gamma_D}(\Omega; \mathbb{R}^d)$, which we denote by $\fespace$, defined over the mesh $\mesh$, and we denote by $\{\vec{\phi}_j\}_{j=1,\dots,N_V}$ a basis of the space $\fespace$. The finite elements discretization of the state variable
\begin{equation*}
\vec{u}_h(\vec{x})= \sum_{j=1}^{N_V} u_j \vec{\phi}_j(\vec{x})
\end{equation*}
is associated to the vector $\vec{U} = (u_1,\dots,u_{N_V}) \in \mathbb{R}^{N_V}$.

On the other hand, the design variables are discretized by the piecewise constant functions:
\begin{equation*}
z_\idxmat^h(\vec{x}) = \sum_{l=1}^{N_e} z_{\idxmat,l} \indicator{e_l}{\vec{x}}  \qquad \idxmat=1,\dots,N .
\end{equation*}
First, we consider problem \eqref{eqn:pb_ct_base}. Introducing the stiffness matrix $\vec{K} \in \mathbb{R}^{N_V \times N_V}$ and the right-hand side $\vec{f} \in \mathbb{R}^{N_V}$ defined as
\begin{equation*}
\begin{split}
K_{ij} &= \sum_{l=1}^{N_e } \sum_{\idxmat=1}^{N} (z_{\idxmat,l})^p  \int_{e_l} \tensel_\idxmat \symgrad{\vec{\phi}_i} : \symgrad{\vec{\phi}_j} \dx\, , \\
f_i &= l(\vec{\phi}_i) = \langle\vec{b},\vec{\phi}_i\rangle_{H^{-1},H^{1}_{\Gamma_D}} + \langle\vec{t},\vec{\phi}_i\rangle_{H^{-1/2},H^{1/2}} \, ,
\end{split}
\end{equation*}
the Galerkin approximation of the state equation simply reads
\begin{equation}
\label{eqn:discrete_FEM_system}
\vec{K} \vec{U} = \vec{f} \, .
\end{equation}
The total mass associated with the design $(z_{\idxmat,l})_{\idxmat,l}$ is given by the formula:
\begin{equation}
\label{eqn:discrete_mass}
\begin{split}
\mass 
&= \sum_{\idxmat=1}^{N} \int_{\Omega} z_\idxmat^h(\vec{x}) \rho_\idxmat \dx 
= \sum_{\idxmat=1}^{N} \sum_{l=1}^{N_e } |e_l| z_{\idxmat,l} \rho_\idxmat \, .
\end{split}
\end{equation}
To sum up, the finite element based discrete version of~\eqref{eqn:pb_ct_base} reads as follows:
\begin{subequations} \label{eqn:pb_fem_base}
	\begin{empheq}[left=\empheqlbrace]{align}
	&     \textstyle 
	\underset{(z_{\idxmat,l})_{\idxmat,l}}{\text{minimize}}
	& & \textstyle  
	\vec{f}^T\vec{U}
	\\ &     \textstyle 
	\text{subject to:} 
	& & \textstyle \label{eqn:pb_fem_base:stateeqn}
	\vec{K} \vec{U} = \vec{f} 
	\\ & & & \textstyle 
	\mass \leq \masslim 
	\\ & & & \textstyle 
	z_{min} \leq z_{\idxmat,l} \leq 1 \qquad \text{for } l=1,\dots,N_e \text{, } \idxmat=1,\dots,N
	\\ & & & \textstyle \label{eqn:pb_fem_base:constraint_sum_z}
	\sum_{\idxmat=1}^{N} z_{\idxmat,l} \leq 1 \qquad \text{for } l=1,\dots,N_e 
	\\ &     \textstyle 
	\text{where:} 
	& & \textstyle 
	K_{ij} = \sum_{l=1}^{N_e } \sum_{\idxmat=1}^{N} (z_{\idxmat,l})^p  \int_{e_l}\tensel_\idxmat \symgrad{\vec{\phi}_i} : \symgrad{\vec{\phi}_j} \dx
	\\ & & & \textstyle 
	\mass = \sum_{\idxmat=1}^{N} \sum_{l=1}^{N_e} |e_l| z_{\idxmat,l} \rho_\idxmat \, .
	\end{empheq}
\end{subequations}
We now describe a numerical strategy to solve~\eqref{eqn:pb_fem_base}. Notice that, thanks to the constraint~\eqref{eqn:pb_fem_base:constraint_sum_z} the upper bound on $z_\idxmat$ is redundant, since the variables $z_\idxmat$ are positive (we are supposing $z_{min} > 0$). Therefore we can omit the upper-bound constraint on $z_\idxmat$ when writing the Lagrangian.
We introduce the Lagrange multipliers $\bar{\vec{U}}$,  $\Lambda$, $\vec{\lambda}_\idxmat$ and $\vec{\mu}$, associated with the state equation~\eqref{eqn:pb_fem_base:stateeqn}, the mass constraint, the lower-bound of the design variables and the constraint~\eqref{eqn:pb_fem_base:constraint_sum_z} respectively. Therefore the Lagrangian reads:
\begin{equation}
\label{eqn:FEM_nf_lagrangiana}
\begin{split}
\mathcal{L}
&= \vec{f}^T\vec{U} 
- \bar{\vec{U}}^T\left( \vec{K}\vec{U} - \vec{f}\right) 
+ \Lambda \left( \sum_{\idxmat=1}^{N} \sum_{l=1}^{N_e } |e_l| z_{\idxmat,l} \rho_\idxmat   - \masslim \right) \\ 
& + \sum_{\idxmat=1}^{N} \sum_{l=1}^{N_e} \lambda_{\idxmat,l} \left(z_{min} -  z_{\idxmat,l} \right)
+ \sum_{l=1}^{N_e} \mu_l \left(\sum_{\idxmat=1}^{N} z_{\idxmat,l} - 1 \right) \, .
\end{split}
\end{equation}
Like in the monomaterial case~\cite{bendsoe:book}, by imposing null first variation of the Lagrangian with respect to the state variable it turns out that the problem is self-adjoint, therefore $\bar{\vec{U}} = {\vec{U}}$. On the other hand, the first variation with respect to each design variable $z_{\idxmat,l}$ reads as follows:
\begin{equation}
\label{eqn:topopt_multimat_dLdzh}
\frac{\partial\mathcal{L}}{\partial z_{\idxmat,l}} =
-\vec{U}^T\frac{\partial \vec{K}}{\partial z_{\idxmat,l}} \vec{U}
+\Lambda |e_l| \rho_\idxmat 
-\lambda_{\idxmat,l}
+ \mu_l 
\, ,
\end{equation}
where the derivative of the stiffness tensor can be computed as follows:
\begin{equation}
\label{eqn:discrete_dE_dz}
\left(\frac{\partial \vec{K}}{\partial z_{\idxmat,l}}\right)_{i,j} = p \, z_{\idxmat,l}^{p-1}   \int_{e_l} \tensel_\idxmat \symgrad{\vec{\phi}_i} : \symgrad{\vec{\phi}_j}  \dx\, .
\end{equation}
Therefore the candidate minima should be looked for in the set of designs fulfilling the Karush-Kuhn-Tucker conditions:
\begin{equation}
\label{eqn:topopt_multimat_optcond_syst1}
\begin{cases}
\vec{U}^T\frac{\partial \vec{K}}{\partial z_{\idxmat,l}} \vec{U} = \Lambda |e_l| \rho_\idxmat -\lambda_{\idxmat,l}+ \mu_l 
\qquad & \forall\, l , \idxmat \\
\Lambda \geq 0, \quad \mass \leq \masslim,\quad \Lambda \left( \masslim - \mass \right) = 0 \\
\lambda_{\idxmat,l} \geq 0, \quad z_{min}\leq z_{\idxmat,l},\quad\lambda_{\idxmat,l}(z_{min}-z_{\idxmat,l}) = 0 
\qquad & \forall\, l ,\idxmat\\
\mu_l \geq 0, \quad \sum_{\idxmat=1}^{N} z_{\idxmat,l} \leq 1,\quad\mu_l\left(\sum_{\idxmat=1}^{N} z_{\idxmat,l} - 1\right) = 0 
\qquad & \forall\, l \, .
\end{cases}
\end{equation}
With a similar approach to that employed to derive the OC method (see~\cite{bendsoe:book}), we can rewrite~\eqref{eqn:topopt_multimat_optcond_syst1} in the following equivalent way, where $\epsilon$ is a positive constant:
\begin{equation}
\label{eqn:topopt_multimat_optcond_syst2}
\begin{cases}
\frac{\vec{U}^T\frac{\partial \vec{K}}{\partial z_{\idxmat,l}} \vec{U} + \epsilon}{\Lambda |e_l| \rho_\idxmat + \mu_l+ \epsilon} 
\begin{cases}
{}=1 \qquad & \text{if $z_{\idxmat,l} > z_{min}$}\\
{}\leq 1 \qquad & \text{if $z_{\idxmat,l} = z_{min}$}\\
\end{cases}
\qquad & \forall\, l ,\idxmat\\
\Lambda \geq 0, \quad \mass \leq \masslim,\quad \Lambda \left( \masslim - \mass \right) = 0 \\
\mu_l \geq 0, \quad \sum_{\idxmat=1}^{N} z_{\idxmat,l} \leq 1,\quad\mu_l\left(\sum_{\idxmat=1}^{N} z_{\idxmat,l} - 1\right) = 0 
\qquad & \forall\, l \, .
\end{cases}
\end{equation}
Guided by~\eqref{eqn:topopt_multimat_optcond_syst2}, we devise the following fixed-point iterative scheme:
\begin{equation}
\label{eqn:discrete_fixedpoint_base}
\begin{split}
z_{\idxmat,l}^{k+1} &=\max\left\{z_{min}, B_{\idxmat,l}^k z_{\idxmat,l}^{k} \right\} \, ,\\
B_{\idxmat,l}^k &= \frac{\left(\vec{U}^k\right)^T\frac{\partial \vec{K}}{\partial z_{\idxmat,l}} \vec{U}^k +\epsilon}{\Lambda^k |e_l| \rho_\idxmat + \mu^k_l
	+\epsilon}
\, ,
\end{split}
\end{equation}
with $\vec{U}^k$ denoting the state variable at the iteration $k$, and the Lagrange multipliers $\Lambda^k \in \mathbb{R}$ and $\vec{\mu}^k \in \mathbb{R}^{N_e}$ are chosen in such a way that it holds true:
\begin{subequations}
	\begin{alignat}{1}
	\label{eqn:FEM_nf_cond1}
	&\mu^k_l \geq 0, \quad \sum_{\idxmat=1}^{N} z_{\idxmat,l}^{k+1} \leq 1,\quad\mu^k_l\left(\sum_{\idxmat=1}^{N} z_{\idxmat,l}^{k+1} - 1\right) = 0 
	\qquad \forall\, l \\
	\label{eqn:FEM_nf_cond2}
	&\Lambda^k \geq 0, \quad \mass^{k+1} \leq \masslim,\quad \Lambda^k \left( \masslim - \mass^{k+1} \right) = 0\, .
	\end{alignat}
\end{subequations}
One may wonder whether the scheme~\eqref{eqn:discrete_fixedpoint_base} is well defined or not, namely whether:
\begin{enumerate}[label=Q\arabic*]
	\item it is always possible to find a couple $\left(\Lambda^k, \vec{\mu}^k\right)$ such that the conditions~\eqref{eqn:FEM_nf_cond1} and~\eqref{eqn:FEM_nf_cond2} hold true; \label{quest:existence}
	\item the couple $\left(\Lambda^k, \vec{\mu}^k\right)$ fulfilling~\eqref{eqn:FEM_nf_cond1} and~\eqref{eqn:FEM_nf_cond2}, if it exists, is unique. \label{quest:uniqueness}
\end{enumerate}
We postpone the answer to Section~\ref{sec:welldefinition}, where it will be given in the most general case, while in the following we generalize the current scheme to problem~\eqref{eqn:pb_ct_generalized}, and we introduce the filtered version of the algorithm.

\subsection{Generalized problem} \label{sec:algorithm_generalized_nf}

Proceeding as in the previous section, it is immediate to see that the discrete version of problem~\eqref{eqn:pb_ct_generalized} reads as
\begin{subequations} \label{eqn:pb_fem_generalized}
	\begin{empheq}[left=\empheqlbrace]{align}
	&     \textstyle 
	\underset{(z_{\idxmat,l},\matm_{\idxmat,l},\theta_{\idxmat,l})_{\idxmat,l}}{\text{minimize}}
	& & \textstyle  
	\vec{f}^T\vec{U}
	\\ &     \textstyle 
	\text{subject to:} 
	& & \textstyle 
	\vec{K} \vec{U} = \vec{f} 
	\\ & & & \textstyle \label{eqn:pb_fem_generalized:masslimit}
	\mass \leq \masslim 
	\\ & & & \textstyle \label{eqn:pb_fem_generalized:zbounds}
	z_{min} \leq z_{\idxmat,l} \leq 1 \qquad \text{for } l=1,\dots,N_e \text{, } \idxmat=1,\dots,N
	\\ & & & \textstyle \label{eqn:pb_fem_generalized:zsum1}
	\sum_{\idxmat=1}^{N} z_{\idxmat,l} \leq 1 \qquad \text{for } l=1,\dots,N_e 
	\\ & & & \textstyle \label{eqn:pb_fem_generalized:mbounds}
	\underline{\matm}_\idxmat \leq \matm_{\idxmat,l} \leq \overline{\matm}_\idxmat \qquad \text{for } l=1,\dots,N_e \text{, } \idxmat=1,\dots,N
	\\ &     \textstyle 
	\text{where:} 
	& & \textstyle \label{eqn:pb_fem_generalized:Kdef}
	K_{ij} = \sum_{l=1}^{N_e } \sum_{\idxmat=1}^{N} (z_{\idxmat,l})^p   \int_{e_l}\Qrot(\theta_{\idxmat,l}) \tensel_\idxmat(\matm_{\idxmat,l}) \symgrad{\vec{\phi}_i} : \symgrad{\vec{\phi}_j} \dx
	\\ & & & \textstyle \label{eqn:pb_fem_generalized:Mdef}
	\mass = \sum_{\idxmat=1}^{N} \sum_{l=1}^{N_e} |e_l| z_{\idxmat,l} \rho_\idxmat(\matm_{\idxmat,l}) \, .
	\end{empheq}
\end{subequations}
The Lagrangian associated to the problem reads as follows:
\begin{equation*}
\begin{split}
\mathcal{L}
&= \vec{f}^T\vec{U} 
- \bar{\vec{U}}^T\left( \vec{K}\vec{U} - \vec{f})\right) 
+ \Lambda \left( \sum_{\idxmat=1}^{N} \sum_{l=1}^{N_e } |e_l| z_{\idxmat,l} \rho_\idxmat(\matm_{\idxmat,l})   - \masslim \right) \\ 
& + \sum_{\idxmat=1}^{N} \sum_{l=1}^{N_e} \lambda_{\idxmat,l} \left(z_{min} -  z_{\idxmat,l} \right)
+ \sum_{l=1}^{N_e} \mu_l \left(\sum_{\idxmat=1}^{N} z_{\idxmat,l} - 1 \right)  \\
& + \sum_{\idxmat=1}^{N} \sum_{l=1}^{N_e} \left[ 
\gamma_{\idxmat,l}^+ \left(\matm_{\idxmat,l} - \overline{\matm}_\idxmat \right) + 
\gamma_{\idxmat,l}^- \left(\underline{\matm}_\idxmat - \matm_{\idxmat,l} \right)
\right]\, ,
\end{split}
\end{equation*}
where we have introduced two further multipliers $\vec{\gamma}_\idxmat^+$ and $\vec{\gamma}_\idxmat^-$, for the upper and lower bounds of $\matm_\idxmat$. Thanks to the symmetry of the matrix $\vec{K}$, the problem turns out to be self-adjoint. By deriving the Lagrangian with respect to each design variable we get the following set of optimality conditions:
\begin{equation}
\label{eqn:FEM_nf_optcond_syst2}
\left\{ 
\begin{aligned}
&\frac{\vec{U}^T\frac{\partial \vec{K}}{\partial z_{\idxmat,l}} \vec{U} + \epsilon }{\Lambda |e_l| \rho_\idxmat(\matm_{\idxmat,l}) + \mu_l + \epsilon} 
\begin{cases}
{}=1 \qquad & \text{if $z_{\idxmat,l} > z_{min}$}\\
{}\leq 1 \qquad & \text{if $z_{\idxmat,l} = z_{min}$}\\
\end{cases}
\qquad & &\forall\, l ,\idxmat\\
&\frac{\vec{U}^T\frac{\partial \vec{K}}{\partial \matm_{\idxmat,l}}\vec{U} + \epsilon}{\Lambda |e_l| z_{\idxmat,l} \rhomprime(\matm_{\idxmat,l}) + \epsilon} 
\begin{cases}
{}\geq 1 \qquad & \text{if $\matm_{\idxmat,l} = \overline{\matm}_\idxmat $}\\
{}=1 \qquad & \text{if $\underline{\matm}_\idxmat < \matm_{\idxmat,l} < \overline{\matm}_\idxmat$}\\
{}\leq 1 \qquad & \text{if $\matm_{\idxmat,l} = \underline{\matm}_\idxmat$}\\
\end{cases}
\qquad & &\forall\, l ,\idxmat\\
&\vec{U}^T\frac{\partial \vec{K}}{\partial \theta_{\idxmat,l}}\vec{U} = 0
\qquad & &\forall\, l ,\idxmat\\
&\mu_l \geq 0, \quad \sum_{\idxmat=1}^{N} z_{\idxmat,l} \leq 1,\quad\mu_l\left(\sum_{\idxmat=1}^{N} z_{\idxmat,l} - 1\right) = 0 
\qquad & &\forall\, l \\
&\Lambda \geq 0, \quad \mass \leq \masslim,\quad \Lambda \left( \mass - \masslim \right) = 0 \, ,& &
\end{aligned}
\right.
\end{equation} 
where
\begin{equation}
\label{eqn:FEM_nf_dE}
\begin{split}
\left(\frac{\partial \vec{K}}{\partial z_{\idxmat,l}}\right)_{i,j} & =  p \, (z_{\idxmat,l})^{p-1}   \int_{e_l} \Qrot(\theta_{\idxmat,l}) \tensel_\idxmat(\matm_{\idxmat,l})\symgrad{\vec{\phi}_i} : \symgrad{\vec{\phi}_j} \dx \, , \\
\left(\frac{\partial \vec{K}}{\partial \matm_{\idxmat,l}}\right)_{i,j} & = (z_{\idxmat,l})^p   \int_{e_l} \Qrot(\theta_{\idxmat,l}) \frac{\partial \tensel_\idxmat}{\partial \matm} (\matm_{\idxmat,l})\symgrad{\vec{\phi}_i} : \symgrad{\vec{\phi}_j} \dx \, ,\\
\left(\frac{\partial \vec{K}}{\partial \theta_{\idxmat,l}}\right)_{i,j} & =  (z_{\idxmat,l})^p \int_{e_l}\frac{\partial\Qrot}{\partial\theta}(\theta_{\idxmat,l}) \tensel_\idxmat(\matm_{\idxmat,l})  \symgrad{\vec{\phi}_i} : \symgrad{\vec{\phi}_j} \dx \, .
\end{split}
\end{equation}
Finally we devise the following fixed-point scheme to update the variables $z_{\idxmat,l}$ and $\matm_{\idxmat,l}$, where $\epsilon$ and $\delta$ are small positive constants:
\begin{subequations} \label{eqn:FEM_nf_fixedpoint}
\begin{align}
\label{eqn:FEM_nf_fixedpoint:z}
z_{\idxmat,l}^{k+1} &=\max\left\{z_{min}, B_{\idxmat,l}^k z_{\idxmat,l}^{k} \right\}\, ,
\\
\label{eqn:FEM_nf_fixedpoint:m}
\matm_{\idxmat,l}^{k+1} & = \max\{ \underline{\matm}_\idxmat , \min\{ \overline{\matm}_\idxmat , \underline{\matm}_\idxmat - \delta + D_{\idxmat,l}^k (\matm_{\idxmat,l}^k - \underline{\matm}_\idxmat + \delta ) \} \}\, ,
\end{align}
\end{subequations}
where
\begin{equation}
\label{eqn:FEM_nf_B_D}
\begin{split}
B_{\idxmat,l}^k &= \frac{\left(\vec{U}^k\right)^T\frac{\partial \vec{K}}{\partial z_{\idxmat,l}} \vec{U}^k +\epsilon}{\Lambda^k |e_l| \rho_\idxmat(\matm_{\idxmat,l}^k) + \mu^k_l +\epsilon}  \, ,\\
D_{\idxmat,l}^k &= \frac{\left(\vec{U}^k\right)^T\frac{\partial \vec{K}}{\partial \matm_{\idxmat,l}} \vec{U}^k+\epsilon}{\Lambda^k|e_l| z_{\idxmat,l}^k \rhomprime (\matm_{\idxmat,l}^k) +\epsilon} 
\, ,
\end{split}
\end{equation}
with $\vec{U}^k$ denoting the state variable at the iteration $k$, and the Lagrange multipliers $\Lambda^k$ and $\vec{\mu}^k$ are chosen in such a way that conditions~\eqref{eqn:FEM_nf_cond1} and~\eqref{eqn:FEM_nf_cond2} hold true. We remark that analogous questions to \ref{quest:existence}-\ref{quest:uniqueness} are natural in this case too. We will deal with them in Section~\ref{sec:welldefinition}.

To update the variables $\theta_{\idxmat,l}$ a steepest descent with line-search algorithm can be applied, knowing that the gradient of the objective functional $\compliance$ is given by:
\begin{equation}
\frac{\partial \compliance}{\partial \theta_{\idxmat,l}} = - \vec{U}^T\frac{\partial \vec{K}}{\partial \theta_{\idxmat,l}}\vec{U} \, .
\end{equation}

\subsection{Filtered problems} \label{sec:algorithm_generalized_filt}

In this section we introduce the filtered version of problems~\eqref{eqn:pb_fem_base} and~\eqref{eqn:pb_fem_generalized}. Consider the following discrete counterpart of the filtering operator~\eqref{eqn:topopt_filtering_definition}:
\begin{equation}
\label{eqn:discrete_filtering_definition}
\hat{\psi}_l = \sum_{r=1}^{N_e} \hat{H}_{lr} \psi_r \, ,
\end{equation}
where the matrix $\hat{\vec{H}} \in \mathbb{R}^{N_e \times N_e}$ is given by:
\begin{equation*}
\hat{H}_{lr} 
=\frac{K(\vec{x}_l-\vec{x}_r) |e_r|}{\sum_{s=1}^{N_e} K(\vec{x}_l-\vec{x}_s) |e_s|}\dx \, .
\end{equation*}
The filtered version of the discrete problems~\eqref{eqn:pb_fem_base} and~\eqref{eqn:pb_fem_generalized} is obtained by replacing the original variables with the filtered ones in the definition of $\vec{K}$ and $M$, that is to say by replacing~\eqref{eqn:pb_fem_generalized:Kdef}--\eqref{eqn:pb_fem_generalized:Mdef} with:
\begin{subequations}
\begin{align}
\label{eqn:pb_fem_generalized_filt:Kdef}
K_{ij} &= \sum_{l=1}^{N_e } \sum_{\idxmat=1}^{N} (\hat{z}_{\idxmat,l})^p   \int_{e_l} \Qrot(\hat{\theta}_{\idxmat,l}) \tensel_\idxmat(\hat{\matm}_{\idxmat,l})\symgrad{\vec{\phi}_i} : \symgrad{\vec{\phi}_j} \dx \, , \\
\label{eqn:pb_fem_generalized_filt:Mdef}
\mass &= \sum_{\idxmat=1}^{N} \sum_{l=1}^{N_e} |e_l| \hat{z}_{\idxmat,l} \rho_\idxmat(\hat{\matm}_{\idxmat,l}) \, .
\end{align}
\end{subequations}
To derive an algorithm for the solution of the discrete filtered problems, it is possible to start from the schemes proposed in Sections~\ref{sec:algorithm_base_nf} and~\ref{sec:algorithm_generalized_nf} (which provide the sensitivities with respect to the filtered variables), and exploit the chain rule to compute the sensitivities with respect to the design variables. Omitting the passages, we report the final result for the general case. The update of the variables $z_{\idxmat,l}$ and $\matm_{\idxmat,l}$ is performed by the same fixed-point scheme~\eqref{eqn:FEM_nf_fixedpoint}, provided that the definitions of~\eqref{eqn:FEM_nf_B_D} are replaced by:
\begin{equation} \label{eqn:BD_filtered}
\begin{split}
B_{\idxmat,l}^k &= \frac{ \left(\vec{U}^k\right)^T\left(\sum_{r=1}^{N_e} \hat{H}_{rl}\frac{\partial \vec{K}}{\partial \hat{z}_{\idxmat,r}}\right) \vec{U}^k +\epsilon}{\Lambda^k \sum_{r=1}^{N_e} \hat{H}_{rl} |e_r| \rho_\idxmat(\hat{\matm}_{\idxmat,r}^k) + \mu^k_l +\epsilon}  \, ,\\
D_{\idxmat,l}^k &= \frac{ \left(\vec{U}^k\right)^T\left(\sum_{r=1}^{N_e} \hat{H}_{rl}\frac{\partial \vec{K}}{\partial \hat{\matm}_{\idxmat,r}}\right) \vec{U}^k +\epsilon}{\Lambda^k  \sum_{r=1}^{N_e} \hat{H}_{rl} |e_r| \rhomprime(\hat{\matm}_{\idxmat,r}^k) \hat{z}_{\idxmat,r}^k  +\epsilon} 	\, .
\end{split}
\end{equation}
Notice that, in this context, the derivatives of the stiffness tensor with respect to the filtered variables are obtained by replacing in Eq.~\eqref{eqn:FEM_nf_dE} the original variables with the filtered ones.
Also in this case one may wonder about existence and uniqueness of the multipliers $\Lambda^k$ and $\vec{\mu}^k$ (see \ref{quest:existence}-\ref{quest:uniqueness}). This issue will be treated in Section~\ref{sec:welldefinition}.

Since the variables associated with the orientation of the materials are periodic, a circular counterpart of the filtering operator~\eqref{eqn:topopt_filtering_definition} should be employed. Denoting by $T_\idxmat$ the period of the variable $\theta_\idxmat$ (for instance we have $T_\idxmat = 2 \pi$ for materials without any symmetry properties and $T_\idxmat = \pi$ for orthotropic materials), we consider the following filtering operator:
\begin{equation}
\hat{\theta}_{\idxmat,l}
= \frac{T_\idxmat}{2\pi} \operatorname{atan2}\left(
\sum_{r=1}^{N_e} \hat{H}_{lr}\sin\left(\frac{2\pi\theta_{\idxmat,r}}{T_\idxmat}\right) ,
\sum_{r=1}^{N_e} \hat{H}_{lr}\cos\left(\frac{2\pi\theta_{\idxmat,r}}{T_\idxmat}\right)
\right) \, ,
\end{equation}
where $\operatorname{atan2}$ denotes the variant of the arctangent function which accounts for the positioning of the angle in the appropriate quadrant:
\begin{equation*}
\operatorname{atan2}(y,x) =
\begin{cases}
\arctan(\frac y x) &\text{if } x > 0 \\
\arctan(\frac y x) + \pi &\text{if } x < 0 \text{ and } y \geq 0 \\
\arctan(\frac y x) - \pi &\text{if } x < 0 \text{ and } y < 0 \\
+\frac{\pi}{2} &\text{if } x = 0 \text{ and } y > 0 \\
-\frac{\pi}{2} &\text{if } x = 0 \text{ and } y < 0 \\
0 &\text{if } x = 0 \text{ and } y = 0 \, .
\end{cases}
\end{equation*}
Finally, by the chain rule the sensitivity of the compliance $\compliance$ with respect to each design variable $\theta_{\idxmat,l}$ is given by:
\begin{equation}
\begin{split}
\frac{\partial \compliance}{\partial \theta_{\idxmat,l}} 
=&- \cos\left(\frac{2\pi\theta_{\idxmat,l}}{T_\idxmat}\right) \sum_{s=1}^{N_e} \hat{H}_{sl} \frac{
	\left(\sum_r \hat{H}_{sr}\cos\left(\frac{2\pi\theta_{\idxmat,r}}{T_\idxmat}\right)\right)
	\vec{U}^T\frac{\partial \vec{K}}{\partial \theta_{\idxmat,s}}\vec{U}}
{\left(\sum_r \hat{H}_{sr}\sin\left(\frac{2\pi\theta_{\idxmat,r}}{T_\idxmat}\right)\right)^2 + 
	\left(\sum_r \hat{H}_{sr}\cos\left(\frac{2\pi\theta_{\idxmat,r}}{T_\idxmat}\right)\right)^2 } \\
& - \sin\left(\frac{2\pi\theta_{\idxmat,l}}{T_\idxmat}\right) \sum_{s=1}^{N_e} \hat{H}_{sl} \frac{
	\left(\sum_r \hat{H}_{sr}\sin\left(\frac{2\pi\theta_{\idxmat,r}}{T_\idxmat}\right)\right)
	\vec{U}^T\frac{\partial \vec{K}}{\partial \theta_{\idxmat,s}}\vec{U}}
{\left(\sum_r \hat{H}_{sr}\sin\left(\frac{2\pi\theta_{\idxmat,r}}{T_\idxmat}\right)\right)^2 + 
	\left(\sum_r \hat{H}_{sr}\cos\left(\frac{2\pi\theta_{\idxmat,r}}{T_\idxmat}\right)\right)^2 } \, .
\end{split}
\end{equation}

\subsection{Properties of the fixed-point algorithms} \label{sec:welldefinition}

In this section we show that the the fixed-point scheme introduced in the previous sections is compliant with the optimality conditions~\eqref{eqn:FEM_nf_optcond_syst2}, in a sense that is clarified by the following proposition. Notice that an analogous result can be proved in the same manner for the filtered version of the fixed point scheme (cf.~\eqref{eqn:BD_filtered}).

\begin{proposition}
	Consider a design $(z_{\idxmat,l},\matm_{\idxmat,l},\theta_{\idxmat,l})_{\idxmat,l}$ such that $\vec{U}^T\frac{\partial \vec{K}}{\partial \theta_{\idxmat,l}}\vec{U} = 0$. Then, it fulfils the optimality conditions~\eqref{eqn:FEM_nf_optcond_syst2} if and only if it is a fixed-point for the iterative scheme~\eqref{eqn:FEM_nf_fixedpoint}.
\end{proposition}
\begin{proof}
	Suppose that the design $(z_{\idxmat,l},\matm_{\idxmat,l},\theta_{\idxmat,l})_{\idxmat,l}$ fulfils~\eqref{eqn:FEM_nf_optcond_syst2}. Then it is easily seen that plugging $z_{\idxmat,l}^{k} = z_{\idxmat,l}^{k+1} = z_\idxmat$, $\matm_{\idxmat,l}^{k} = \matm_{\idxmat,l}^{k+1} = \matm_\idxmat$, $\Lambda^k=\Lambda$ and $\vec{\mu}^k=\vec{\mu}$ into~\eqref{eqn:FEM_nf_fixedpoint}\eqref{eqn:FEM_nf_cond1}\eqref{eqn:FEM_nf_cond1} all conditions are satisfied. By uniqueness of $\Lambda^k$ and $\vec{\mu}^k$ (see Proposition~\ref{prop:topopt_multimal_mL_exist_unique}), the design is a fixed point for the update scheme~\eqref{eqn:FEM_nf_fixedpoint}.
	
	On the other hand, suppose that a design $(z_{\idxmat,l}^k,\matm_{\idxmat,l}^k,\theta_{\idxmat,l}^k)_{\idxmat,l}$ is a fixed point for the scheme~\eqref{eqn:FEM_nf_fixedpoint}. Then it is easy to see that $z_{\idxmat,l}^k$, $\matm_{\idxmat,l}^k$, $\Lambda^k$, $\vec{\mu}^k$ fulfil the optimality conditions~\eqref{eqn:FEM_nf_optcond_syst2}.
\end{proof}

We now give an answer to the questions \ref{quest:existence} and \ref{quest:uniqueness}. The result is presented for the most general case, namely for the filtered version of the generalized problem, but analogous results for the unfiltered problem \eqref{eqn:FEM_nf_fixedpoint} and for the base problem \eqref{eqn:discrete_fixedpoint_base} can be obtained as particular cases. The proof of the proposition is reported in appendix~\ref{app:proof_topopt_multimal_mL_exist_unique}.

\begin{proposition}
	\label{prop:topopt_multimal_mL_exist_unique}
	Consider a provisional design $(z_{\idxmat,l}^k, \matm_{\idxmat,l}^k,\theta_{\idxmat,l}^k)$ fulfilling the constraints~\eqref{eqn:pb_fem_generalized:masslimit}--\eqref{eqn:pb_fem_generalized:mbounds}. Consider the update scheme~\eqref{eqn:FEM_nf_fixedpoint}, dependent on the parameters $\Lambda^k$ and $\vec{\mu}^k$ (which we denote simply by $\Lambda$ and $\vec{\mu}$), where $B_{\idxmat,l}^k$ and $D_{\idxmat,l}^k$ are given by~\eqref{eqn:BD_filtered}, and $\epsilon$ is a positive constant. Suppose that~\eqref{eqn:topopt_multimat_nontrivialparam} holds true, and that the specific weight is non decreasing in each $\matm_\idxmat$:
	\begin{equation} \label{eqn:densitynondecreasing}
	\rhomprime(\matm) \geq 0 \qquad \forall \underline{\matm}_\idxmat \leq \matm \leq \overline{\matm}_\idxmat, \text{ for } \idxmat = 1,\dots,N.
	\end{equation} 
	Then:
	\begin{enumerate}[label=(\alph*)]
		\item \label{point:topopt_multimat_proof_a} For each $\Lambda \geq 0$, in each mesh element $l=1,\dots,N_e$, the map $\mu_l \mapsto \sum_{\idxmat=1}^{N} z_{\idxmat,l}^{k+1}$ is continuous and non increasing.
		\item \label{point:topopt_multimat_proof_b} For each $\Lambda \geq 0$, there exists a unique $\vec{\mu}$, which we denote by $\vec{\mu}^{\Lambda}$, such that: 
		\begin{equation*}
		\mu_l \geq 0, \quad \sum_{\idxmat=1}^{N} z_{\idxmat,l}^{k+1} \leq 1,\quad\mu_l\left(\sum_{\idxmat=1}^{N} z_{\idxmat,l}^{k+1} - 1\right) = 0 \qquad \forall l=1,\dots,N_e\, .
		\end{equation*}
		\item \label{point:topopt_multimat_proof_c} With the choice $\vec{\mu} = \vec{\mu}^{\Lambda}$, the map $\Lambda \mapsto  \mass^{k+1}  = \sum_{\idxmat=1}^{N} \sum_{l=1}^{N_e} |e_l| \hat{z}_{\idxmat,l}^{k+1} \rho_\idxmat(\hat{\matm}_{\idxmat,l}^{k+1}) $ is continuous and non increasing.
		\item \label{point:topopt_multimat_proof_d} With the choice $\vec{\mu} = \vec{\mu}^{\Lambda}$, there exists a unique $\Lambda$ such that:
		\begin{equation*}
		\Lambda \geq 0, \quad \mass^{k+1} \leq \masslim,\quad \Lambda \left( \mass^{k+1} - \masslim \right) = 0 \, .
		\end{equation*}
	\end{enumerate}
\end{proposition}



\begin{remark}[Enhanced fixed point algorithm]\label{rem:enahanced}
	
The efficiency of the fixed-point scheme~\eqref{eqn:FEM_nf_fixedpoint} can be enhanced, as for the monomaterial case (see for instance~\cite{bendsoe:book}), by introducing the move limits $\zeta_z > 0$ and $\zeta_\matm > 0$ and a tuning parameter $\eta \in (0,1)$:
\begin{equation}
\label{eqn:FEM_nf_fixedpoint_enhanced}
\begin{split}
z_{\idxmat,l}^{k+1} &= \begin{cases}
\max\{ z_{\idxmat,l}^k - \zeta_z, z_{min}\} \qquad & \text{if $ (B_{\idxmat,l}^k)^\eta z_{\idxmat,l}^k \leq \max\{ z_{\idxmat,l}^k - \zeta_z, z_{min}\}  $} \\
z_{\idxmat,l}^k + \zeta_z \qquad& \text{if $ (B_{\idxmat,l}^k)^\eta z_{\idxmat,l}^k \geq z_{\idxmat,l}^k + \zeta_z $} \\
(B_{\idxmat,l}^k)^\eta z_{\idxmat,l}^k \qquad& \text{else,}
\end{cases}
\\
\matm_{\idxmat,l}^{k+1} &= \begin{cases}
\max\{ \matm_{\idxmat,l}^k - \zeta_\matm, \underline{\matm}_\idxmat\} \qquad &  \text{if $\tilde{\matm}_{\idxmat,l}^{k+1} \leq \max\{ \matm_{\idxmat,l}^k - \zeta_\matm, \underline{\matm}_\idxmat\}$} \\
\min\{ \matm_{\idxmat,l}^k + \zeta_\matm, \overline{\matm}_\idxmat\} \qquad &  \text{if $\tilde{\matm}_{\idxmat,l}^{k+1} \geq \min\{ \matm_{\idxmat,l}^k + \zeta_\matm, \overline{\matm}_\idxmat\}$} \\
\tilde{\matm}_{\idxmat,l}^{k+1} \qquad &  \text{else,}
\end{cases}
\end{split}
\end{equation}
where
\begin{equation*}
\tilde{\matm}_{\idxmat,l}^{k+1} = \underline{\matm}_\idxmat - \delta + (D_{\idxmat,l}^k)^\eta (\matm_{\idxmat,l}^k - \underline{\matm}_\idxmat + \delta )\, .
\end{equation*}
For this version of the algorithm a result analogous to Proposition~\ref{prop:topopt_multimal_mL_exist_unique} can be proved in the same manner; however, in this case, the uniqueness of the multipliers $\Lambda$ and $\vec{\mu}$ is lost, but the resulting updated design variables $z_{\idxmat,l}^{k+1}$ and $\matm_{\idxmat,l}^{k+1}$ are still unique.

\end{remark}
\subsection{The algorithm}

We notice that Proposition~\ref{prop:topopt_multimal_mL_exist_unique} suggests that the value for the multipliers $\Lambda^k$ and $\vec{\mu}^k$ can be found by means of a bisection loop. Therefore we devise Algorithm~\ref{alg:topopt_multimat}: at each iteration an external bisection loop detects the value of the multiplier $\Lambda^k$, and in each point of the domain an inner iteration loop determines the local value of the multiplier $\vec{\mu}^k$. A good initial guess for the Lagrange multipliers is given by the corresponding value at the previous iteration. Algorithm~\ref{alg:topopt_multimat} is written in a rather general way, since the strategy to follow in the iteration loops to find the value of the multipliers is not specified. A good strategy comprises two phases: in a first bracketing phase an interval containing the desired value is detected, and then the exact value is computed (with the desired accuracy) by bisection of the interval. An example of bisection loop can be found in \cite{bruggi2018topology}.

\begin{algorithm}
	\caption{Multi-material Topology Optimization}
	\label{alg:topopt_multimat}
	\begin{algorithmic}
		\For{$\idxmat\gets 1,\dots, N \, , \, l \gets 1,\dots,N_e$}
		\State Initialize $z_{\idxmat,l}$ and $\matm_{\idxmat,l}$
		\EndFor		
		\While{\textit{not converged}}
		\State Solve equilibrium (FEM) \Comment{see Eq.~\eqref{eqn:discrete_FEM_system}}
		\State Compute sensitivities \Comment{see Eq.~\eqref{eqn:FEM_nf_dE}}
		\State Initial guess for $\Lambda$
		\Loop		
		\For{$\idxmat\gets 1,\dots, N$}
		\State Update $\matm_{\idxmat,l}$ \Comment{see Eq.~\eqref{eqn:FEM_nf_fixedpoint_enhanced}}
		\EndFor
		\For{$l \gets 1,\dots,N_e$}
		\State Initial guess for $\mu_l$ 
		\Loop
		\For{$\idxmat\gets 1,\dots, N$}
		\State Update $z_{\idxmat,l}$ \Comment{see Eq.~\eqref{eqn:FEM_nf_fixedpoint_enhanced}}
		\EndFor
		\State $Z_l \gets \sum_{\idxmat=1}^{N} z_{\idxmat,l}$
		\If{($Z_l < 1$ \textbf{and} $\mu_l=0$) \textbf{or} $|Z_l-1|<\epsilon_{tol}$ } \State \textbf{break loop}
		\ElsIf{$Z_l < 1$} \State Increase $\mu_l$
		\ElsIf{$Z_l > 1$} \State Decrease $\mu_l$ (not below $0$)
		\EndIf
		\EndLoop
		\EndFor
		\State Update $\mass$  \Comment{see Eq.~\eqref{eqn:pb_fem_generalized_filt:Mdef}}
		\If{($\mass < \masslim$ \textbf{and} $\Lambda=0$) \textbf{or} $|\mass-\masslim|<\epsilon_{tol}$} \State \textbf{break loop}			
		\ElsIf{$\mass > \masslim$}  \State Increase $\Lambda$
		\ElsIf{$\mass < \masslim$} \State Decrease $\Lambda$ (not below $0$)
		\EndIf						
		\EndLoop
		\EndWhile
		\State Post-processing
	\end{algorithmic}
\end{algorithm}

\section{A leading application: TopOpt of self-assembling materials} \label{sec:applicationtodiblockcopolymers}

In this section we consider one of the applications which motivate this work. In particular we show how the problem of controlling the phase separation of self-assembling diblock copolymers to build bodies with optimal elastic properties can be formulated in the form of problem~\eqref{eqn:pb_ct_generalized}. We remark that in this work we consider the 2D case, but the presented methodology can be straightfully generalized to the 3D case. 

\subsection{Diblock copolymers self-assembly} \label{sec:diblock_selfassembly}

A diblock copolymer is a linear chain molecule made of two subchains, covalentely joined to each other, consisting respectively on $N_A$ monomers of type $A$ and of $N_B$ monomers of type $B$. A density functional theory for the phase separation of diblock compolymers, first proposed by Ohta and Kawasaki~\cite{ohta1986equilibrium} (see~\cite{choksi2003derivation} for further details), foresees that copolymer melts rearrange in such a way they minimize a suitable energy functional.

Let $\Omega$ be a domain in $\mathbb{R}^d$. We consider the order parameter $\phi(\vec{x}) = \phi_A(\vec{x}) - \phi_B(\vec{x}) \in [-1,1]$, defined as the difference between the mass fractions of the two monomers, denoted respectively by $\phi_A$ and $\phi_B$ (we have $\phi(\vec{x})=-1$ in pure $B$ zones and $\phi(\vec{x})=1$ in pure $A$ zones). Moreover, we define the parameter $\CHm = (N_A - N_B) / (N_A+N_B)$, which measures the mutual proportion of the two monomers inside each chain. Notice that we have $\CHm = \fint_{\Omega} \phi(\vec{x}) \dx$. Given the previous definitions, the Ohta-Kawasaki energy functional, defined over the set $\{\phi \in H^1(\Omega) \quad \text{s.t. } \fint_{\Omega} \phi(\vec{x}) \dx = \CHm\}$ reads as follows:	
\begin{equation} \label{eqn:ohta-kawasaki}
\mathcal{H}(\phi) = \int_{\Omega} \Bigl( \frac{1}{2\gamma^2} |\nabla \phi|^2 + F(\phi)\Bigr) \dx + \frac{1}{2} \int_{\Omega\times\Omega} G(\vec{x},\vec{y}) (\phi(\vec{x})-\CHm)(\phi(\vec{y})-\CHm) \diff{\vec{x}} \diff{\vec{y}} \, ,
\end{equation}
where $G$ is the Green's function of $-\laplacian$ with periodic or Neumann homogeneous boundary conditions, $F(s)$ is a double-well potential, the most common choice being $F(s) = \frac{1}{4}(1-s^2)^2$, and $\gamma$ is a positive parameter. In the first integral, the first term accounts for surface tension effects, while the second penalizes intermediate densities favouring pure phases: the overall effect is the formation of large $A$-rich and $B$-rich subdomains. On the other hand, the second integral, which models the local interactions among subchains, favours rapid oscillations of $\phi$. As a result of the competition of the two integrals, segregation takes place on the mesoscale, where regular periodic patterns can be observed.

The phase separation process of diblock copolymers can be modelled by writing the gradient flow of the Ohta-Kawasaki functional~\eqref{eqn:ohta-kawasaki} in the space $H^{-1}$, defined as the dual space of the zero-mean ${H}^1(\Omega)$ functions, i.e. $\mathring{H}^1(\Omega) = \left\{u \in {H}^1(\Omega) \quad \text{s.t. } (u,1)_{L^2(\Omega)} = 0\right\}$ (see e.g.~\cite{choksi2009_3d}).
This leads to the so-called Cahn-Hilliard-Oono equation:
\begin{equation} \label{eqn:CHO}
\begin{cases}
\partial_t \phi = \laplacian \mu - (\phi -\CHm) \quad & \text{in }\Omega_T := \Omega \times [0,T] \\
\mu = -\gamma^{-2} \laplacian \phi + F'(\phi)  & \text{in }\Omega_T \\
\partial_{\vec{\nu}} \phi = \partial_{\vec{\nu}} \mu = 0 & \text {in } \partial \Omega \times [0,T] \quad \text{\textit{or}} \quad \phi,\mu \quad \text{$\Omega$-periodic}\\
\phi = \phi_0 & \text {in }  \Omega \times \{t=0\}\, .
\end{cases}
\end{equation}
By numerical approximation of Eq.~\eqref{eqn:CHO}, it is possible to obtain the minimizers of the Ohta-Kawasaki functional for a given value of $\gamma$ and $\CHm$. A throughout analysis of the phase plane of diblock copolymers in 2D can be found in~\cite{choksi2011_2d}. To our purposes, i.e. predicting the pattern according to the control variable $\CHm$, the fundamental observation is that for a given value of $\gamma$ (which is fixed once the kind of monomers are chosen, and thus cannot be regarded as a control parameter) we have four regions (see Figure~\ref{fig:CHm_dependence}):
\begin{itemize}
	\item $\CHm \in [-1, -\CHm_2) \cup (\CHm_2, 1]$: disorder;
	\item $\CHm \in (-\CHm_2,-\CHm_1)$: spots of material $A$ inside material $B$ (denoted in the following by $A$-spots);
	\item $\CHm \in (-\CHm_1,\CHm_1)$: stripes;
	\item $\CHm \in (\CHm_1,\CHm_2)$: spots of material $B$ inside material $A$ ($B$-spots).
\end{itemize}
Since we are interested in a regular structure, we will restrict our interest to spots and stripes configurations. In the following we will consider the case when $\gamma=20$, and the corresponding thresholds $\CHm_1=0.2$, $\CHm_2=0.6$.


\subsection{Homogenization of periodic microstructures} \label{sec:applicationtodiblockcopolymers:homo}

The length-scale of diblock copolymers patterns is typically comprised between 1 and 50 nanometres, some orders of magnitude smaller than the geometrical length scale even for microscale structural applications. Therefore, a direct simulation of the whole domain is infeasible and an up-scaling strategy is thus required. Homogenization theory allows to derive macroscopic effective properties of microscopically heterogeneous media~\cite{allaire:book}. In the framework of elasticity, given a microscopic structure of known elastic properties, it aims at determining the so-called homogenized tensor, i.e. a stiffness tensor which models the macroscopic behaviour of the microstructure. As such, homogenization theory provides a theoretical framework for the mathematical modelling of composites materials, which are obtained by mixing at a very fine scale different materials.

This operation is performed by embedding the original two-scales problem in a family of problems with increasing separation between the two scales. Each problem is parametrized by the microscopic length-scale $\epsilon$, and the limit for $\epsilon \rightarrow 0$ is taken into account. The case of periodic microstructures can be treated by means of a two-scales asymptotic expansion, which provides an explicit formula for the homogenized tensor~\cite{allaire:book}. The most general theory in homogenization is based on the notion of \textit{H}-convergence, introduced by Spagnolo under the name of $G$-convergence~\cite{spagnolo1968convergenza,spagnolo1976convergence}, and later generalized by Tartar and Murat~\cite{francfort1986homogenization,murat1997hconvergence}. This defines an adequate topology for the notion of convergence of problems as $\epsilon$ goes to zero.

We now go back to diblock copolymers melts, and address the problem of finding a homogenized tensor capable of describing the medium at a macroscopic level. As mentioned in Section~\ref{sec:diblock_selfassembly}, when a diblock copolymer solution undergoes a critical temperature, phase separation takes place, ending up with an ordered structure consisting of regions rich of material $A$ or material $B$. Equilibrium configurations can be recovered as minima of the Ohta-Kawasaki functional~\eqref{eqn:ohta-kawasaki}, and are described by the order parameter $\phi(\vec{x}) \in [-1,1]$, defined as the difference between the relative density of the two monomers. Thus we have $\phi=-1$ in pure $B$ zones and $\phi=1$ in pure $A$ zones. 

We suppose that a controller has faculty to choose the composition of the solution in the different points of the domain, and thus we introduce the control variable $\CHm \in L^{\infty}(\Omega;[\underline{\CHm},\overline{\CHm}])$, where $\underline{\CHm} = - \CHm_2$ and $\overline{\CHm} = \CHm_2$. The variable $\CHm$ determines the microstructure in each point of the domain. Where $\CHm \in (-\CHm_1,\CHm_1)$, for instance, the body exhibits stripes, while where $\CHm \in (\CHm_1,\CHm_2)$ $B$-spots are present. We notice that, since the parameter $\gamma$ is set once the couple of monomers $A$ and $B$ are chosen, $\gamma$ does not account as a control variable.

Denoting by $\epsilon$ the length-scale of the microstructure, we consider the following two-scales formulation for the order parameter:
\begin{equation*}
\phi^\epsilon(\vec{x}) = \phi \big(\vec{x},\frac{\vec{x}}{\epsilon}\big)\, ,
\end{equation*}
where for each $\vec{x} \in \Omega$ the function $\vec{y} \mapsto \phi(\vec{x},\vec{y})$ is the periodic microstructure attaining the minimum of the Ohta-Kawasaki functional for $\CHm=\CHm(\vec{x})$. 
We remark that the periodicity cell, which we denote by $Y_\CHm \subset \mathbb{R}^d$, may depend on the interval which $\CHm(\vec{x})$ belongs to: in the interval $(-\CHm_1,\CHm_1)$ we have rectangular elementary cell, while in the regions $(\CHm_1,\CHm_2)$ and $(-\CHm_2,-\CHm_1)$ the elementary cell has hexagonal shape (see Figure~\ref{fig:DB_2D_periodic_tiling}). We remark that, when different choices of elementary cell are available, the resulting homogenized tensor does not depend on the particular choice, since the associated periodic function $\vec{y} \mapsto \phi(\vec{x},\vec{y})$ is not affected by the choice.


\begin{figure}
	\centering
	\subfloat[][] {\includegraphics[width=.2\textwidth]{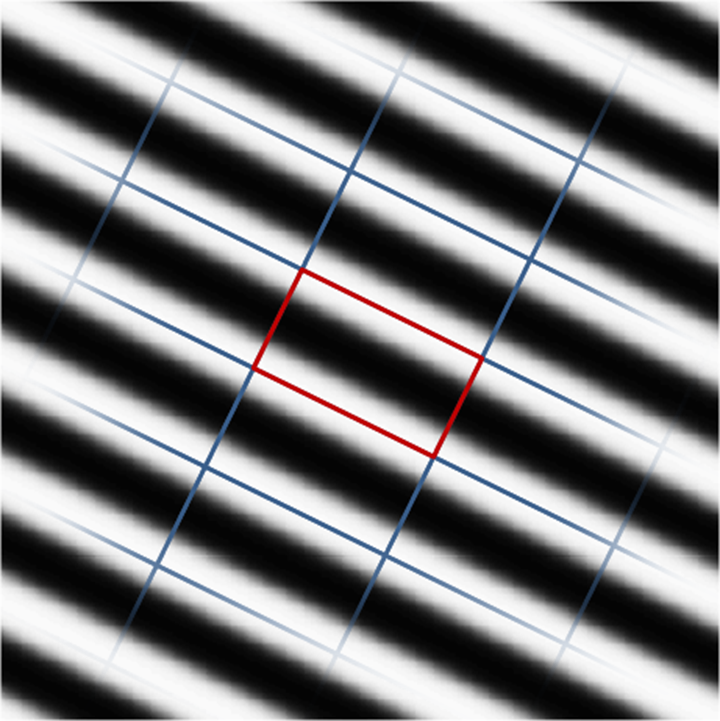}\label{fig:DB_2D_lamellae_grid_phade}} \quad
	\subfloat[][] {\includegraphics[width=.2\textwidth]{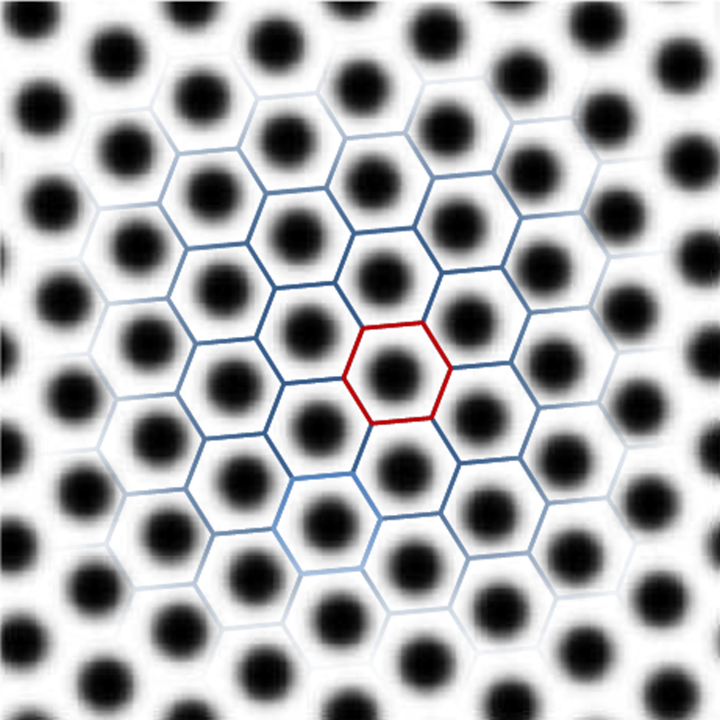}\label{fig:DB_2D_spots_hex_phade}}\quad
	\caption{Periodicity cells for diblock copolymers tilings in 2D.}
	\label{fig:DB_2D_periodic_tiling}
\end{figure}

We suppose that pure materials $A$ and $B$ can be modelled as linearly elastic media, with Lam\'{e} moduli $(\mu^A,\lambda^A)$ and $(\mu^B,\lambda^B)$ respectively. In the region of intermediate composition, we suppose that the solid behaves as a linearly elastic isotropic medium, with Lam\'{e} moduli which are a convex combination of those of pure materials. In other words, in each point $\vec{x} \in \Omega$, we suppose:
\begin{equation}
\tensel^\epsilon(\vec{x}) 
= 2\left( \frac{\mu^A + \mu^B}{2} + \frac{\mu^A - \mu^B}{2} \phi^\epsilon(\vec{x}) \right) \, \Ifour 
+  \left( \frac{\lambda^A + \lambda^B}{2} + \frac{\lambda^A - \lambda^B}{2} \phi^\epsilon(\vec{x})\right) \, \Itwo \otimes \Itwo\, ,
\end{equation}
where $\Itwo$ and $\Ifour$ are the second-order and fourth-order identity tensors respectively. 

Although the pure phases are modelled as isotropic media, the composite material is not isotropic in general: the properties of the material may depend on its orientation. For this reason the model has to take into account the orientation of the microstructure. To do so, we introduce another control variable, $\theta \in L^{\infty}(\Omega;[0,2 \pi))$, which describes the local orientation of the diblock copolymer pattern.

Therefore, we introduce the rotated periodicity cell (where $\vec{R}$ is defined in~\eqref{eqn:tensorQ_tensorR}):
\begin{equation*}
Y_{\CHm,\theta} = \vec{R}(\theta) Y_\CHm = \Big\{ \vec{y} \in \mathbb{R}^d \quad \text{s.t. } \vec{R}(\theta)^{-1} \vec{y} \in Y_\CHm\Big\} \, .
\end{equation*} 
We introduce the space of candidate minima for the Ohta-Kawasaki functional for the monomers proportion $\CHm$ and orientation $\theta$, defined as follows,
\begin{equation*}
V_{\CHm,\theta} := \Big\{ \phi \in H^1(Y_{\CHm,\theta}) \,\text{, } \text{$Y_{\CHm,\theta}$-periodic, s.t } \fint_{Y_{\CHm,\theta}} \phi = \CHm \Big\}\, ,
\end{equation*}
and the associated Ohta-Kawasaki functional
\begin{equation*}
\mathcal{H}_{\CHm,\theta}(\phi) 
:= \int_{Y_{\CHm,\theta}} \Bigl( \frac{1}{2 \gamma^2} |\nabla \phi|^2 + F(\phi)\Bigr) \diff{\vec{y}} 
+ \frac{1}{2} \int_{Y_{\CHm,\theta}}\int_{Y_{\CHm,\theta}} G(\vec{y}_1,\vec{y}_2) (\phi(\vec{y}_1)-\CHm)(\phi(\vec{y}_2)-\CHm) \diff{\vec{y}_1} \diff{\vec{y}_2}\, .
\end{equation*}
Therefore, for each $\vec{x} \in \Omega$ we define the function $\phi(\vec{x}, \cdot)$ as follows:
\begin{equation}
\label{eqn:twoscales_phi_argmin}
\phi(\vec{x}, \cdot) = \underset{\psi \in V_{\CHm(\vec{x}),\theta(\vec{x})}}{\text{argmin}} \, \mathcal{H}_{\CHm(\vec{x}),\theta(\vec{x})}(\psi) \, .
\end{equation}
To sum up, the elastic properties of the medium are modelled by the following two-scales model:
\begin{equation}
\label{eqn:twoscales_model}
\begin{split}
\tensel^\epsilon(\vec{x}) 
&= 2\left( \frac{\mu^A + \mu^B}{2} + \frac{\mu^A - \mu^B}{2} \phi(\vec{x})^\epsilon \right) \, \Ifour 
+ \left( \frac{\lambda^A + \lambda^B}{2} + \frac{\lambda^A - \lambda^B}{2} \phi(\vec{x})^\epsilon\right) \, \Itwo \otimes \Itwo \, ,\\
\phi^\epsilon(\vec{x}) &= \phi \left(\vec{x},\frac{\vec{x}}{\epsilon}\right) \, ,\\
\phi(\vec{x}, \cdot) &= \underset{\psi \in V_{\CHm(\vec{x}),\theta(\vec{x})}}{\text{argmin}} \, \mathcal{H}_{\CHm(\vec{x}),\theta(\vec{x})}(\psi)\, .
\end{split}
\end{equation}
Since we have written the tensor in the form $\tensel^\epsilon(\vec{x}) = \tensel\left(\vec{x},\frac{\vec{x}}{\epsilon}\right)$, then the homogenized tensor $\tensel^*$ is given by the homogenization formula (see~\cite{allaire:book}):
\begin{equation}
\label{eqn:twoscales_Estar}
\tensel^*_{ijkl}(\vec{x}) = \frac{1}{|Y_{\CHm(\vec{x}),\theta(\vec{x})}|}\int_{Y_{\CHm(\vec{x}),\theta(\vec{x})}} \tensel(\vec{x},\vec{y}) 
\big( e^{ij} + \symgradsubscript{\vec{y}}{\vec{w}^{ij}} \big) :
\big( e^{kl} + \symgradsubscript{\vec{y}}{\vec{w}^{kl}} \big) \diff{\vec{y}} \, ,
\end{equation}
where the function $\vec{y} \mapsto \vec{w}^{ij}(\vec{x},\vec{y})$ solves the cell problem:
\begin{equation}
\label{eqn:twoscales_cellpb_rotated}
\begin{cases}
-\diverg{_{\vec{y}}} \Bigl( \tensel(\vec{x},\vec{y}) \symgradsubscript{\vec{y}}{\vec{w}^{ij}(\vec{x},\vec{y})} \Bigr)=
\diverg{_{\vec{y}}} \Bigl( \tensel(\vec{x},\vec{y}) e^{ij} \Bigr) & \text{in } Y_{\CHm(\vec{x}),\theta(\vec{x})}\\
\vec{y} \mapsto \vec{w}^{ij} (\vec{x},\vec{y}) & \text{$Y_{\CHm(\vec{x}),\theta(\vec{x})}$-periodic,}
\end{cases}
\end{equation}
and, denoting by $\{e_i\}_{1\leq i \leq d}$ the canonical basis of $\mathbb{R}^d$, we define the following basis for the space of symmetric second order tensors: 
\begin{equation}
e^{ij} = \frac{1}{2} \bigl( e_i \otimes e_j + e_j \otimes e_i \bigr) \, .
\end{equation}
The expression~\eqref{eqn:twoscales_Estar} may be cumbersome in actual implementations, since it requires the solution of the cell problem~\eqref{eqn:twoscales_cellpb_rotated} in each point of the domain; however, Eq.~\eqref{eqn:twoscales_Estar} can be restated in an equivalent and simpler form. First, we notice that $\tensel^*(\vec{x})$ depends on $\vec{x}$ only through the control variables $\CHm$ and $\theta$, and thus we can write $\tensel^*(\vec{x}) = \tensel^*_{\CHm(\vec{x}),\theta(\vec{x})}$, where $\tensel^*_{\CHm,\theta}$ denotes the homogenized stiffness tensor associated with the monomers proportion $\CHm$ and the microstructure orientation $\theta$. Moreover, the effect of $\theta$ is simply a rotation of the local frame of reference. Then, instead of computing the homogenized tensor on the rotated domain $Y_{\CHm,\theta}$, it is more convenient to compute the homogenized tensor on the non-rotated domain $Y_\CHm$, and then to rotate the tensor. As a matter of fact, through a change of variables $\vec{z} = \vec{R}(\theta)^{-1} \vec{y}$, it can be proved that 
\begin{equation}
\label{eqn:homo_tensor_rotation}
\tensel^*_{\CHm,\theta} = \vec{Q}(\theta) \tensel^*_{\CHm,0}\, .
\end{equation}
The task of computing $\tensel^*(\vec{x})$ thus reduces to that of computing $\tensel^*_\CHm := \tensel^*_{\CHm,0}$ for $\CHm = \CHm(\vec{x})$. 
Thus, by means of homogenization theory we have that the homogenized tensor $\tensel^*$ is given by the following:
\begin{equation}
\label{eqn:twoscales_model_smart}
\begin{split}
\tensel^*(\vec{x}) &= \vec{Q}(\theta(\vec{x})) \tensel^*_{\CHm(\vec{x})} \\
(\tensel^*_\CHm)_{ijkl} &= \frac{1}{|Y_\CHm|}\int_{Y_\CHm} \tensel_\CHm(\vec{y}) 
\big( e^{ij} + \symgrad{\vec{w}_\CHm^{ij}} \big) :
\big( e^{kl} + \symgrad{\vec{w}_\CHm^{kl}} \big) \diff{\vec{y}} \\
\tensel_\CHm(\vec{y}) &= 2 \left( \frac{\mu^A + \mu^B}{2} + \frac{\mu^A - \mu^B}{2} \phi_\CHm(\vec{y}) \right)\, \Ifour 
+ \left( \frac{\lambda^A + \lambda^B}{2} + \frac{\lambda^A - \lambda^B}{2} \phi_\CHm(\vec{y})\right) \, \Itwo \otimes \Itwo\\ 
\phi_\CHm &= \underset{\psi \in V_{\CHm,0}}{\text{argmin}} \, \mathcal{H}_{\CHm,0}(\psi)\, ,
\end{split}
\end{equation}
where the function $\vec{w}^{ij}_\CHm(\vec{y})$ solves the cell problem:
\begin{equation}
\label{eqn:twoscales_cellpb}
\begin{cases}
-\diverg{} \Bigl( \tensel_\CHm(\vec{y}) \symgrad{\vec{w}_\CHm^{ij}(\vec{y})} \Bigr)=
\diverg{} \Bigl( \tensel_\CHm(\vec{y}) e^{ij} \Bigr) & \text{in } Y_\CHm\\
\vec{w}_\CHm^{ij} \quad \text{$Y_\CHm$-periodic} \, . &
\end{cases}
\end{equation}
Since it would be computationally unaffordable, the homogenized tensor $\tensel^*(\vec{x})$ cannot be computed for each value of $\vec{x}$ (or, better, for each element of the mesh).
To overcome this latter obstruction, we solve Eq.~\eqref{eqn:twoscales_cellpb} by means of the Finite Element method for some values of $\CHm$, and we build a database of homogenized tensors. 
Then, the homogenized tensors for remaining values $\CHm$ are approximated by interpolating the computed values. Figure~\ref{fig:homogenized_entries} shows the results of this process for two different ratios of stiffness of pure phases $E^A/E^B$. Remarkably, the homogenized tensor of both spots configurations turns out to be isotropic. For this reason, for these classes of materials, we do not report the graph of the dependence of the homogenized tensor on $\theta$. 


\begin{figure}[p]
	\centering
	\subfloat[][$E^A/E^B=10$]  {\includegraphics[width=.5\textwidth,page = 1,trim={1.65cm 1.7cm 1.25cm 0},clip]{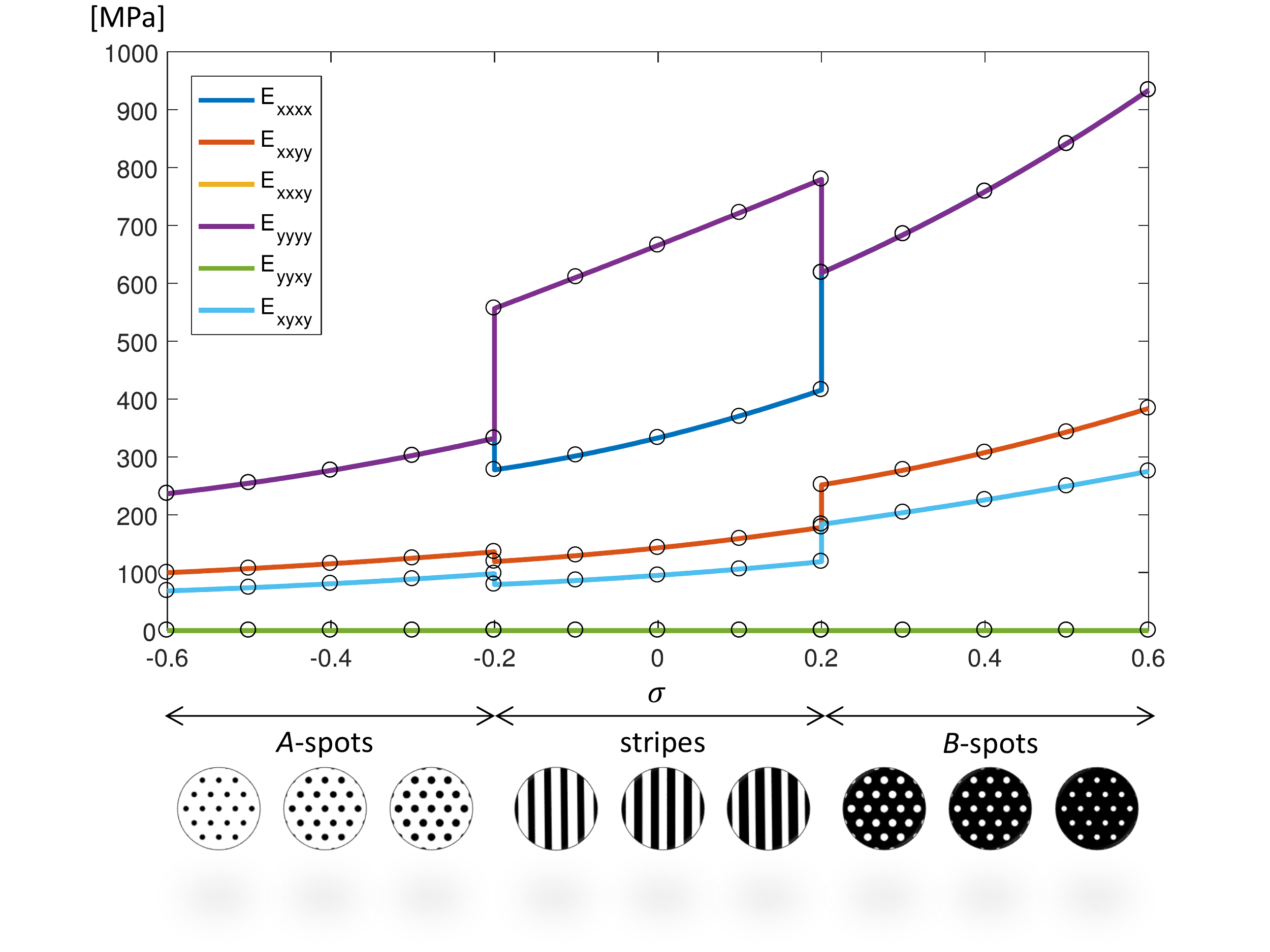}} \hfill
	\subfloat[][$E^A/E^B=1000$]{\includegraphics[width=.5\textwidth,page = 2,trim={1.65cm 1.7cm 1.25cm 0},clip]{homogenized_tensor.pdf}} \\ 
	\subfloat[][$E^A/E^B=10$]  {\includegraphics[width=.5\textwidth,page = 1,trim={0 2.75cm 0 0},clip]{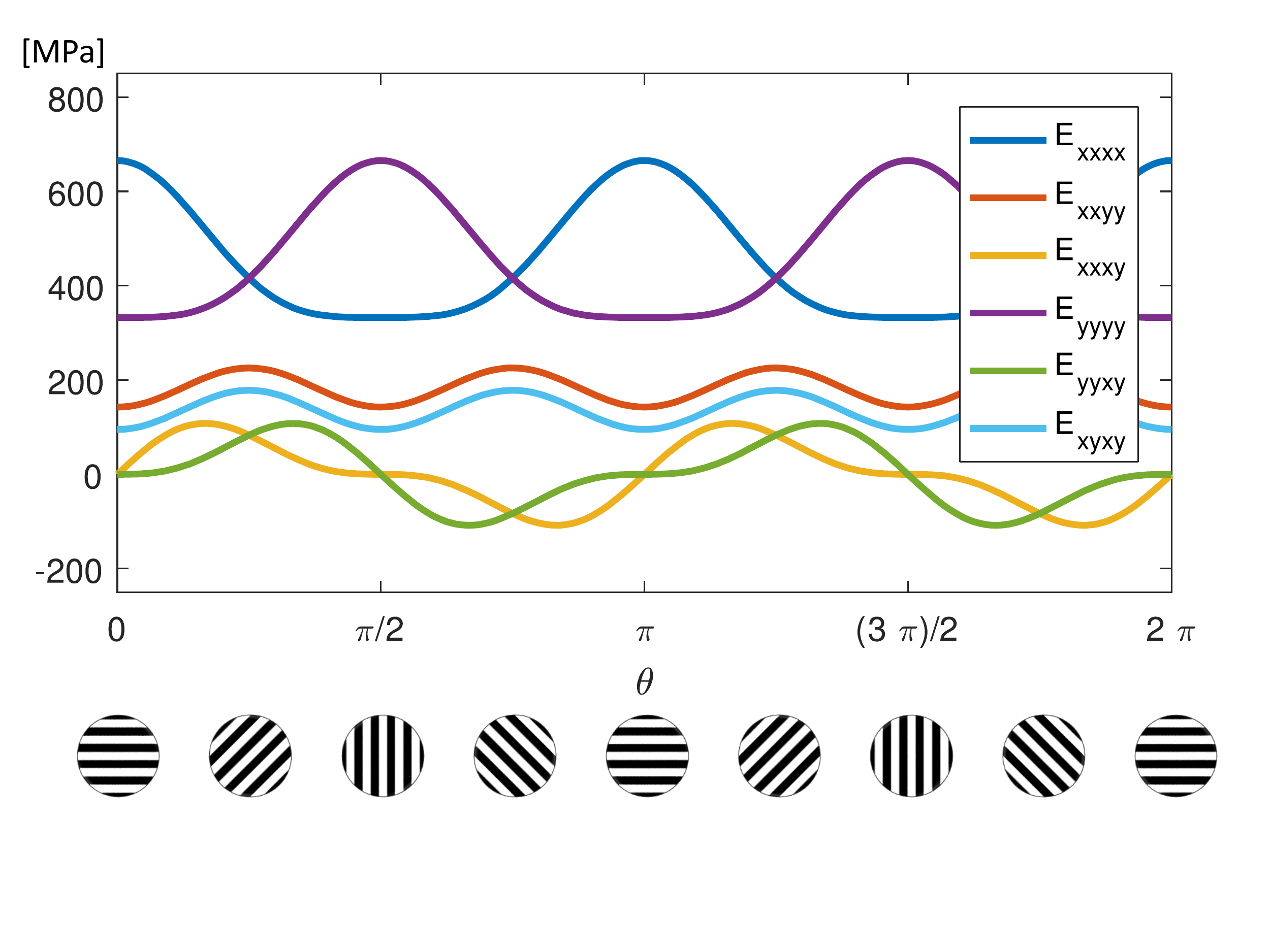}} \hfill
	\subfloat[][$E^A/E^B=1000$]{\includegraphics[width=.5\textwidth,page = 2,trim={0 2.75cm 0 0},clip]{rotated_tensor_lamellae.pdf}} 
	\caption{Entries of the homogenized stiffness tensor for different $E^A/E^B$ ratios (reported in the tagline), in dependence of $\CHm$ (first row) and on $\theta$ (second row). First row: plots of the application $\CHm \mapsto \tensel^*_\CHm$. Circles represent the values computed numerically, while lines are the three-points polynomial interpolators, computed by interpolating the extreme end middle points of each interval. Second row: plots of the application $\theta \mapsto \vec{Q}(\theta) \tensel^*_\CHm$ for $\CHm = 0$, corresponding to a stripes configuration. The equivalent plot for spots configurations is not reported, since such microstructures are isotropic at the macroscale.}
	\label{fig:homogenized_entries}
\end{figure}

\subsection{TopOpt problem formulation}

We finally show how the problem of finding the optimal distribution of diblock copolymers patterns inside a reference domain can be formulated in the form of problem~\eqref{eqn:pb_ct_generalized}. In 2D we have $N=3$ classes of materials ($\idxmat=1$ is associated to $A$-spots, $\idxmat=2$ with stripes and $\idxmat=3$ with $B$-spots). 
In this case each variable $\matm_\idxmat$ can be interpreted as the physical variable $\CHm$ associated to the $\idxmat$-th material. In the optimal configuration obtained by the algorithm the optimal value of $\CHm$ will be given by:
\begin{equation} \label{eqn:CHm_convexcombination}
\CHm(\vec{x}) = \sum_{\idxmat=1}^{N} z_\idxmat(\vec{x}) \matm_\idxmat(\vec{x}) \, .
\end{equation} 
Therefore, in pure material regions (where for any $\vec{x}$ we have $z_{\hat{\idxmat}}(\vec{x}) =1$ for some $\hat{\idxmat}$, and $z_\idxmat(\vec{x}) = 0$ for any $\idxmat \neq \hat{\idxmat}$), we will have $\CHm(\vec{x}) = \matm_{\hat{\idxmat}}(\vec{x})$.

Thus, the lower and upper bounds $\underline{\matm}_\idxmat$ and $\overline{\matm}_\idxmat$ are given by (see Section~\ref{sec:diblock_selfassembly}):
\begin{equation*}
\underline{\matm}_1 = -\CHm_2, \quad
\overline{\matm}_1 = \underline{\matm}_2 = -\CHm_1, \quad
\overline{\matm}_2 = \underline{\matm}_3 = \CHm_1, \quad
\overline{\matm}_3 = \CHm_2.
\end{equation*}
The stiffness tensor of each $\idxmat = 1,\dots,N$ is given by $\tensel_\idxmat(\matm_\idxmat) = \tensel^*_{\matm_\idxmat}$ (see Eq.~\eqref{eqn:twoscales_model_smart}).

We are left with the characterization of the specific weights $\rho_\idxmat(\matm)$. By denoting by $\rho^A$ and $\rho^B$ the specific weights of pure $A$ and $B$ phases respectively, we suppose that:
\begin{equation}
\rho_\idxmat(\matm) = \frac{\rho^A+\rho^B}{2} + \frac{\rho^A-\rho^B}{2} \matm \qquad \forall\, \idxmat = 1,\dots,N\, .
\end{equation}

\section{Numerical results} \label{sec:results}

\begin{figure}
	\centering
	\subfloat[][Test \textit{square}]{\includegraphics[width=.33\textwidth, trim={7cm 5cm 7cm 4.5cm},clip,page=1]{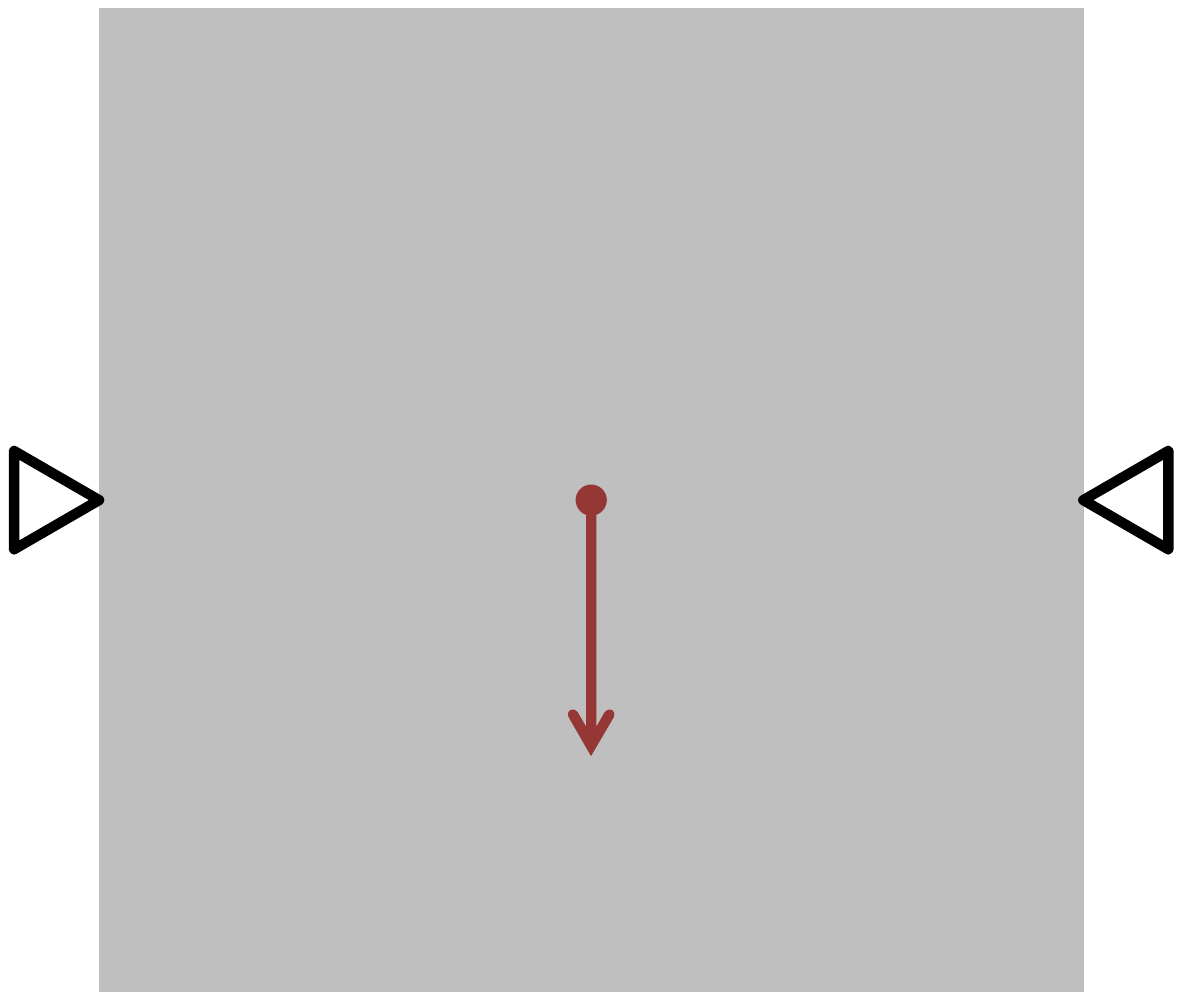}\label{fig:model_square}}
	\subfloat[][Test \textit{MBB beam}]{\includegraphics[width=.33\textwidth, trim={7cm 5cm 7cm 4.5cm},clip,page=2]{model_ALL.pdf}\label{fig:model_overbridge}}
	\subfloat[][Test \textit{L-shape}]{\includegraphics[width=.33\textwidth, trim={7cm 5cm 7cm 4.5cm},clip,page=3]{model_ALL.pdf}\label{fig:model_Lshape}}
	\caption{Geometry and boundary conditions of the the case studies considered in this paper.}	
	\label{fig:models}
\end{figure}

In this section we show the results of some numerical test of a 2D implementation of Algorithm~\ref{alg:topopt_multimat}. In Section~\ref{sec:results:basepb} the basic problem~\eqref{eqn:pb_fem_base} is addressed, while in Section~\ref{sec:results:generalized} the algorithm for the generalized problem~\eqref{eqn:pb_fem_generalized} is applied to the optimization of self-assembling diblock copolymers. Before discussing the numerical results we give some additional details about the implementation.

Piecewise linear (P1) finite elements defined over a triangular mesh are employed for the discretization of the state variable $\vec{u}$, while piecewise constant (P0) elements are used for the design variables $z_\idxmat$, $\matm_\idxmat$ and $\theta_\idxmat$. The enhanced version of the fixed-point algorithm presented in Remark~\ref{rem:enahanced} is implemented, with the following parameters: $z_{min} = \delta = 10^{-3}$, $\zeta_z = \zeta_\matm = 0.05$, $\eta = \frac{1}{2}$. In the post-processing stage, variables $z_\idxmat$ are rounded to the closest value among 0 and 1. 
Mesh adaptivity techniques could be employed to decrease the computational burden on large scale problems, while increasing the accuracy in the evaluation of the objective function (see e.g. \cite{maute1995adaptive,bruggi2011fully}).

\begin{remark}[Dealing with local minima]
	It is well known that TopOpt problems are highly non-convex and may feature many local minima, thus small variations in the initial design may result in drastic changes in the result of the optimization algorithm~\cite{sigmund:checkerboards,bendsoe:book}. The most common strategies to deal with the problem of local minima rely on \textit{continuation} methods: the original problem is replaced by a modified one, featuring better properties, which is gradually changed to recover the original formulation. In our implementation we set for the first 30 iterations $p=1$; then, in the following 30 iterations, we gradually raise its value until the target value $p=3$ is reached.
\end{remark}

Figure~\ref{fig:models} shows the geometries of four different test cases, popular in the context of TopOpt. The grey region indicates the design domain $\Omega$, red arrows denote loads, while the other symbols (point clamps and linear clamps) indicate where homogeneous Dirichlet conditions are set. In the remaining part of the boundary, homogeneous Neumann conditions (stress free conditions) are set.

\subsection{Basic problem: discrete orientation optimization} \label{sec:results:basepb}

Consider first the basic problem~\eqref{eqn:pb_fem_base}, which consists in finding the optimal distribution of matter inside a domain, by choosing among a set of $N$ candidate materials. To perform numerical tests, we consider a set of $N$ orthotropic materials with different principal directions. We denote by $\alpha_\idxmat$ the principal stiffness direction of the material $\idxmat$, and by $\Theta$ the set of candidate directions $\Theta = (\alpha_1, \dots, \alpha_N)$. For simplicity we consider the case when all candidate materials can be obtained by rotation of a single reference material with stiffness tensor $\tensel^0$, with the following non-null entries ($\tensel^0$ is the homogenized tensor of a material made of equally spaced horizontal stripes of two isotropic materials of Poisson's coefficient $\nu = 0.3$ and Young moduli $10^2$ and $10^3$, respectively):
\begin{equation*}
\tensel^0_{xxxx} = 665.5 , \,
\tensel^0_{yyyy} = 332.8 , \,
\tensel^0_{xxyy} = 142.6 , \,
\tensel^0_{xyxy} = 95.2 \, .
\end{equation*}
Therefore, the stiffness tensor of the material $\idxmat$ reads as:
\begin{equation*}
\tensel_\idxmat = \Qrot(\alpha_\idxmat) \tensel^0 \, .
\end{equation*}
Notice that, in this setting, problem~\eqref{eqn:pb_fem_base} can be interpreted as a discrete orientation optimization of a single anisotropic material.

The results of the tests considered in this section are presented through two types of graphs. The first one shows through different colours (defined in the circular colourmap) the principal stiffness direction of the material selected in each point of the domain. The region where void is selected is represented in white. The second graph shows a visual rendering of the optimal structure, where the anisotropic materials are depicted by stripes oriented as the principal stiffness direction. 

In Figures~\ref{fig:PBbase_Lshape}--\ref{fig:PBbase_overbridge} two popular test cases, corresponding to cases (c) and (b) in Fig.~\ref{fig:models}, are considered. As usual in TopOpt, the solutions are truss structures (see e.g.~\cite{bendsoe:book}); moreover, in most cases, the material with an orientation similar to the beam turns out to be the most convenient, and, in turn, beams tend to align themselves with one of the available direction.


\begin{figure}
	\centering
	\includegraphics[width=.32\textwidth]{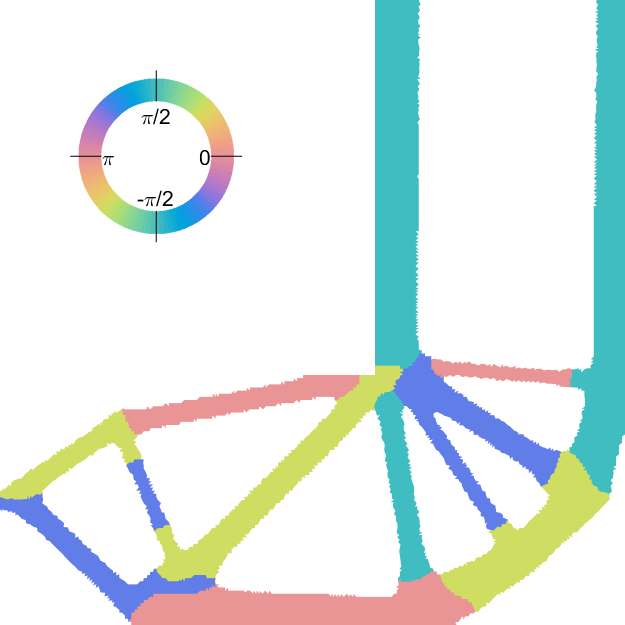} 
	\hspace{.2cm}
	\includegraphics[width=.32\textwidth]{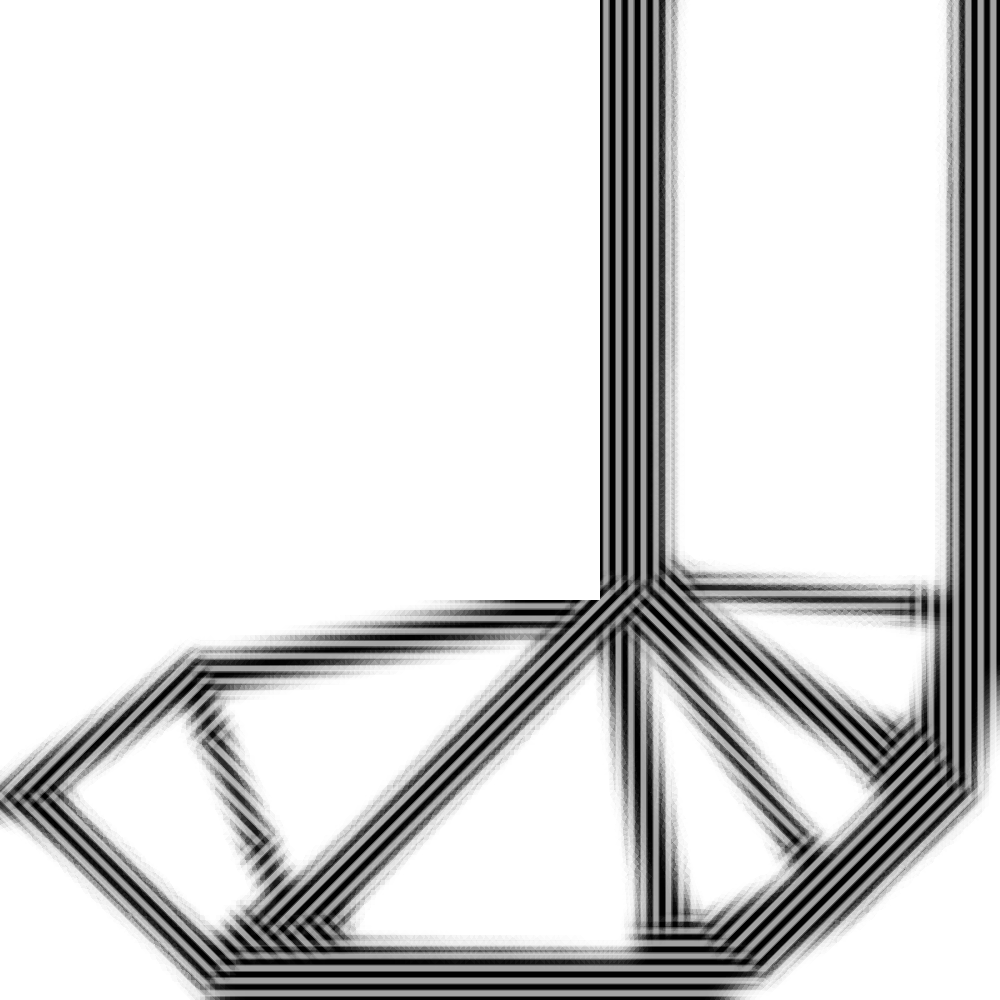} 
	\caption{Numerical results of the basic problem~\eqref{eqn:pb_fem_base} for the test \textit{L-shape} with $\Theta = (0, \frac{1}{4}\pi, \frac{1}{2}\pi, \frac{3}{4}\pi)$. Domain size $8 \times 8$, filtering radius $0.15$.}
	\label{fig:PBbase_Lshape}
\end{figure}

\begin{figure}
	\centering
	\includegraphics[width=.6\textwidth]{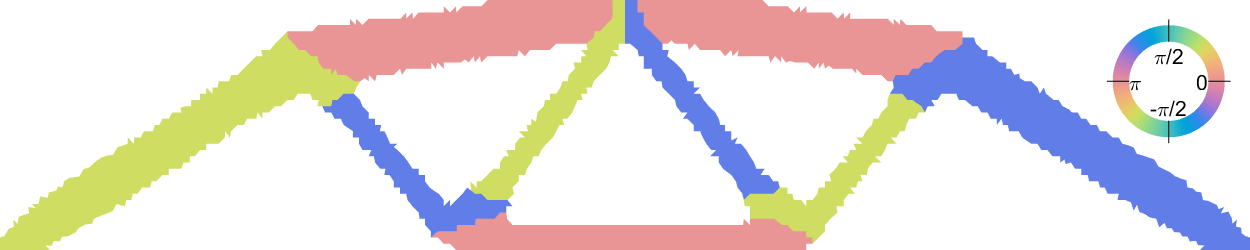} 
	\\	\vspace{.2cm}
	\includegraphics[width=.6\textwidth]{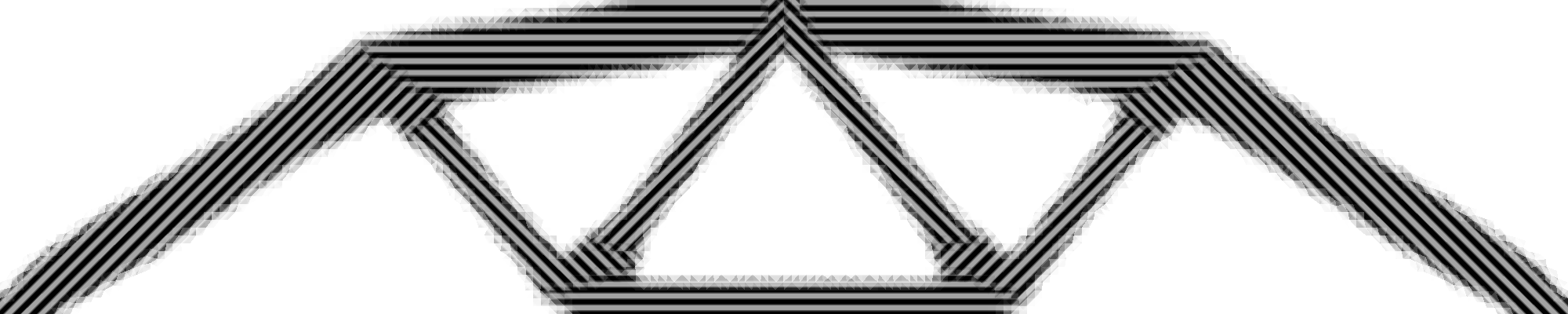} 
	\caption{Numerical results of the basic problem~\eqref{eqn:pb_fem_base} for the test \textit{MBB beam} with $\Theta = (0, \frac{1}{4}\pi, \frac{1}{2}\pi, \frac{3}{4}\pi)$. Domain size $20 \times 4$, filtering radius $0.15$.}
	\label{fig:PBbase_overbridge}
\end{figure}

In Figure~\ref{fig:PBbase_Square} the case study \textit{square} is considered. The solutions obtained with different sets $\Theta$ are compared. We notice that the topology and the shape of the solution is affected by the orientation of the available materials. In each case the optimal solution envisages two arches connecting the two clamps (located at the mid-points of the vertical sides of the domain), and some beams connecting the arches to the load point (located at the centre of the domain), but the curvature of the arches and the number of the beams depend on the set $\Theta$ of available orientations.

\begin{figure}
	\centering
	\includegraphics[width=.33\textwidth]{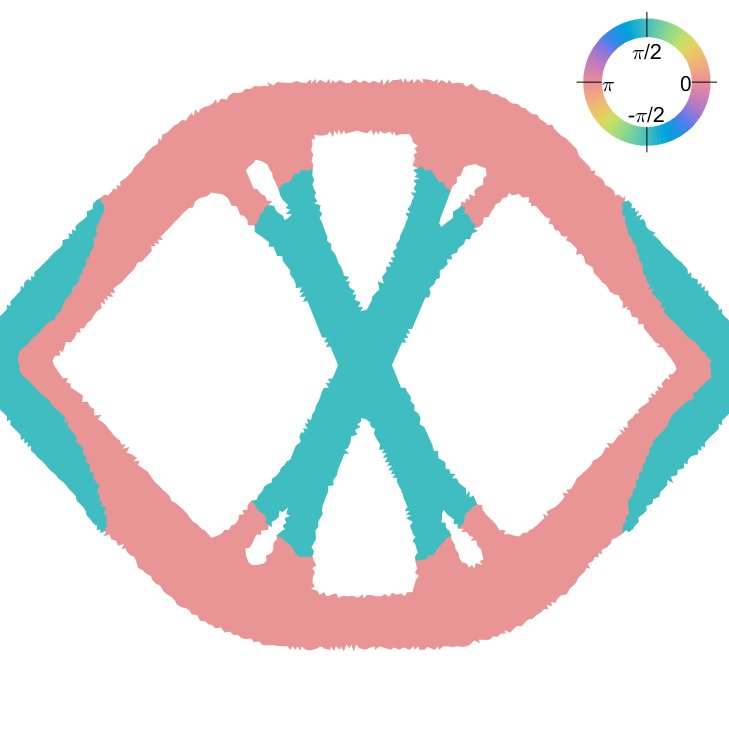} 
	\hspace{.2cm}
	\includegraphics[width=.33\textwidth]{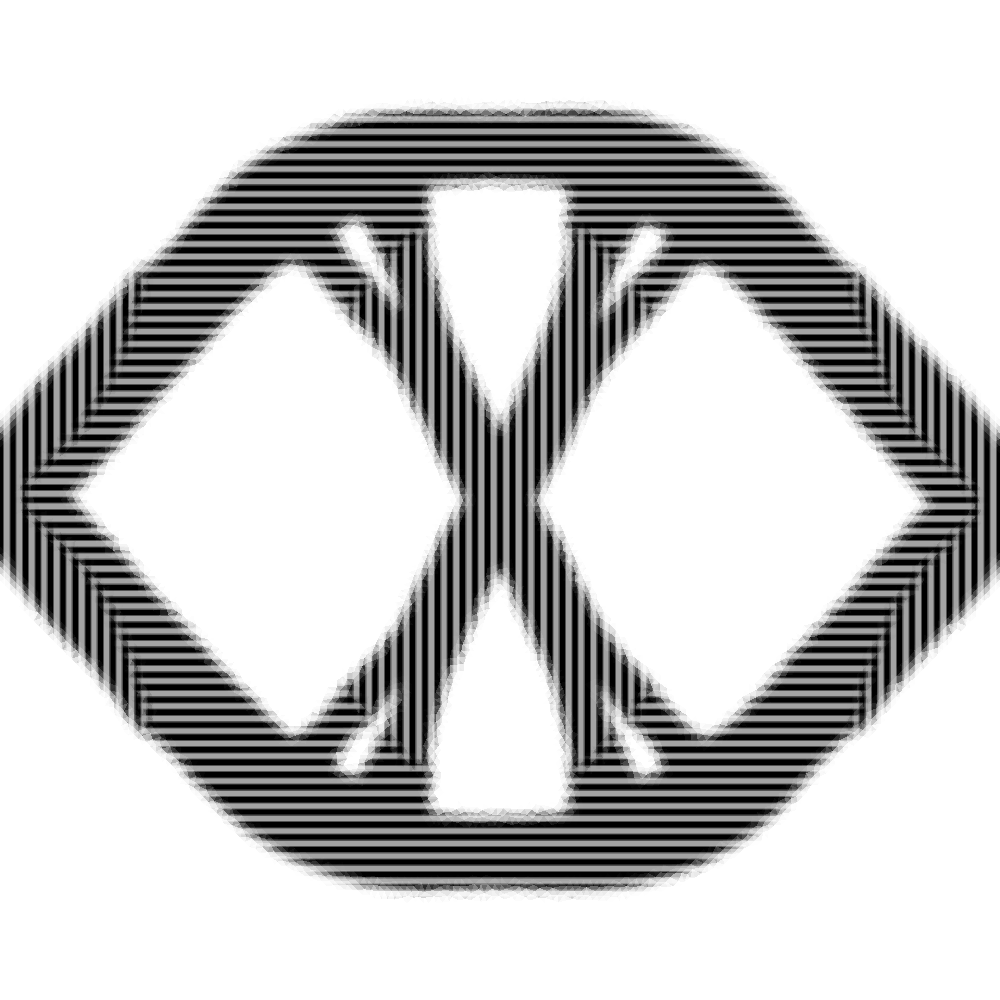} 
	\\	\vspace{-.25cm}
	\includegraphics[width=.33\textwidth]{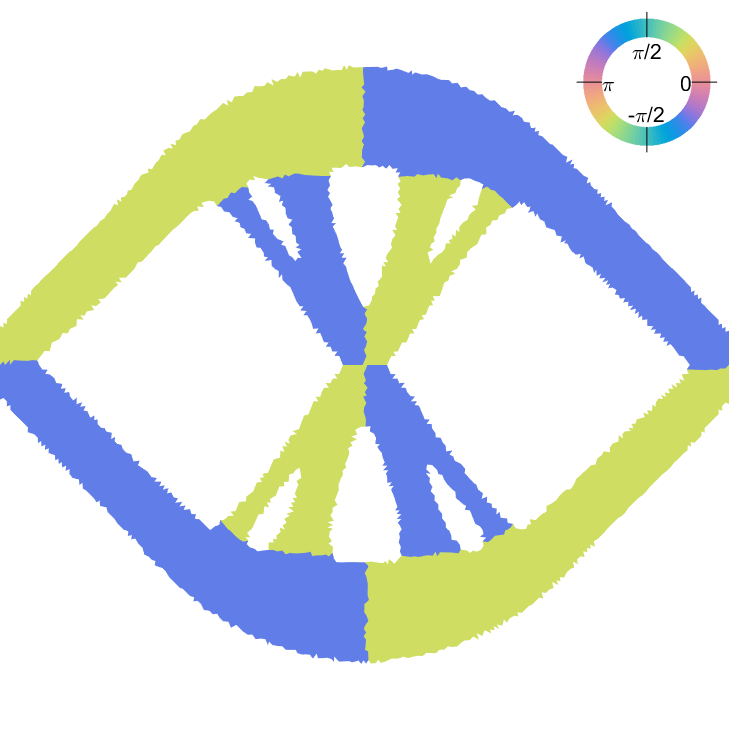} 
	\hspace{.2cm}
	\includegraphics[width=.33\textwidth]{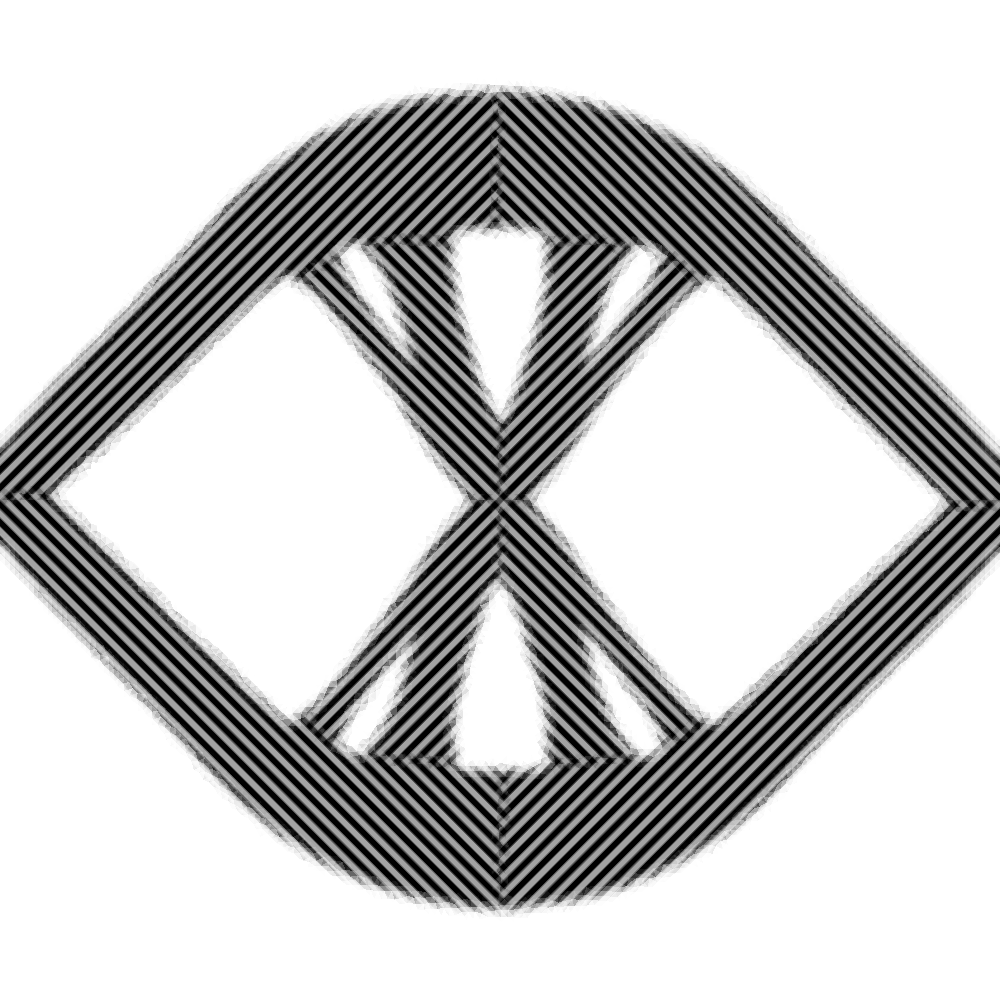} 
	\\	\vspace{-.25cm}
	\includegraphics[width=.33\textwidth]{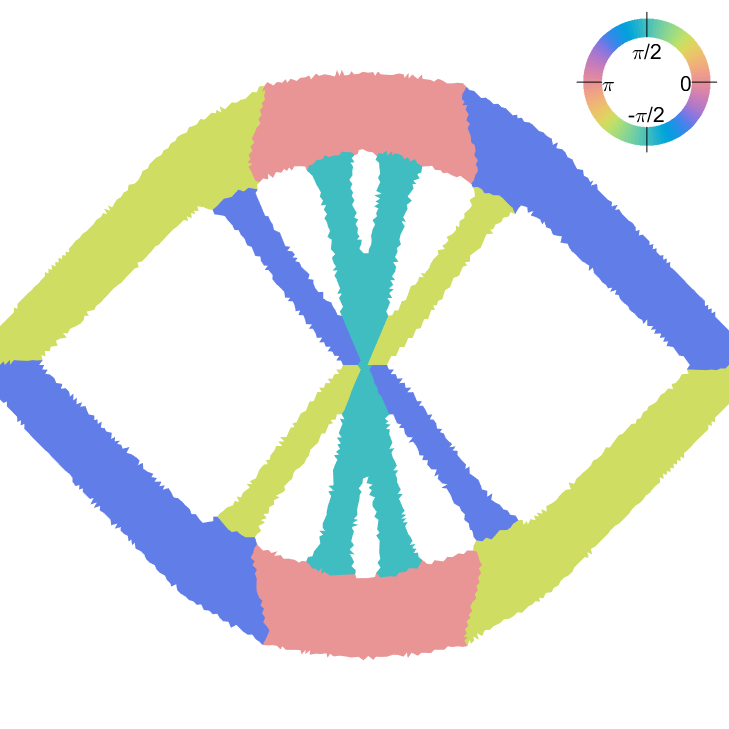} 
	\hspace{.2cm}
	\includegraphics[width=.33\textwidth]{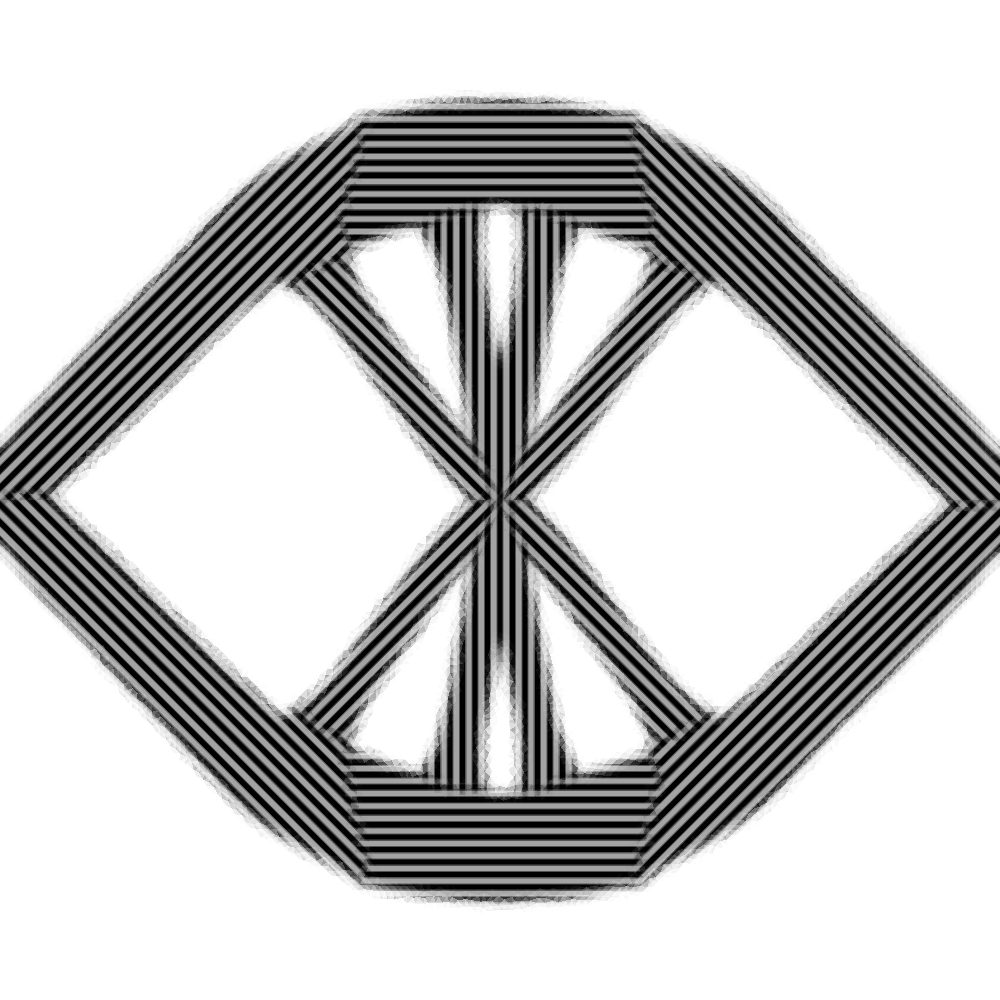} 
	\\	\vspace{-.25cm}
	\includegraphics[width=.33\textwidth]{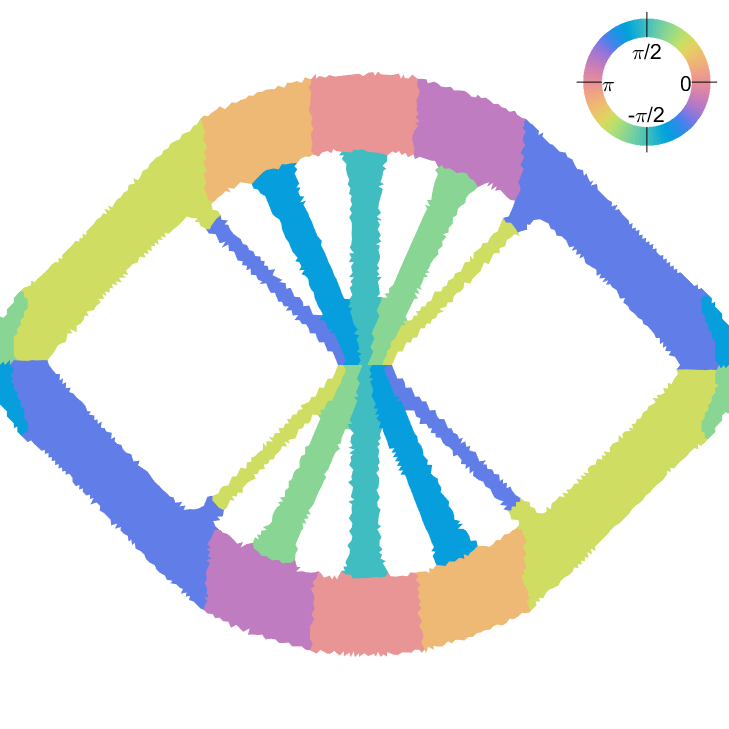} 
	\hspace{.2cm}
	\includegraphics[width=.33\textwidth]{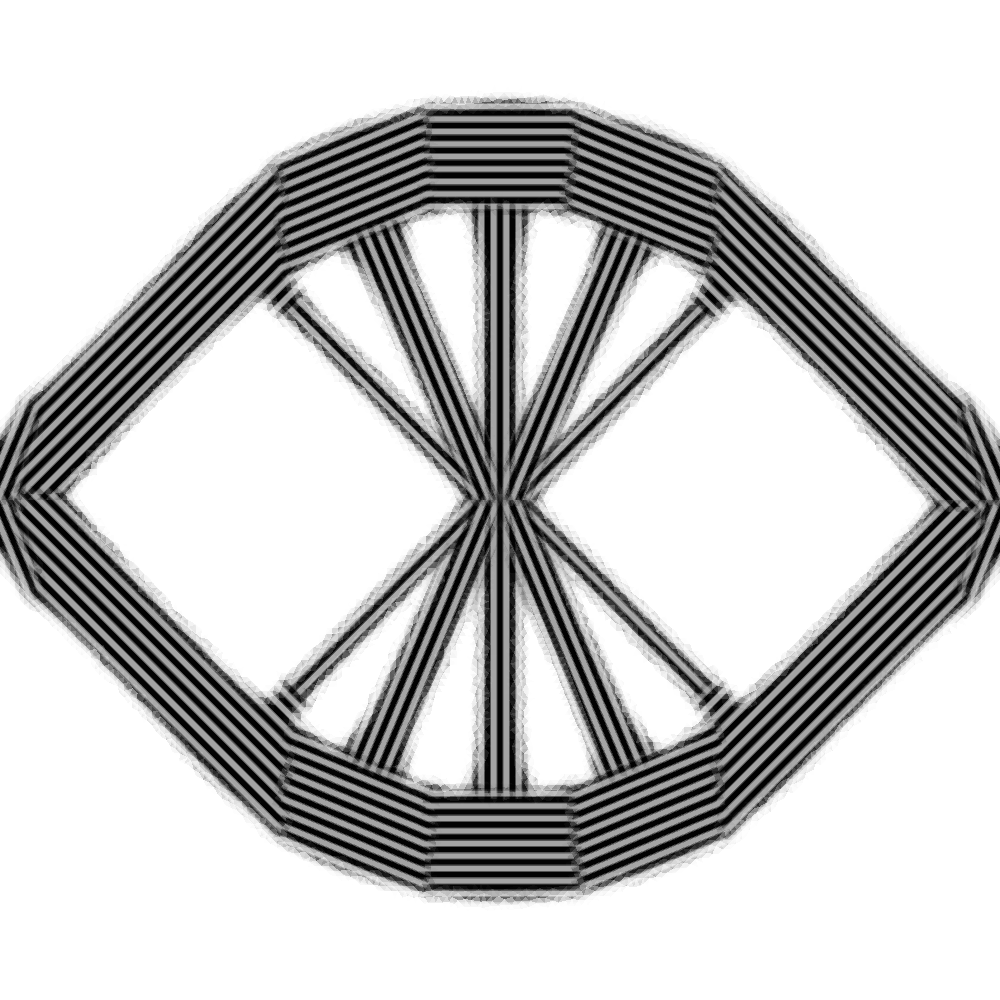} 
	\caption{Numerical results of the basic problem~\eqref{eqn:pb_fem_base} for the test \textit{square} with different sets of candidate materials: first row, $\Theta = (0, \frac{1}{2}\pi)$; second row, $\Theta = (\frac{1}{4}\pi, \frac{3}{4}\pi)$; third row, $\Theta = (0, \frac{1}{4}\pi, \frac{1}{2}\pi, \frac{3}{4}\pi)$; fourth row, $\Theta = (0, \frac{1}{8}\pi, \frac{1}{4}\pi, \frac{3}{8}\pi, \frac{1}{2}\pi, \frac{5}{8}\pi, \frac{3}{4}\pi, \frac{7}{8}\pi)$. Domain size $8 \times 8$, filtering radius $0.1$.}
	\label{fig:PBbase_Square}
\end{figure}

\subsection{Generalized problem: TopOpt of self-assembling materials} \label{sec:results:generalized}

In this section we consider the generalized problem~\eqref{eqn:pb_fem_generalized}, applied to the TopOpt of self-assembling diblock copolymers (see Section~\ref{sec:applicationtodiblockcopolymers}). In the tests we performed, the Lam\'{e} moduli of pure phases are given by
\begin{equation*}
\mu^*=\frac{E^*}{2(1+\nu^*)}, \qquad
\lambda^*=\frac{E^*\nu^*}{(1+\nu^*)(1-2\nu^*)} \, ,
\end{equation*}
where $*$ stands for either $A$ or $B$, and we set the Young moduli $E^A = 1000$, $E^B = 100$ or $E^B = 1$, and the Poisson coefficients $\nu^A=\nu^B=0.3$. In order to ensure a fair competition between materials, we set each time $\rho^A/\rho^B= E^A/E^B$.

Figures~\ref{fig:PBCHO_Overbridge} and \ref{fig:PBCHO_Square} show the results of the algorithm applied to test cases (b) and (a) respectively. The numerical results are presented through three types of graph: 
\begin{itemize}
	\item \textit{Allocation of microstructures}. This chart shows how the different kinds of microstructures ($A$-spots, lamellae, $B$-spots) are allocated in the design region. This chart is integrated with some enlargements revealing the micro-structures present in some specific points.
	\item \textit{Variable $\CHm$}. In this chart the optimal value of the variable $\CHm$, given by Eq.~\eqref{eqn:CHm_convexcombination} is shown. 
	\item \textit{Variable $\theta$}. This chart shows the variable $\theta_\idxmat$ in the regions where an anisotropic microstructure (i.e. stripes) is present. The optimal value is obtained as a weighted average, in a similar fashion to Eq.~\eqref{eqn:CHm_convexcombination}.
	Regions where an isotropic microstructure (i.e. spots configurations) is present are marked with grey. Since in orthotropic materials the orientation $\theta=\pi$ can be identified with $\theta=0$, a circular colourmap which identifies the over-mentioned angles is employed. 
\end{itemize}
\begin{figure}
	\centering
	\includegraphics[width=.6\textwidth]{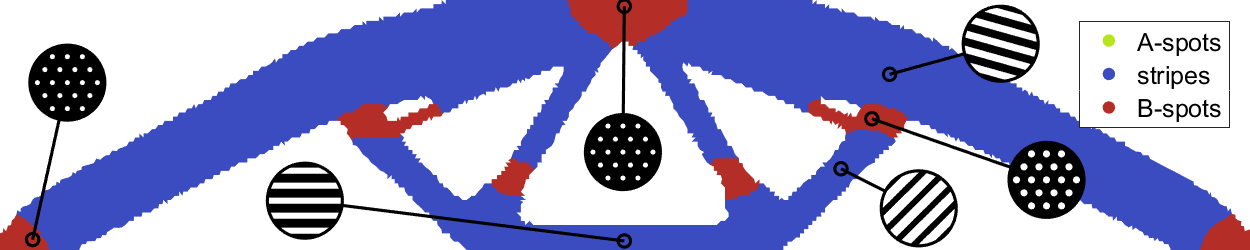}
	\\	\vspace{.2cm}
	\includegraphics[width=.6\textwidth]{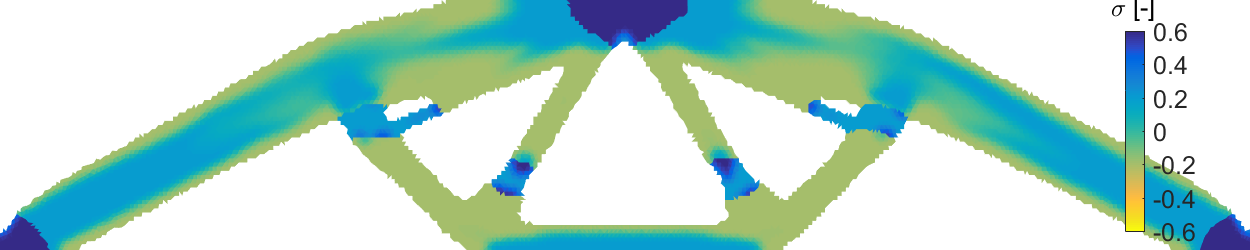}
	\\	\vspace{.2cm}
	\includegraphics[width=.6\textwidth]{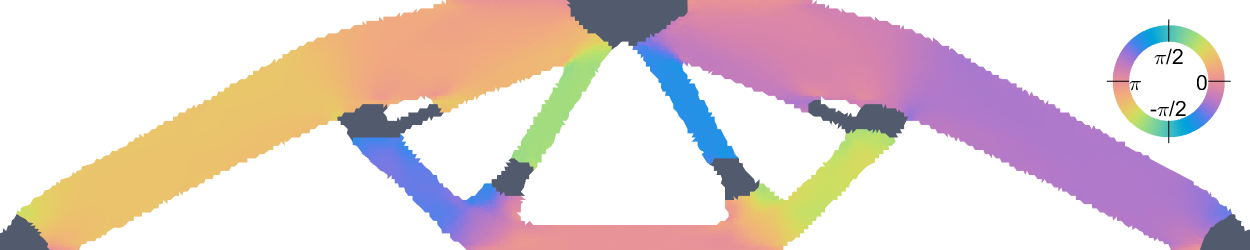}
	\caption{Numerical results of the generalized problem~\eqref{eqn:pb_fem_generalized} for the test \textit{MBB beam}, with the following set of parameters: stiffness ratio $E^A/E^B=10$, domain size $20 \times 4$, filtering radius $0.15$. From top to bottom: allocation of microstructures, variable $\CHm$, variable $\theta$.}
	\label{fig:PBCHO_Overbridge}
\end{figure}
\begin{figure}
	\centering
	\includegraphics[width=.32\textwidth]{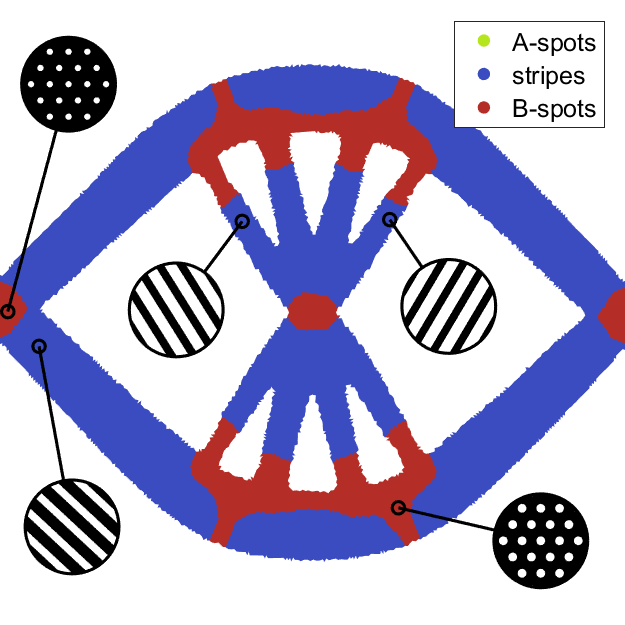}
	\includegraphics[width=.32\textwidth]{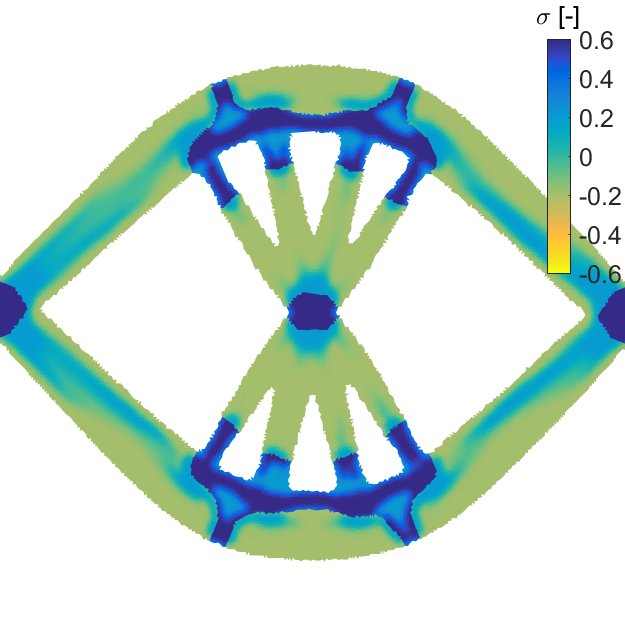}
	\includegraphics[width=.32\textwidth]{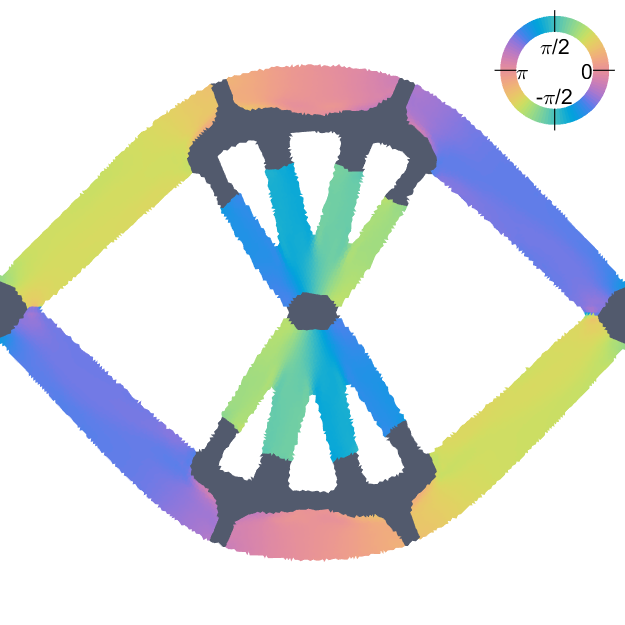}
	\caption{Numerical results of the generalized problem~\eqref{eqn:pb_fem_generalized} for the test \textit{square}, with the following set of parameters: stiffness ratio $E^A/E^B=10^3$, domain size $8 \times 8$, filtering radius $0.1$. From left to right: allocation of microstructures, variable $\CHm$, variable $\theta$.}
	\label{fig:PBCHO_Square}
\end{figure}
Once again, the solution is typically a truss structure. Moreover, inside beams, anisotropic stripes microstructures aligned with the beam itself result to be the most convenient; at the junction points instead, where there is not any clear preferential direction, isotropic spots configurations are chosen. 

We notice that inside regions with the same microstructure type (e.g. B-spots), the variable $\CHm$ may feature different optimal values. For instance, in Fig.~\ref{fig:PBCHO_Overbridge}, B-spots regions located at the junctions between beams report a lower value of $\CHm$ (i.e. lower stiffness) compared to the B-spots regions located near the application point of force and near the supports, where the structure has to bear a more demanding load. Indeed, from the enlargements in the first graph of Fig.~\ref{fig:PBCHO_Overbridge}, one can see that the size of the spots of material $B$ is smaller in this latter case. Similar comments apply for the stripes regions, where the relative size of $A$-stripes with respect to $B$-stripes can vary, as is is evident from the enlargements.


Remarkably, in all the test case we performed, the final design featured the presence of stripes and $B$-spots, but never of $A$-spots.
To give an heuristic explanation of this result, we notice that the optimal design is typically a truss structure, so for the sake of explanation we focus on a single beam and we consider the benchmark problem where one has to choose the optimal material for a single beam, subject to a traction load. In this case the solution can be found analytically. Indeed, consider a rectangular beam, oriented as the axis $x$, of length $L$ and cross-sectional area $A$ (notice that in 2D the area actually a length), fixed at one side ($x = 0$) and forced by a traction load $F \, \vec{e}_x$ at the other side ($x = L$). We suppose that the load is uniformly distributed on the area of application, so that the load per unit area is $\vec{t} = F/A \, \vec{e}_x$. If we suppose that the material is such that $\tensel_{xxxy} = \tensel_{yyxy} = 0$ (this is the case of orthotropic materials), it is easily shown that the equilibrium displacement in $x$ direction is $u_x(\vec{x}) = \frac{F}{A E} x$, where $E = \tensel_{xxxx} - \tensel_{xxyy}/\tensel_{yyyy}$. Since total mass of the beam is given by $\mass = \rho LA$, the area attaining the mass constraint $\mass = \masslim$ is $A = \masslim / (\rho L)$. The compliance is thus given by 
\begin{equation*}
\compliance = A \vec{u}|_{x=L} \cdot \vec{t} = \frac{F^2 L}{A E} = \frac{F^2 L^2}{\masslim} \left(\frac{E}{\rho}\right)^{-1}\, .
\end{equation*}
Therefore, the compliance is minimized when the $E/\rho$ ratio is maximal.

This example suggests that, due to the mass constraint, the quantity to be maximized is not the stiffness $E$, but rather the $E/\rho$ ratio. Incidentally, this observation allows to give an explanation about how the SIMP method works. Indeed, for the standard monomaterial SIMP method, denoting by $z$ the dimensionless density, we have $E = z^p E_0$ and $\rho = z\,\rho_0$. Therefore the ratio $E/\rho = z^{p-1} E_0/\rho_0$ is maximized by $z=1$, whenever $p>1$, while without the SIMP penalization (i.e. with $p=1$) the ratio does not depend on $z$, so grey areas are not penalized.

Hinging upon the previous discussion, we come back to the diblock copolymers case and consider Figure~\ref{fig:homogenized_entries_norm}, which shows the values of the entries of the homogenized stiffness tensor, normalized by the density $\rho$, in dependence on $\CHm$. For both the $E^A/E^B$ ratios considered, the normalized stiffness of the principal direction is maximal for the stripes configurations, which explains the choice of this type of microstructure inside beams, where there is a preferential direction for stress and strain. At the junction between beams, instead, the stiffness associated to each direction plays a significant role, so the stripes configurations are penalized for their low $E/\rho$ ratio in the direction orthogonal to the principal one, and the $B$-spots configurations result to be the most convenient. We notice that the $B$-spots configurations are always preferred to the $A$-spots ones, since they feature a better $E/\rho$ ratio in each direction.

\begin{figure}
	\centering
	\subfloat[][$E^A/E^B=10$]  {\includegraphics[width=.5\textwidth,page = 3,trim={1.65cm 1.7cm 1.25cm 0},clip]{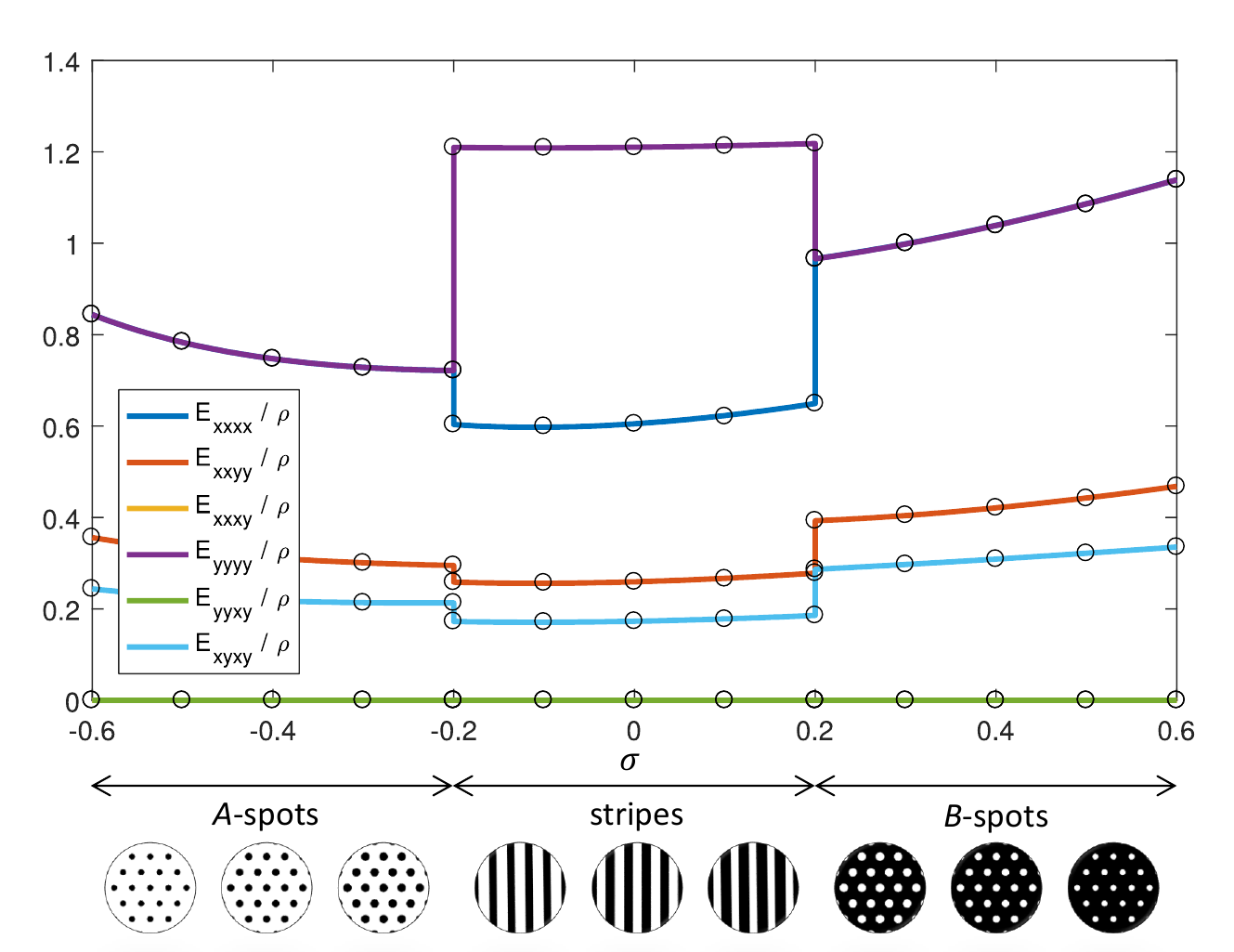}} \hfill
	\subfloat[][$E^A/E^B=1000$]{\includegraphics[width=.5\textwidth,page = 4,trim={1.65cm 1.7cm 1.25cm 0},clip]{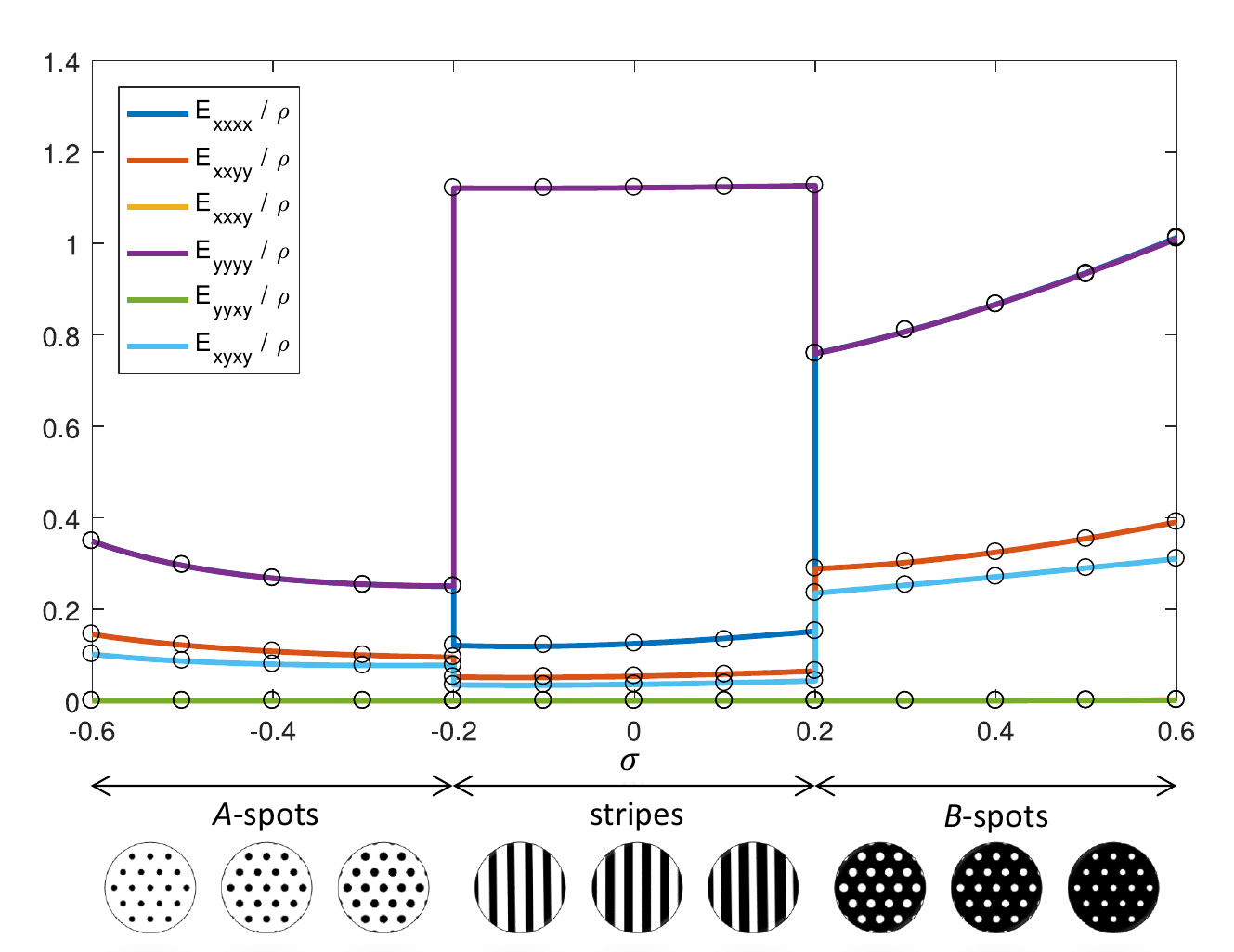}} \\ 
	\caption{Entries of the homogenized stiffness tensor normalized by the density $\rho$, for different ratios $E^A/E^B$ (reported in the tagline), in dependence of $\CHm$. As in Figure~\ref{fig:homogenized_entries}, circles represent the values computed numerically through the process described in Section~\ref{sec:applicationtodiblockcopolymers:homo}, see Eq.~\eqref{eqn:twoscales_model_smart}, while lines are the three-points polynomial interpolators.}
	\label{fig:homogenized_entries_norm}
\end{figure}

%
%

In Figure~\ref{fig:PBCHO_Lshape} the influence of the $E^A/E^B$ ratio on the optimal design is studied on the test case (b). The same test case is considered twice, the first time with $E^A/E^B = 10$, the second time with $E^A/E^B  = 10^3$. We notice that the optimal design is quite robust to changes in the quantity $E^A/E^B$, even if $A$-spots microstructures are slightly more convenient, compared to stripe configurations, when the ratio $E^A/E^B$ is high. This fact can be explained looking again at Figure~\ref{fig:homogenized_entries_norm}: 
while for $E^A/E^B = 10$ stripes configurations feature significant values of the $E/\rho$ ratio also in the non-preferential direction, for $E^A/E^B = 10^3$ the degree of anisotropy of stripes configurations is more pronounced, so that such microstructures are very complaint in the non-preferential direction. On the other hand, the stiffness of $A$-spots configurations is not affected so much by the value of the $E^A/E^B$ ratio. Therefore, stripes configurations turn out to be more competitive for high values of the $E^A/E^B$ ratio.


%

\begin{figure}[h]
	\centering
	\includegraphics[width=.32\textwidth]{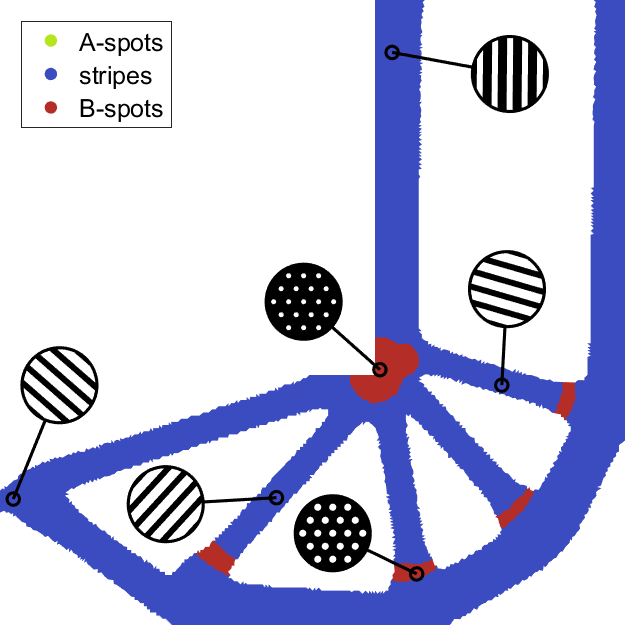}
	\includegraphics[width=.32\textwidth]{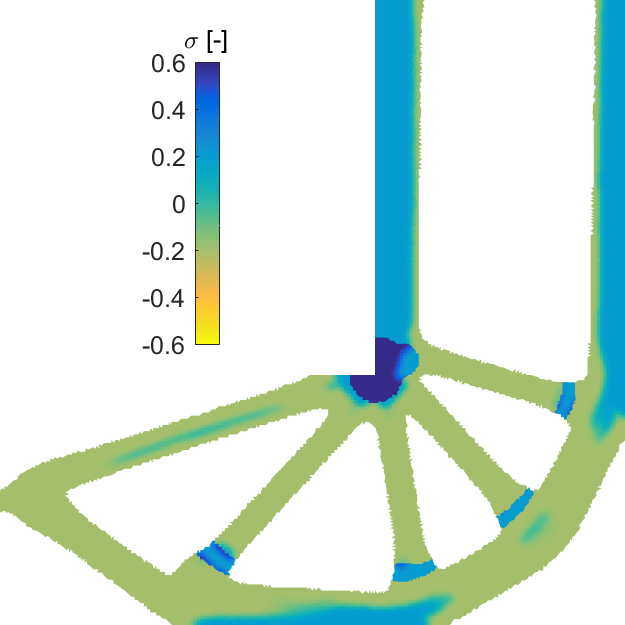}
	\includegraphics[width=.32\textwidth]{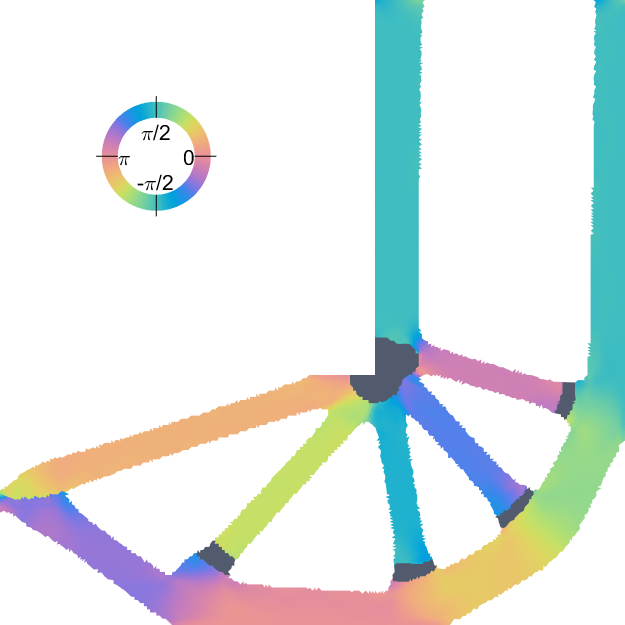}
	\\	\vspace{.2cm}
	\includegraphics[width=.32\textwidth]{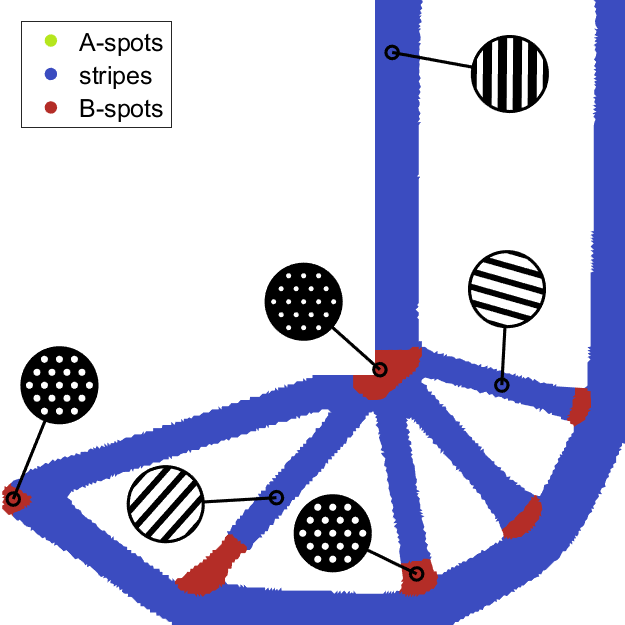}
	\includegraphics[width=.32\textwidth]{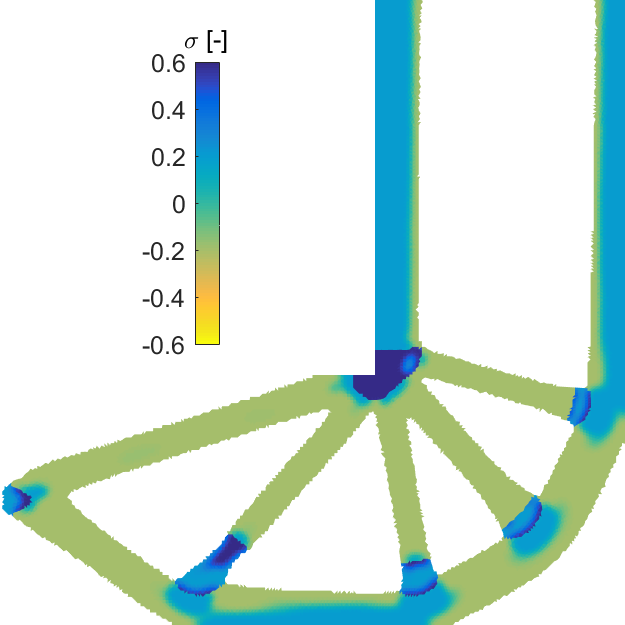}
	\includegraphics[width=.32\textwidth]{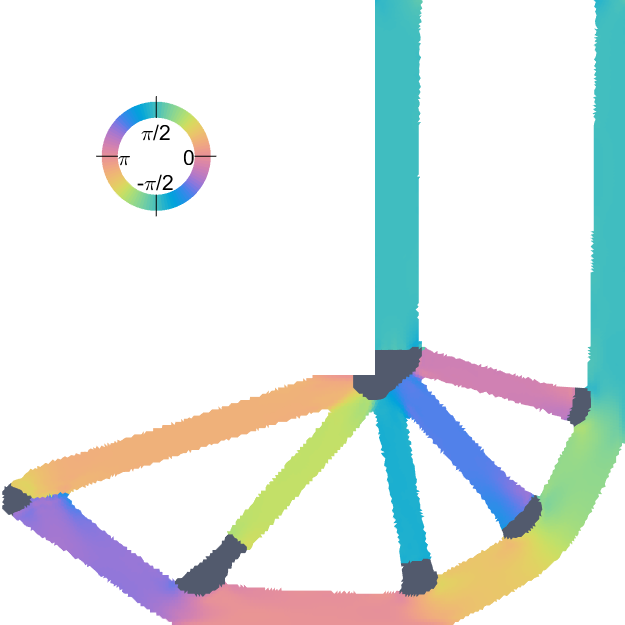}
	\caption{Numerical results of the generalized problem~\eqref{eqn:pb_fem_generalized} for the test \textit{L-shape} with different ratios of stiffness (First row: $E^A/E^B = 10$. Second row: $E^A/E^B = 10^3$). Domain size $8 \times 8$, filtering radius $0.1$. From left to right: allocation of microstructures, variable $\CHm$, variable $\theta$.}
	\label{fig:PBCHO_Lshape}
\end{figure}

\section{Conclusions} \label{sec:conclusions}

We considered an Optimality Criteria algorithm for the solution of a multimaterial version of the minimum compliance TopOpt problem, which consists in finding the optimal allocation of $N$ given possibly anisotropic materials inside a reference domain, given the total mass. 

As a novel contribution, we have then generalized the algorithm to a more general class of problems, where the candidate materials are allowed to change their orientation, which should be optimized, and their properties parametrized by a scalar variable, which is subject to optimization too. In this formulation the local properties of the materials and the macroscopic shape and topology of the body are simultaneously optimized, in order to maximize the stiffness (i.e. minimize its compliance) of the body. Moreover, we have shown well-posedness of this problem, by proving the existence of a minimizer. We have also proved well-definition of the proposed algorithm, in both the basic and the generalized case.

After recalling how the phase separation process of diblock copolymers, a class of self assembling materials, can be modelled, we have shown how the problem of controlling this process, in order to manufacture optimized bodies, can be formulated in terms of the previously introduced generalized multimaterial TopOpt problem. With this aim we have considered a two scales formulation, based on the homogenization theory, which allows to give a macroscopic characterization of the self-assembled microstructures of diblock copolymers.

Numerical results displaying the capabilities of the proposed method have been reported and commented.

We remark that, despite in this work only a two-dimensional implementation has been carried out, the performed analysis can be applied to the three dimensional case as well. The extension of the model to three-dimensional problems involving diblock copolymers would be particularly interesting, since the phase plane of diblock copolymers in 3D is richer that in 2D (lamellae, gyroids, cylinders and spheres patterns are present, see e.g. \cite{choksi2009_3d}), and thus a wider variety of microstructures with different anisotropy features would compete each other. 


We remark that the interest for diblock copolymers falls beyond structural applications, and the presented methodology can be generalized by considering other functional to be optimized, concerning, for instance, the thermal, magnetic or optical properties of the materials.

\section{Acknowledgements}

We gratefully thank Maurizio Grasselli for the interesting and useful discussions about the modelling of diblock copolymers phase separation.

\begin{appendices}
	\section{Proof of Theorem~\ref{thm:topopt_exists}} \label{app:proof_topopt_exists}	
	
	In order to prove Theorem~\ref{thm:topopt_exists} we need the following lemma.
	
	\begin{lemma}
		\label{lemma:weakstarbounds}
		Given a domain $\Omega \subset \mathbb{R}^d$, consider a sequence $\{\phi_j\}_{j \in \mathbb{N}} \subset L^{\infty}(\Omega)$, such that:
		\begin{equation*}
		\begin{split}
		a \leq \phi_j(\vec{x}) \leq b \qquad &\text{a.e. in $\Omega$}\\
		\phi_j \overset{*}{\rightharpoonup} \phi \qquad &\text{in $L^{\infty}(\Omega)$}
		\end{split}
		\end{equation*}
		for some $a,b \in \mathbb{R}$ and $\phi \in L^{\infty}(\Omega)$. Then:
		\begin{equation*}
		a \leq \phi(\vec{x}) \leq b \qquad \text{a.e. in $\Omega$.}
		\end{equation*}
	\end{lemma}
	
	\begin{proof}
		By lower semi-continuity of the norm with respect to $\text{weak}^*$ convergence (see~\cite{brezis2010functional}), we have, for a.a. $\vec{x} \in \Omega$:
		\begin{equation*}
		\phi(\vec{x}) - a \leq \| \phi -a\|_{L^{\infty}(\Omega)} \leq \underset{j}{\liminf} \| \phi_j -a\|_{L^{\infty}(\Omega)} \leq b-a\, ,
		\end{equation*}
		which entails $\phi(\vec{x}) \leq b$. On the other hand:
		\begin{equation*}
		b - \phi(\vec{x}) \leq \| b - \phi\|_{L^{\infty}(\Omega)} \leq \underset{j}{\liminf} \| b - \phi_j\|_{L^{\infty}(\Omega)} \leq b-a\, ,
		\end{equation*}	
		whence the second inequality.
	\end{proof}
		
		
	\begin{proof}[Proof of Theorem~\ref{thm:topopt_exists}]
		In the following we will denote by $\control = (z_\idxmat, \matm_\idxmat, \theta_\idxmat)_{\idxmat=1,\dots,N}$ an admissible design, and by $\controlset$ the space of admissible designs:
		\begin{equation*}
		\controlset = \left\{(z_\idxmat, \matm_\idxmat, \theta_\idxmat)_{\idxmat=1,\dots,N} : 
		\begin{split}
		z_\idxmat \in L^{\infty}(\Omega; [z_{min},1])  \qquad & \idxmat=1,\dots,N \\
		\matm_\idxmat \in L^{\infty}(\Omega; [\underline{\matm}_\idxmat,\overline{\matm}_\idxmat]) \qquad & \idxmat=1,\dots,N \\
		\theta_\idxmat \in L^{\infty}(\Omega; [0,2\pi)) \qquad & \idxmat=1,\dots,N \\
		\sum_{\idxmat=1}^{N} z_\idxmat(\vec{x}) \leq 1 \qquad &\forall\, \vec{x} \in \Omega \\
		\sum_{\idxmat=1}^{N} \int_{\Omega} \hat{z}_\idxmat(\vec{x}) \rho(\hat{\matm}_\idxmat(\vec{x})) \dx&\leq \masslim
		\end{split}
		\right\} \, .
		\end{equation*}	
		We denote by $\compliance(\control)$ the compliance $l(\vec{u})$ associated with the design $\control$. Let $\{\control_j\}_{j \in \mathbb{N}} \subset \controlset$ be a minimizing sequence for $\compliance$ in $\controlset$. Since the space $\controlset$ is a bounded subset of the space $L^{\infty}(\Omega;\mathbb{R}^{3N})$, by Banach-Alaoglu-Bourbaki theorem there exists a subsequence of $\{\control_{j}\}_j$ (which we do not relabel) and an $\control^\star \in \controlset$ such that $\control_{j} \overset{*}{\rightharpoonup} \control^\star$ in $L^{\infty}(\Omega;\mathbb{R}^{3N})$ as $j \rightarrow +\infty$. By definition of $\text{weak}^*$ convergence, since $K	\in L^1(\Omega)$, the sequence of filtered controls converges pointwise to the filtered counterpart of $\control^\star$ as $j \rightarrow +\infty$: 
		\begin{equation*}
		\hat{\control}_{j} (\vec{x})
		= \frac{\int_{\Omega} \control_{j}(\vec{z})  {K}(\vec{x} - \vec{z}) \diff{\vec{z}}}{\left({K} * 1 \right) (\vec{x})}
		\longrightarrow		 
		\frac{\int_{\Omega} \control^\star(\vec{z})  {K}(\vec{x} - \vec{z}) \diff{\vec{z}} }{\left({K} * 1 \right) (\vec{x})} 
		= \hat{\control}^\star(\vec{x})  \qquad \forall \vec{x} \in \Omega\, .
		\end{equation*}	
		Thanks to the regularization properties of the convolution operator it is possible to prove more than pointwise convergence. Indeed, consider the filtered sequence $\{\hat{\control}_{j}\}_j$. It satisfies the properties of:
		\begin{itemize}
			\item \textbf{equiboundedness}, since, by Lemma~\ref{lemma:weakstarbounds}, the filtered variables $\hat{\control}_{j}$ satisfy the same bounds of their non-filtered counterparts;
			\item \textbf{equicontinuity}, since for any $\vec{x},\vec{x}' \in \Omega$ and $j \in \mathbb{N}$, denoting by $L_K$ the Lipschitz constant of $K$:
			\begin{align*}
			\left\| \hat{\control}_{j} (\vec{x}) - \hat{\control}_{j} (\vec{x}')\right\|_{\mathbb{R}^{3N}} 
			&=\left\| 
			\frac{\int_{\Omega} {\control}_{j} (\vec{z})  {K}(\vec{x}  - \vec{z}) \diff{\vec{z}}}{\int_{\Omega} {K}(\vec{x}  - \vec{w}) \diff{\vec{w}}} -
			\frac{\int_{\Omega} {\control}_{j} (\vec{z})  {K}(\vec{x}' - \vec{z}) \diff{\vec{z}}}{\int_{\Omega} {K}(\vec{x}' - \vec{w}) \diff{\vec{w}}}
			\right\|_{\mathbb{R}^{3N}} \\
			&=\left\| 
			\frac{\int_{\Omega}\int_{\Omega} {\control}_{j} (\vec{z}) \left[ 
				{K}(\vec{x} - \vec{z}){K}(\vec{x}' - \vec{w})-{K}(\vec{x}' - \vec{z}){K}(\vec{x} - \vec{w})	
				\right]\diff{\vec{z}}\diff{\vec{w}}}
			{\left(\int_{\Omega} {K}(\vec{x} - \vec{w}) \diff{\vec{w}}\right)\left(\int_{\Omega} {K}(\vec{x}' - \vec{w}) \diff{\vec{w}}\right)}
			\right\|_{\mathbb{R}^{3N}} \\
			&\leq \frac{1}{K_{min}^2}\Bigl\| \int_{\Omega}  {\control}_{j} (\vec{z})  \int_{\Omega}  
			\left[K(\vec{x} - \vec{z})\left(K(\vec{x}' - \vec{w})-K(\vec{x} -  \vec{w}) \right) \right.\Bigr.
			\\*&\qquad\qquad\qquad\qquad\qquad \Bigl.\left.+K(\vec{x} - \vec{w})\left(K(\vec{x}  - \vec{z})-K(\vec{x}' - \vec{z}) \right)\right]
			\diff{\vec{w}}\diff{\vec{z}} \Bigr\|_{\mathbb{R}^{3N}} \\
			&\leq \frac{1}{K_{min}^2} \int_{\Omega}   
			\underbrace{\|{\control}_{j} (\vec{z}) \|_{\mathbb{R}^{3N}}}_{\leq \|{\control}_{j} \|_{L^{\infty}(\Omega;\mathbb{R}^{3N})}} 
			\Bigl[K(\vec{x} - \vec{z}) \int_{\Omega}\underbrace{\left|K(\vec{x}' - \vec{w})-K(\vec{x} -  \vec{w}) \right|}_{\leq L_K \|\vec{x}-\vec{x}'\|} \diff{\vec{w}} \Bigr.
			\\*&\qquad \Bigl.+ \int_{\Omega}K(\vec{x} - \vec{w})\diff{\vec{w}} \underbrace{\left|K(\vec{x}' - \vec{z})-K(\vec{x} -  \vec{z}) \right|}_{\leq L_K \|\vec{x}-\vec{x}'\|}\Bigr]
			\diff{\vec{z}} \\
			&\leq \frac{2 \|{\control}_{j} \|_{L^{\infty}(\Omega;\mathbb{R}^{3N})} \|K \|_{L^1(\mathbb{R}^d)}|\Omega| L_K}{K_{min}^2} \|\vec{x}-\vec{x}'\|\, .
			\end{align*}					
		\end{itemize}		
		Then, by Ascoli-Arzel\`{a} theorem, there exists a further subsequence of $\{\hat{\control}_{j}\}$ (which again we do not relabel) uniformly convergent to some limit. Since uniform convergence entails pointwise convergence, and by uniqueness of the limit, the limit coincides with $\hat{\control}^\star$. Therefore $\hat{\control}_{j} \rightarrow \hat{\control}^\star$ in $L^{\infty}(\Omega;\mathbb{R}^{3N})$ as $j \rightarrow +\infty$.
	
		We will show that the design-to-solution map $\hat{\control} \mapsto \vec{u}$ is continuous from the space $\controlset$ (endowed with the $L^{\infty}(\Omega;\mathbb{R}^{3N})$ topology) to $H^1_{\Gamma_D}(\Omega;\mathbb{R}^d)$. Then, thanks to the continuity of the linear form $l(\cdot)$ in $H^1_{\Gamma_D}(\Omega;\mathbb{R}^d)$, it follows that $\compliance(\control_{j}) \rightarrow \compliance(\control^\star)$ as $j \rightarrow +\infty$, which entails $\compliance(\control^\star) = \inf_{\control \in \controlset}\compliance(\control)$, since the sequence $\{\control_{j}\}_j$ is minimizing. We notice that $\control^\star$ is also admissible. Indeed by Lemma~\ref{lemma:weakstarbounds} the upper and lower bounds of the design variables are transferred to their $\text{weak}^*$ limits (the same holds true for $\sum_{\idxmat=1}^{N} z_\idxmat$), and by uniform convergence of the filtered variables it is possible to pass to the limit into the mass constraint. Therefore we have $\control^\star \in \controlset$.
		
		We are left to show the continuity of the design-to-solution map. First of all, we show that the bilinear forms associated with filtered designs in the space $\controlset$ are equicoercive. Consider a filtered design $\hat{\control} = (\hat{z}_\idxmat, \hat{\matm}_\idxmat, \hat{\theta}_\idxmat)_{\idxmat=1,\dots,N} \in \controlset$, and the associated stiffness tensor $\tensel(\vec{x})$. Thanks to the lower bound on the variables $z_\idxmat$ and Korn's inequality
		, the equicoercivity property is transferred to the bilinear form $a(\cdot,\cdot)$:
		\begin{equation*}
		\begin{split}
		a(\vec{u},\vec{u}) 
		& = \int_{\Omega} \sum_{\idxmat=1}^{N} {\underbrace{\hat{z}_\idxmat(\vec{x})}_{\geq z_{min}}}^p 
		\underbrace{\left( \Qrot(\hat{\theta}_\idxmat(\vec{x})) \tensel_\idxmat(\hat{\matm}_\idxmat(\vec{x})) \right) \symgrad{\vec{u}(\vec{x})} : \symgrad{\vec{u}(\vec{x})}}_{\geq \alpha \left| \symgrad{\vec{u}(\vec{x})} \right|^2} \dx \\
		&\geq N z_{min}^p \alpha \left\| \symgrad{\vec{u}} \right\|^2_{L^2} \geq  N z_{min}^p \alpha C_K \left\| {\vec{u}} \right\|^2_{H^1}\, ,
		\end{split}
		\end{equation*}
		where the constant $\bar{\alpha} := N z_{min}^p \alpha C_K $ does not depend on $\hat{\control}$.	
		
		Consider therefore two filtered designs $\hat{\control}_1, \hat{\control}_2 \in \controlset$, and the associated bilinear forms $a_1(\cdot,\cdot)$ and $a_2(\cdot,\cdot)$. The corresponding states $\vec{u}_1$ and $\vec{u}_2$ satisfy:
		\begin{equation*}
		\begin{split}
		a_1(\vec{u}_1,\vec{v}) &= l(\vec{v}) \qquad \forall \, \vec{v} \in H^1_{\Gamma_D} \, ,\\
		a_2(\vec{u}_2,\vec{v}) &= l(\vec{v}) \qquad \forall \, \vec{v} \in H^1_{\Gamma_D} \, .\\
		\end{split}
		\end{equation*}		
		By subtracting the previous equations, adding and subtracting $a_1(\vec{u}_2,\vec{v})$, and choosing $\vec{v} = \vec{u}_2 - \vec{u}_1$, we get:
		\begin{equation}
		\label{eqn:existencethm_chain_ineq}
		\begin{split}
		\underbrace{a_1(\vec{u}_2 - \vec{u}_1,\vec{u}_2 - \vec{u}_1)}_{\geq \bar{\alpha} \left\| \vec{u}_2 - \vec{u}_1 \right\|^2_{H^1} }
		&= a_1(\vec{u}_2,\vec{u}_2 - \vec{u}_1) - a_2(\vec{u}_2,\vec{u}_2 - \vec{u}_1) \\
		&= \int_{\Omega} \left(\tensel_1 - \tensel_2\right) \symgrad{\vec{u}_2} : \symgrad{(\vec{u}_2 - \vec{u}_1)} \dx \\
		&\leq \left\| \tensel_1 - \tensel_2 \right\|_{L^{\infty}}   \left\| \symgrad{\vec{u}_2} \right\|_{L^2}   \left\| \symgrad{(\vec{u}_2 - \vec{u}_1)} \right\|_{L^2} \, .
		\end{split}
		\end{equation}		
		The last two factors of the last line of~\eqref{eqn:existencethm_chain_ineq} can be estimated as follows:
		\begin{itemize}
			\item $ \left\| \symgrad{\vec{u}_2} \right\|_{L^2} \leq \frac{1}{\bar{\alpha}}  \left\| l \right\|_{H^{-1}}$ by Lax-Milgram lemma;
			\item $\left\| \symgrad{(\vec{u}_2 - \vec{u}_1)} \right\|_{L^2} \leq \left\| \vec{u}_2 - \vec{u}_1 \right\|_{H^1}$.
		\end{itemize}
		
		Therefore, plugging the estimates into~\eqref{eqn:existencethm_chain_ineq} we conclude:
		\begin{equation}
		\label{eqn:proofexist_estimate}
		\left\| \vec{u}_2 - \vec{u}_1 \right\|_{H^1} \leq \frac{M \left\| l \right\|_{H^{-1}}}{\bar{\alpha}^2} \left\| \tensel_1 - \tensel_2 \right\|_{L^{\infty}} \, .
		\end{equation}		
		By the Heine-Cantor Theorem, being the stiffness tensors $\tensel_\idxmat(\matm)$ continuous on the compacts $[\underline{\matm}_\idxmat,\overline{\matm}_\idxmat]$, the functions $\tensel_\idxmat(\matm)$ are indeed \textit{uniformly} continuous. Therefore, the pointwise application
		\begin{equation*}
		\hat{\control} \mapsto \tensel = \sum_{\idxmat=1}^{N} \hat{z}_\idxmat^p \Qrot(\hat{\theta}_\idxmat) \tensel_\idxmat(\hat{\matm}_\idxmat)
		\end{equation*}
		is uniformly continuous from $\mathbb{R}^{3N}$ to $\mathbb{R}^{d^4}$, and thus the map $\hat{\control}(\vec{x}) \mapsto \tensel(\vec{x})$ is continuous from $L^\infty(\Omega;\mathbb{R}^{3N})$ to $L^\infty(\Omega;\mathbb{R}^{d^4})$. Finally, thanks to the estimate~\eqref{eqn:proofexist_estimate}, we can conclude the continuity of the design-to-solution map.
	\end{proof}

	\section{Proof of Proposition~\ref{prop:topopt_multimal_mL_exist_unique}} \label{app:proof_topopt_multimal_mL_exist_unique}

	\begin{proof}[Proof of Proposition~\ref{prop:topopt_multimal_mL_exist_unique}]
				
		
		We denote by $Z_l$ the total local density: $Z_l(\Lambda,\mu_l) = \sum_{\idxmat=1}^{N} z_{\idxmat,l}^{k+1}$. Moreover, in each mesh element $l$, we define the following sets, which track for which indexes $\idxmat$ and in which mesh elements the constraints~\eqref{eqn:pb_fem_generalized:zbounds}, \eqref{eqn:pb_fem_generalized:zsum1} and~\eqref{eqn:pb_fem_generalized:mbounds} are active: 
		\begin{equation*}
		\begin{split}
		\freeIDXzl{l}(\Lambda,\mu_l) &= \{ \idxmat = 1,\dots, N \text{ s.t. }  z_{\idxmat,l}^{k+1} > z_{min} \}  ,\\
		\freeIDXz(\Lambda,\vec{\mu}) &= \{ l = 1,\dots,N_e \text{ s.t. } \freeIDXzl{l} \neq \emptyset \} ,  \\
		\freeIDXzUp(\Lambda,\vec{\mu}) &= \{ l = 1,\dots,N_e \text{ s.t. } Z_l(\Lambda,\mu_l) = 1 \} ,  \\
		\freeIDXml{l}(\Lambda,\mu_l) &= \{ \idxmat = 1,\dots, N \text{ s.t. } \underline{\matm}_{\idxmat} < \matm_{\idxmat,l}^{k+1} < \overline{\matm}_{\idxmat} \}.\\
		\end{split}
		\end{equation*}	
		Notice that by definition $\freeIDXzUp \subset \freeIDXz$. It is easily seen by~\eqref{eqn:FEM_nf_dE} that the tensors $\frac{\partial\vec{K}}{\partial \hat{z}_{\idxmat,r}}$ and $\frac{\partial\vec{K}}{\partial \hat{\matm}_{\idxmat,r}}$ are positive semidefinite. Therefore both the numerators and the denominators of the fractions in~\eqref{eqn:BD_filtered} are not less than $\epsilon$, so we have:
		\begin{equation}
		\begin{aligned}
		& \frac{\partial z_{\idxmat,l}^{k+1}}{\partial \mu_l} = 
		- \frac{ 
			\left(\vec{U}^k\right)^T\left(\sum_{r=1}^{N_e} \hat{H}_{rl}\frac{\partial \tilde{\vec{K}}}{\partial \hat{z}_{\idxmat,r}}\right) \vec{U}^k +\epsilon
		}{
			\left(\Lambda \sum_{r=1}^{N_e} \hat{H}_{rl} |e_r| \rho_\idxmat(\hat{\matm}_{\idxmat,r}^k) + \mu_l +\epsilon\right)^2
		} 
		z_{\idxmat,l}^k < 0 \qquad 
		&& \forall \,\idxmat\in \freeIDXzl{l}, \\
		& \frac{\partial z_{\idxmat,l}^{k+1}}{\partial \Lambda} = 
		\left(\sum_{r=1}^{N_e} \hat{H}_{rl} |e_r| \rho_\idxmat(\hat{\matm}_{\idxmat,r}^k)\right) \frac{\partial z_{\idxmat,l}^{k+1}}{\partial \mu_l} <0\qquad 
		&& \forall \,\idxmat\in \freeIDXzl{l}, \\
		& \frac{\partial \matm_{\idxmat,l}^{k+1}}{\partial \Lambda} = 
		-\frac{ 
			\left(\vec{U}^k\right)^T\left(\sum_{r=1}^{N_e} \hat{H}_{rl}\frac{\partial \vec{K}}{\partial \hat{\matm}_{\idxmat,r}}\right) \vec{U}^k +\epsilon
		}{
			\left(\Lambda  \sum_{r=1}^{N_e} \hat{H}_{rl} |e_r| \rhomprime(\hat{\matm}_{\idxmat,r}^k) \hat{z}_{\idxmat,r}^k  +\epsilon\right)^2
		} \cdot && \\ & \qquad \qquad 
		\left(\sum_{r=1}^{N_e} \hat{H}_{rl} |e_r| \rhomprime(\hat{\matm}_{\idxmat,r}^k) \hat{z}_{\idxmat,r}^k\right)
		\left(\matm_{\idxmat,l}^k - \underline{\matm}_\idxmat + \delta\right)  \leq 0 \qquad 
		&& \forall \,\idxmat\in \freeIDXml{l}. \\
		\end{aligned}
		\end{equation}		
		Fix $\Lambda \geq 0 $ and $l$. It is easily seen that the $Z_l$ is continuous and non increasing in $\mu_l$, being the composition of continuous and non increasing functions, then~\ref{point:topopt_multimat_proof_a} is proved. Moreover, in a neighbourhood of those $\mu_l$ such that the set $\freeIDXzl{l}$ is not empty (i.e. $l \in \freeIDXz$), $Z_l$ is \textit{strictly} decreasing in $\mu_l$. 
		
		To prove~\ref{point:topopt_multimat_proof_b}, we consider the two possibilities:
		\begin{itemize}
			\item $Z_l(\Lambda,0) < 1$. Then $\mu_l^{\Lambda}=0$. This choice is unique, since for any $\mu_l>0$ we have $Z_l(\Lambda,\mu_l)\leq Z_l(\Lambda,0) < 1$.
			\item $Z_l(\Lambda,0) \geq 1$. 
			Since $Z_l \rightarrow N\,z_{min}<1 $ for $\mu_l \rightarrow +\infty$, there exists at least a $\mu_l > 0$ such that $Z_l=1$. The choice in unique, since, for those values of $\mu_l$, the set $\freeIDXzl{l}$ cannot be empty (thanks to~\eqref{eqn:topopt_multimat_nontrivialparam1}), and so $Z_l$ is \textit{strictly} decreasing in a neighbourhood.
		\end{itemize}		
		Therefore, for $l \notin \freeIDXzUp(\Lambda,\vec{\mu}^{\Lambda})$ we have $\mu_l^\Lambda \equiv 0$ in a neighbourhood of that $\Lambda$, and so $\frac{\partial \mu_l^{\Lambda}}{\partial\Lambda} = 0$. On the other hand, for $l \in \freeIDXzUp(\Lambda,\vec{\mu}^{\Lambda})$, the map $(\Lambda, \mu_l) \mapsto Z_l$ is of class $\mathcal{C}^1$, thus by Dini's theorem the correspondence $\Lambda \mapsto {\mu}_l^{\Lambda} $ is $\mathcal{C}^1$ as well, and it holds:
		\begin{equation}
		\label{eqn:topopt_multimat_proof_dmdL}
		\frac{\partial \mu_l^{\Lambda}}{\partial\Lambda}   
		= - \frac{\frac{\partial Z_l}{\partial\Lambda}}{\frac{\partial Z_l}{\partial\mu_l}}
		= - \frac{\sum_{\idxmat\in \freeIDXzl{l}} \frac{\partial z_{\idxmat,l}^{k+1}}{\partial\Lambda}}{\sum_{\idxmat\in \freeIDXzl{l}} \frac{\partial z_{\idxmat,l}^{k+1}}{\partial\mu_l}}
		\qquad
		\forall \, l \in \freeIDXzUp(\Lambda,\vec{\mu}^{\Lambda})\, .
		\end{equation}	
		Consider now the map $\Lambda \mapsto \mass^{k+1}(\Lambda,\vec{\mu}^{\Lambda})$: it is continuous, being the composition of continuous functions. Moreover:
		\begin{equation*}
		\frac{d M^{k+1}}{d \Lambda} 
		= 
		\underbrace{
			\sum_{\idxmat=1}^{N} \sum_{r=1}^{N_e} |e_r| \frac{d \hat{z}_{\idxmat,r}^{k+1}}{d \Lambda}  \rho_\idxmat(\hat{\matm}_{\idxmat,r}^{k+1})
		}_{(I)} 
		+ 
		\underbrace{
			\sum_{\idxmat=1}^{N} \sum_{r=1}^{N_e} |e_r| \hat{z}_{\idxmat,r}^{k+1} \rhomprime(\hat{\matm}_{\idxmat,r}^{k+1}) \frac{d \hat{\matm}_{\idxmat,r}^{k+1}}{d \Lambda} 
		}_{(II)} \, .
		\end{equation*}
		We treat the two terms separately. Denoting $\rho_\idxmat(\hat{\matm}_{\idxmat,r}^{k+1})$ simply by $\rho_\idxmat$ we have:
		\begin{align*}
		(I) 
		&=
		\sum_{\idxmat=1}^{N} \sum_{r=1}^{N_e} |e_r| \sum_{l=1}^{N_e} \hat{H}_{rl} \frac{d z_{\idxmat,l}^{k+1}}{d \Lambda}  \rho_\idxmat
		\\
		&=
		\sum_{l \in \freeIDXzUp} \sum_{\idxmat \in \freeIDXzl{l}} \sum_{r=1}^{N_e} |e_r| \hat{H}_{lr} \left(\frac{\partial z_{\idxmat,l}^{k+1}}{\partial \Lambda} + \frac{\partial z_{\idxmat,l}^{k+1}}{\partial \mu_l} \frac{\partial \mu^\Lambda_l}{\partial \Lambda} \right)  \rho_\idxmat 
		+ 
		\sum_{l \in \freeIDXz \setminus \freeIDXzUp} \sum_{\idxmat \in \freeIDXzl{l}} \sum_{r=1}^{N_e} |e_r| \hat{H}_{lr} \underbrace{\frac{\partial z_{\idxmat,l}^{k+1}}{\partial \Lambda}}_{{}<0} \rho_\idxmat
		\\
		&\leq 
		\sum_{l \in \freeIDXzUp} \left(\sum_{r=1}^{N_e}  |e_r|  \hat{H}_{rl} \right)	
		\left[
		\sum_{\idxmat\in \freeIDXzl{l}} \rho_\idxmat \left(\sum_{r=1}^{N_e} \hat{H}_{rl} |e_r| \rho_\idxmat\right) \frac{\partial z_{\idxmat,l}^{k+1}}{\partial \mu_l}
		\right.
		\\*  
		& 
		\phantom{\sum_{l \in \freeIDXzUp} \left(\sum_{r=1}^{N_e}  |e_r|  \hat{H}_{rl} \right)	
			\left[ \right.} 
		\left.
		{}-\left(\sum_{\idxmat\in \freeIDXzl{l}} \rho_\idxmat \frac{\partial z_{\idxmat,l}^{k+1}}{\partial \mu_l^\Lambda} \right)
		\frac{ \sum_{\idxmat\in \freeIDXzl{l}} \left(\sum_{r=1}^{N_e} \hat{H}_{rl} |e_r| \rho_\idxmat\right) \frac{\partial z_{\idxmat,l}^{k+1}}{\partial \mu_l} }
		{ \sum_{\idxmat\in \freeIDXzl{l}} \frac{\partial z_{\idxmat,l}^{k+1}}{\partial \mu_l} } \right] \\
		\\
		&=
		\sum_{l \in \freeIDXzUp}	 \frac{ \left(\sum_{r=1}^{N_e}  |e_r|  \hat{H}_{rl} \right)^2 }{\sum_{\idxmat\in \freeIDXzl{l}} \frac{\partial z_{\idxmat,l}^{k+1}}{\partial \mu_l}} \sum_{\idxmat,j\in \freeIDXzl{l}}
		\left[	
		\rho_j^2 \frac{\partial z_{\idxmat,l}^{k+1}}{\partial \mu_l} \frac{\partial z_{j,l}^{k+1}}{\partial \mu_l}
		-
		\rho_j \rho_\idxmat \frac{\partial z_{\idxmat,l}^{k+1}}{\partial \mu_l} \frac{\partial z_{j,l}^{k+1}}{\partial \mu_l}
		\right]
		\\
		&=
		\sum_{l \in \freeIDXzUp}	 \frac{ \left(\sum_{r=1}^{N_e}  |e_r|  \hat{H}_{rl} \right)^2 }{\sum_{\idxmat\in \freeIDXzl{l}} \frac{\partial z_{\idxmat,l}^{k+1}}{\partial \mu_l}} \left[
		\sum_{\idxmat<j}
		\rho_j (\rho_j-\rho_\idxmat) \frac{\partial z_{\idxmat,l}^{k+1}}{\partial \mu_l} \frac{\partial z_{j,l}^{k+1}}{\partial \mu_l}
		+
		\sum_{j<\idxmat}
		\rho_j (\rho_j-\rho_\idxmat) \frac{\partial z_{\idxmat,l}^{k+1}}{\partial \mu_l} \frac{\partial z_{j,l}^{k+1}}{\partial \mu_l}
		\right]
		\\
		&=
		\sum_{l \in \freeIDXzUp} 	 \frac{ \left(\sum_{r=1}^{N_e}  |e_r|  \hat{H}_{rl} \right)^2 }{\sum_{\idxmat\in \freeIDXzl{l}} \frac{\partial z_{\idxmat,l}^{k+1}}{\partial \mu_l}} \left[
		\sum_{\idxmat<j}
		\rho_j (\rho_j-\rho_\idxmat) \frac{\partial z_{\idxmat,l}^{k+1}}{\partial \mu_l} \frac{\partial z_{j,l}^{k+1}}{\partial \mu_l}
		-
		\sum_{\idxmat<j}
		\rho_\idxmat (\rho_j-\rho_\idxmat) \frac{\partial z_{\idxmat,l}^{k+1}}{\partial \mu_l} \frac{\partial z_{j,l}^{k+1}}{\partial \mu_l}
		\right]
		\\
		&=
		\sum_{l \in \freeIDXzUp} \underbrace{\frac{ \left(\sum_{r=1}^{N_e}  |e_r|  \hat{H}_{rl} \right)^2 }{\sum_{\idxmat\in \freeIDXzl{l}} \frac{\partial z_{\idxmat,l}^{k+1}}{\partial \mu_l}}}_{{}< 0} \left[
		\sum_{\idxmat<j}
		(\rho_j-\rho_\idxmat)^2 \underbrace{\frac{\partial z_{\idxmat,l}^{k+1}}{\partial \mu_l}}_{{}< 0} \underbrace{\frac{\partial z_{j,l}^{k+1}}{\partial \mu_l}}_{{}< 0}
		\right] \leq 0 \, .
		\\
		\end{align*}	
		Notice that the equality holds only in the case the set $\freeIDXz$ is empty. Moreover:
		\begin{equation*}
		(II) 
		=
		\sum_{\idxmat=1}^{N} \sum_{l=1}^{N_e} |e_l| \hat{z}_{\idxmat,l}^{k+1} \rhomprime(\hat{\matm}_{\idxmat,l}^{k+1}) \sum_{r=1}^{N_e} \hat{H}_{lr}\frac{d \matm_{\idxmat,r}^{k+1}}{d \Lambda} \leq 0 \, .
		\end{equation*}	
		Then point~\ref{point:topopt_multimat_proof_c} is proved. We are left to show point~\ref{point:topopt_multimat_proof_d}. We have the following two possibilities:
		\begin{itemize}
			\item $\mass(0,\vec{\mu}^0) < \masslim$. Then $\Lambda=0$ does the job. This choice is unique, since for any $\Lambda>0$ we have $\mass(\Lambda,\vec{\mu}^\Lambda)\leq \mass(0,\vec{\mu}^0) < \masslim$.
			\item $\mass(0,\vec{\mu}^0) \geq \masslim$. In the limit $\Lambda \rightarrow +\infty$, thanks to~\eqref{eqn:topopt_multimat_nontrivialparam2} we have: 
			\begin{equation*}
			\mass \rightarrow |\Omega| z_{min} \sum_{\idxmat=1}^{N} \rho_\idxmat(\underline{\matm}_\idxmat) < \masslim\, ,
			\end{equation*}		
			therefore, by continuity we have at least one $\Lambda$ such that $\mass(\Lambda,\vec{\mu}^\Lambda) = \masslim$. Moreover we know that $\mass$ is \textit{strictly} decreasing in $\Lambda$, unless the set $\freeIDXz$ is empty. But for the values of $\Lambda$ for which the set $\freeIDXz$ is empty, we have $\mass \leq |\Omega| z_{min} \sum_{\idxmat=1}^{N} \rho_\idxmat(\overline{\matm}_\idxmat) < \masslim$, then we can conclude that the values of $\Lambda$ fulfilling the mass constraint lie in the region of strict monotonicity of the mass. Therefore the choice of $\Lambda$ is unique.
		\end{itemize}
	\end{proof}
\end{appendices}


\bibliographystyle{elsarticle-num}
\bibliography{biblioTopOpt,biblioHomogenization,biblioCH,biblioOthers}

\end{document}